\documentclass[preprint]{elsarticle}

\usepackage{algorithm}
\usepackage{algorithmic}
\usepackage{amssymb}
\usepackage{amsmath}
\usepackage{amsthm}
\usepackage{array}
\usepackage{booktabs}
\usepackage{caption}
\usepackage{comment}
\usepackage{float}
\usepackage[latin1]{inputenc}
\usepackage[pdfborder={0 0 0}, colorlinks=true, linkcolor=blue, citecolor=red]{hyperref}\hypersetup{pdfborder=0 0 0}
\usepackage{mathrsfs}
\usepackage{multicol}
\usepackage{multirow}
\usepackage{subfigure}
\usepackage{wrapfig}
\usepackage{ulem}
\usepackage{upgreek}

\usepackage{soul,xargs}
\usepackage[pdftex,dvipsnames]{xcolor}

\usepackage[colorinlistoftodos,prependcaption,textsize=tiny]{todonotes}
\newcommandx{\question}[2][1=]{\todo[linecolor=red,backgroundcolor=red!20,bordercolor=red,#1]{#2}}
\newcommandx{\change}[2][1=]{\todo[linecolor=blue,backgroundcolor=blue!25,bordercolor=blue,#1]{#2}}
\newcommandx{\info}[2][1=]{\todo[linecolor=OliveGreen,backgroundcolor=OliveGreen!25,bordercolor=OliveGreen,#1]{#2}}
\newcommandx{\improve}[2][1=]{\todo[linecolor=Plum,backgroundcolor=Plum!25,bordercolor=Plum,#1]{#2}}
\newcommandx{\thiswillnotshow}[2][1=]{\todo[disable,#1]{#2}}
\newcommandx{\answer}[2][1=]{\todo[linecolor=blue,backgroundcolor=White!25,bordercolor=Plum,#1]{#2}}
\newcommandx{\sriram}[2][1=]{\todo[linecolor=orange,backgroundcolor=orange!25,bordercolor=orange,#1]{\textbf{SM:} #2}}

\newtheorem{theorem}{Theorem}
\newtheorem{definition}{Definition}
\newtheorem{corollary}{Corollary}
\newtheorem{lemma}{Lemma}
\newtheorem{proposition}{Proposition}
\newtheorem{remark}{Remark}

\newcommand*\patchAmsMathEnvironmentForLineno[1]{%
   \expandafter\let\csname old#1\expandafter\endcsname\csname #1\endcsname
   \expandafter\let\csname oldend#1\expandafter\endcsname\csname end#1\endcsname
   \renewenvironment{#1}%
      {\linenomath\csname old#1\endcsname}%
      {\csname oldend#1\endcsname\endlinenomath}}%
\newcommand*\patchBothAmsMathEnvironmentsForLineno[1]{%
   \patchAmsMathEnvironmentForLineno{#1}%
   \patchAmsMathEnvironmentForLineno{#1*}}%
\AtBeginDocument{%
\patchBothAmsMathEnvironmentsForLineno{equation}%
\patchBothAmsMathEnvironmentsForLineno{align}%
\patchBothAmsMathEnvironmentsForLineno{flalign}%
\patchBothAmsMathEnvironmentsForLineno{alignat}%
\patchBothAmsMathEnvironmentsForLineno{gather}%
\patchBothAmsMathEnvironmentsForLineno{multline}%
}

\usepackage{lineno}

\makeatletter
\def\ps@pprintTitle{%
  \let\@oddhead\@empty \let\@evenhead\@empty
  \def\@oddfoot{\footnotesize\itshape Preprint, October 2026\hfill}%
  \let\@evenfoot\@oddfoot}
\makeatother

\begin{document}

\newcommand{\TODO}[1]{ \fbox{\parbox{3in}{\bf TODO: #1}}}

\newcommand{\grbf}[1] {\mbox{\boldmath${#1}$\unboldmath}}
\newcommand{\gbf}[1] {\mathbf{#1}}

\newcommand{\CPP}{C\nolinebreak\hspace{-.05em}\raisebox{.4ex}{\tiny\bf +}\nolinebreak\hspace{-.10em}\raisebox{.4ex}{\tiny\bf +}}

\newcommand{\beq} {\begin{equation}}
\newcommand{\eeq} {\end{equation}}
\newcommand{\bdm} {\begin{displaymath}}
\newcommand{\edm} {\end{displaymath}}
\newcommand{\bit}{\begin{itemize}}
\newcommand{\eit}{\end{itemize}}
\newcommand{\bde}{\begin{description}}
\newcommand{\ede}{\end{description}}
\newcommand{\bce}{\begin{center}}
\newcommand{\ece}{\end{center}}
\newcommand{\ben} {\begin{enumerate}}
\newcommand{\een} {\end{enumerate}}
\newcommand{\bea} {\begin{eqnarray}}
\newcommand{\eea} {\end{eqnarray}}
\newcommand{\barr} {\begin{array}}
\newcommand{\earr} {\end{array}}
\newcommand{\bean} {\begin{eqnarray*}}
\newcommand{\eean} {\end{eqnarray*}}
\newcommand{\edoc} {\end{document}}
\newcommand{\On}{\hspace{0.2cm} \mbox{on} \hspace{0.2cm}}
\newcommand{\In}{\hspace{0.2cm} \mbox{in} \hspace{0.2cm}}
\newcommand{\Minimize}{\hspace{0.2cm} \mbox{minimize} \hspace{0.2cm}}
\newcommand{\Maximize}{\hspace{0.2cm} \mbox{maximize} \hspace{0.2cm}}
\newcommand{\SubjectTo}{\hspace{0.2cm} \mbox{subject} \hspace{0.15cm}
            \mbox{to} \hspace{0.2cm}}

\def\etal{{\it et al.~}}

\newcommand{\point}[1] {\ensuremath{\boldsymbol{#1}}}
\newcommand{\vvec}[1] {\ensuremath{\boldsymbol{#1}}}
\renewcommand{\vec}[1] {\ensuremath{\boldsymbol{#1}}}
\newcommand{\cvec}[1] {\ensuremath{\boldsymbol{{#1}}}}
\newcommand{\ten}[1] {\ensuremath{\boldsymbol{#1}}}
\newcommand{\tenfour}[1] {\ensuremath{\boldsymbol{\mathsf{#1}}}}
\newcommand{\comp}[1]{\begin{pmatrix}{ #1 }\end{pmatrix}}
\newcommand{\vv}{\ensuremath{\vec{v}}}
\newcommand{\vr}{\left(\ensuremath{\vec{r}}\right)}
\newcommand{\att}{\left(t\right)}
\newcommand{\tvv}{\ensuremath{\tilde{\vec{v}}}}
\newcommand{\ww}{\ensuremath{\vec{w}}}
\newcommand{\GG}{\ensuremath{\ten{G}}}
\newcommand{\FF}{\ensuremath{\boldsymbol{\mathfrak{F}}}}
\newcommand{\tEE}{\ensuremath{\tilde{\vec{E}}}}
\newcommand{\nn}{\ensuremath{\vec{n}}}
\renewcommand{\SS}{\ensuremath{\ten{S}}}
\newcommand{\tSS}{\ensuremath{\tilde{\ten{S}}}}
\newcommand\tK{{\bf K}}
\newcommand{\cp}{\ensuremath{{c_p}}}
\newcommand{\cs}{\ensuremath{{c_s}}}
\newcommand{\tcs}{\ensuremath{{\tilde{c}_s}}}
\newcommand{\Vp}{\ensuremath{{V^e_p}}}
\newcommand{\Vs}{\ensuremath{{V^e_s}}}
\newcommand{\Ss}{\ensuremath{{S^e_s}}}
\newcommand{\nxnx}[1] {\ensuremath{\nn\times\left(\nn\times{#1}\right)}}

\newcommand{\dd}[2] {\ensuremath{\frac{\partial {#1}}{\partial {#2}}}}
\newcommand{\DD}[2] {\ensuremath{\frac{d {#1}}{d {#2}}}}
\newcommand{\Grad} {\ensuremath{\nabla}}
\newcommand{\Gradv}[1]{\ensuremath{\nabla_{#1}}}
\newcommand{\Div} {\ensuremath{\nabla\cdot}}
\newcommand{\Divv}[1] {\ensuremath{\nabla_{#1}\cdot}}
\newcommand{\Curl} {\ensuremath{\nabla\times}}
\newcommand{\Cten}{\ensuremath{\tenfour{C}}}

\newcommand{\eval}[2][\right]{\relax
  \ifx#1\right\relax \left.\fi#2#1\rvert}

\providecommand{\abs}[1]{\left\lvert#1\right\rvert}
\providecommand{\norm}[1]{\left\Vert#1\right\Vert}

\newcommand{\Nel} {\ensuremath{{N_T}}}
\newcommand{\Ned} {\ensuremath{{N_\text{ed}}}}
\newcommand{\De} {\ensuremath{{\mathsf{D}^e}}}
\newcommand{\Dep} {\ensuremath{{\mathsf{D}^{e'}}}}
\newcommand{\Dhat} {\ensuremath{\hat{\mathsf{D}}}}
\newcommand{\Dehat} {\ensuremath{{\hat{\mathsf{D}}^e}}}
\newcommand{\half} {\ensuremath{\frac{1}{2}}}
\newcommand{\IN} {\ensuremath{{\mathcal{I}_N}}}
\newcommand{\INe} {\ensuremath{{\mathcal{I}_{N_e}}}}

\newcommand{\mc}[1]{\mathcal{#1}}
\newcommand{\mcC}[1]{\mathcal{#1}}
\newcommand{\mb}[1]{\mathbold{#1}}
\newcommand{\mbb}[1]{\mathbb{#1}}
\newcommand{\mi}[1]{\mathit{#1}}
\newcommand{\mf}[1]{\mathfrak{#1}}
\newcommand{\Oi}[1]{\int_{\mc{D}} #1 \, d\vec{x}}
\newcommand{\TOi}[1]{\int_{0}^T \int_{\mc{D}} #1 \, d\vec{x}dt}
\newcommand{\Ti}[1]{\int_{0}^T #1 \, dt}
\newcommand{\TDei}[1]{\int_{0}^T\int_{\De} #1 \, d\vec{x}dt}
\newcommand{\TDeid}[1]{\int_{0}^T\int_{\De,N} #1 \, d\vec{x}dt}
\newcommand{\Dei}[1]{\int_{\De} #1 \, d\vec{x}}
\newcommand{\DeI}[1]{\int_{D} #1 \, d\vec{x}}
\newcommand{\DeiM}[1]{\int_{\Dhat} #1 \, d\vec{r}}
\newcommand{\DeMd}[1]{\int_{\De,N_e} #1 \, d\vec{x}}
\newcommand{\DeiMd}[1]{\int_{\Dhat,N_e} #1 \, d\vec{r}}
\newcommand{\DeMdN}[1]{\int_{D,N} #1 \, d\vec{x}}
\newcommand{\DeiMdN}[1]{\int_{\Dhat,N} #1 \, d\vec{r}}
\newcommand{\DeiMdNC}[2]{\int_{\Dhat,N_{#2}} #1 \, d\vec{r}}
\newcommand{\pDei}[1]{\int_{\partial \De} #1 \, d\vec{x}}
\newcommand{\pDeiM}[1]{\int_{\partial \Dhat} #1 \, d\vec{r}}
\newcommand{\pDeiMd}[1]{\int_{\partial \Dhat,N_e} #1 \, d\vec{r}}
\newcommand{\pDeiMdN}[1]{\int_{\partial \Dhat,N} #1 \, d\vec{r}}
\newcommand{\pDeiMdNC}[2]{\int_{\partial \Dhat,N_{#2}} #1 \, d\vec{r}}
\newcommand{\pMortar}[2]{\int_{\mc{M}_{#2}} #1 \, d\vec{r}}
\newcommand{\pMortard}[2]{\int_{\mc{M}^{#2},N_{#2}} #1 \, d\vec{r}}
\newcommand{\pMortardd}[3]{\int_{\mc{M}^{#2},N_{#3}} #1 \, d\vec{r}}
\newcommand{\pMortardh}[3]{\int_{\mc{M}_{#2},N_{#3}} #1 \, d\vec{r}}
\newcommand{\pMortardhn}[4]{\int_{{#2}_{#3},N_{#4}} #1 \, d\vec{r}}
\newcommand{\TpDei}[1]{\int_{0}^T\int_{\partial \De} #1 \, d\vec{x}}
\newcommand{\TpDeid}[1]{\int_{0}^T\int_{\partial \De,N} #1 \, d\vec{x}}
\newcommand{\Deid}[1]{\int_{\De,N_e} #1 \, d\vec{x}}
\newcommand{\DeidN}[1]{\int_{\De,N} #1 \, d\vec{x}}
\newcommand{\pDeidN}[1]{\int_{\partial \De,N} #1 \, d\vec{x}}
\newcommand{\pDeid}[1]{\int_{\partial \De,N_e} #1 \, d\vec{x}}
\newcommand{\nor}[1]{\left\| #1 \right\|}
\newcommand{\snor}[1]{\left| #1 \right|}
\newcommand{\LRp}[1]{\left( #1 \right)}
\newcommand{\LRs}[1]{\left[ #1 \right]}
\newcommand{\LRa}[1]{\left< #1 \right>}
\newcommand{\LRc}[1]{\left\{ #1 \right\}}
\newcommand{\pp}[2]{\frac{\partial #1}{\partial #2}}
\newcommand{\ppcp}[1]{\frac{\partial #1}{\partial \vec{c_p}}}
\newcommand{\ppcpj}[1]{\frac{\partial #1}{\partial c_p}}
\newcommand{\ppcpmf}[1]{\frac{\partial #1}{\partial \vec{\mf{c_p}}}}
\newcommand{\ppm}[1]{\frac{\partial #1}{\partial \vec{m}}}
\newcommand{\ppho}[3]{\frac{\partial^{#3} #1}{{\partial #2}^{#3}}}
\newcommand{\ppmi}[1]{\frac{\partial #1}{\partial m_i}}
\newcommand{\ppmj}[1]{\frac{\partial #1}{\partial m_j}}
\newcommand{\ppcpm}[1]{\frac{\partial #1}{\partial \vec{c_p}^-}}
\newcommand{\ppcpp}[1]{\frac{\partial #1}{\partial \vec{c_p}^+}}
\newcommand{\ppcs}[1]{\frac{\partial #1}{\partial \vec{c_s}}}
\newcommand{\ppcsm}[1]{\frac{\partial #1}{\partial \vec{c_s}^-}}
\newcommand{\ppcsp}[1]{\frac{\partial #1}{\partial \vec{c_s}^+}}
\newcommand{\td}[2]{\frac{d #1}{d #2}}
\newcommand{\Rvert}[1]{\left. #1 \right|}
\newcommand{\bs}[1]{\boldsymbol{#1}}
\newcommand{\ipDed}[3]{\LRp{ #1, #2}_{\De,N}^{#3}}

\newcommand{\bigO}{\mathcal{O}}

\newcommand{\iOm}[1]{\int_{\Omega} #1 \, d\Omega}
\newcommand{\iGb}[1]{\int_{\Gamma_{b}} #1 \, ds}
\newcommand{\iGs}[1]{\int_{\Gamma_{s}} #1 \, ds}
\newcommand{\iGsy}[1]{\int_{\Gamma_{s}} #1 \, ds(\mb{y})}
\newcommand{\RE}{\mbox{{I}}\hspace{-0.5ex}\mbox{{R}}}
\newcommand{\iC}[1]{\int_{0}^{2\pi} #1 \, d\theta}
\newcommand{\iGr}[1]{\int_{\Gamma_{\infty}} #1 \, ds}

\newcommand{\jump}[1] {\ensuremath{[\![{#1}]\!]}}
\newcommand{\diff}[1] {\ensuremath{\LRs{#1}}}
\newcommand{\average}[1] {\ensuremath{\LRc{\!\!\LRc{#1}\!\!}}}

\newcounter{tablecell}

\newcommand*{\numbercell}{%
  \stepcounter{tablecell}%
  \thetablecell. 
}

\newtheorem{propo}{Proposition}
\newtheorem{coro}{Corollary}
\newtheorem{theo}{Theorem}
\newtheorem{lem}{Lemma}

\newtheorem{defi}{definition}

\newtheorem{rema}{remark}

\newcommand{\figlab}[1]{\label{fig:#1}}
\newcommand{\eqnlab}[1]{\label{eq:#1}}
\newcommand{\theolab}[1]{\label{theo:#1}}
\newcommand{\corolab}[1]{\label{coro:#1}}
\newcommand{\propolab}[1]{\label{propo:#1}}
\newcommand{\lemlab}[1]{\label{lem:#1}}
\newcommand{\defilab}[1]{\label{defi:#1}}
\newcommand{\remalab}[1]{\label{rema:#1}}
\newcommand{\tablab}[1]{\label{tab:#1}}

\newcommand{\figref}[1]{\ref{fig:#1}}
\newcommand{\theoref}[1]{\ref{theo:#1}}
\newcommand{\defiref}[1]{\ref{defi:#1}}
\newcommand{\remaref}[1]{\ref{rema:#1}}
\newcommand{\cororef}[1]{\ref{coro:#1}}
\newcommand{\proporef}[1]{\ref{propo:#1}}
\newcommand{\lemref}[1]{\ref{lem:#1}}
\newcommand{\eqnref}[1]{\eqref{eq:#1}}
\newcommand{\alglab}[1]{\label{alg:#1}}
\newcommand{\algref}[1]{\ref{alg:#1}}
\newcommand{\seclab}[1]{\label{sect:#1}}
\newcommand{\secref}[1]{\ref{sect:#1}}
\newcommand{\tabref}[1]{\ref{tab:#1}}

\renewcommand{\u}{u}
\newcommand{\uk}{\u^k}
\newcommand{\upk}{\u^{+k}}
\newcommand{\ukp}{\u^{k+1}}
\newcommand{\umkp}{\u^{-k+1}}
\newcommand{\umk}{\u^{-k}}
\newcommand{\ubk}{\mb{\u}_{\mb{k}}}
\newcommand{\vbk}{\mb{\v}_{\mb{k}}}
\newcommand{\fbk}{\mb{f}_{\mb{k}}}
\newcommand{\dbk}{\mb{d}_{\mb{k}}}
\newcommand{\gbk}{\mb{g}_{\mb{k}}}
\newcommand{\lbk}{\mb{l}_{\mb{k}}}
\newcommand{\sbk}{\mb{s}_{\mb{k}}}
\newcommand{\Ob}{\mathbold{\omega}}
\newcommand{\obk}{\mathbold{\omega}_{\mb{k}}}
\newcommand{\rhok}{\rho_{\mb{k}}}
\renewcommand{\v}{v}
\newcommand{\vk}{{\v}^k}
\newcommand{\vkp}{{\v}^{k+1}}
\newcommand{\un}{\u_h}
\newcommand{\e}{e}
\newcommand{\w}{w}
\newcommand{\wb}{\mb{\w}}
\newcommand{\J}{J}
\newcommand{\Jb}{\mb{J}}
\newcommand{\Jbk}{\Jb_{\mb{k}}}
\newcommand{\kb}{\mb{k}}
\newcommand{\lb}{\mb{l}}
\newcommand{\pk}{\p_{\mb{k}}}
\newcommand{\q}{q}
\renewcommand{\r}{r}
\newcommand{\qh}{\hat{\q}}
\newcommand{\qs}{\q^*}
\newcommand{\qht}{\tilde{\qh}}
\newcommand{\qbar}{\overline{\q}}
\newcommand{\qb}{{\mb{\q}}}
\newcommand{\sig}{{\sigma}}
\newcommand{\Sk}{S_{\mb{k}}}
\newcommand{\thetah}{{\hat{\theta}}}
\newcommand{\thetas}{{\theta^*}}
\newcommand{\thetahn}{{\thetah_h}}
\newcommand{\sign}[1]{\text{ sgn}\!\LRp{#1}}
\newcommand{\sigh}{\hat{\sig}}
\newcommand{\sigs}{\sig^*}
\newcommand{\sigk}{\sig^{k}}
\newcommand{\sigpk}{\sig^{+k}}
\newcommand{\sigkp}{\sig^{k+1}}
\newcommand{\sigmkp}{\sig^{-k+1}}
\newcommand{\sigmk}{\sig^{-k}}
\newcommand{\sigb}{{\bs{\sig}}}
\newcommand{\sigbk}{\sigb^k}
\newcommand{\sigbkp}{\sigb^{k+1}}
\newcommand{\sigbn}{\sigb_h}
\newcommand{\sigbs}{\sigb^*}
\newcommand{\taub}{{\bs{\tau}}}
\newcommand{\taustar}{\tau_*}
\newcommand{\kappab}{{\bs{\kappa}}}
\newcommand{\gamb}{{\bs{\gamma}}}
\newcommand{\ut}{\tilde{\u}}
\newcommand{\sigbt}{\tilde{\sigb}}
\newcommand{\vt}{\tilde{\v}}
\newcommand{\taubt}{\tilde{\taub}}
\newcommand{\ub}{\mb{\u}}
\newcommand{\ubl}{\mb{\u}_{\lambda}}
\newcommand{\ubf}{\mb{\u}_{f}}
\newcommand{\ubp}{\mb{\u}^+}
\newcommand{\ubm}{\mb{\u}^-}
\newcommand{\vb}{\mb{\v}}
\newcommand{\vs}{\mc{\vb}}
\newcommand{\diffusion}{\mb{\d}}
\newcommand{\xb}{\mb{x}}
\newcommand{\um}{\u^-}
\newcommand{\up}{\u^+}
\newcommand{\ubkp}{\ub^{k+1}}

\newcommand{\m}{{m}}

\newcommand{\ud}{{\dot{\u}}}
\newcommand{\vd}{{\dot{\v}}}
\newcommand{\sigbd}{{\dot{\sigb}}}
\newcommand{\taubd}{\dot{\taub}}

\renewcommand{\d}{{d}}
\newcommand{\R}{{\mathbb{R}}}
\newcommand{\kap}{\kappa}
\newcommand{\F}{\mb{F}}
\newcommand{\f}{f}
\newcommand{\g}{g}
\newcommand{\fb}{\mb{\f}}
\newcommand{\gb}{\mb{\g}}
\newcommand{\gD}{g_D}
\newcommand{\pOmega}{\partial \Omega}
\newcommand{\K}{T}
\newcommand{\Kp}{\K^+}
\newcommand{\Km}{\K^-}
\newcommand{\pK}{{\partial \K}}
\newcommand{\pKin}{{\pK^{\text{in}}}}
\newcommand{\pKout}{{\pK^{\text{out}}}}
\newcommand{\Kj}{{\K_j}}
\newcommand{\Gh}{{\mc{E}_h}}
\newcommand{\Gho}{{\mc{E}_h^o}}
\newcommand{\Ghb}{{\mc{E}_h^{\partial}}}
\newcommand{\Poly}{{\mc{P}}}
\newcommand{\Cb}{{\mb{C}}}
\newcommand{\Db}{{\mb{D}}}
\newcommand{\Fb}{{\mb{F}}}
\newcommand{\Mb}{{\mb{M}}}
\newcommand{\D}{{D}}
\newcommand{\p}{{p}}
\newcommand{\px}{\partial_x}
\newcommand{\pt}{\partial_t}
\newcommand{\ppx}{\partial^2_x}
\newcommand{\pres}{{q}}
\newcommand{\presl}{{q}_{\lambda}}
\newcommand{\presf}{{q}_{f}}
\newcommand{\Ts}{{\mb{T}}}
\newcommand{\Vb}{{\mb{V}}}
\newcommand{\V}{{V}}
\newcommand{\Lam}{{\Lambda}}
\newcommand{\Lamb}{\bs{\Lambda}}
\newcommand{\mub}{\bs{\mu}}
\newcommand{\lambdab}{\bs{\lambda}}
\newcommand{\Thn}{\mathcal{T}_h}
\newcommand{\Th}{{\Ts_h}}
\newcommand{\Vh}{{\V_h}}
\newcommand{\Vbh}{{\Vb_h}}
\newcommand{\Lamh}{{\Lam_h}}
\newcommand{\Lambh}{{\Lamb_h}}
\newcommand{\Lamho}{{\overline{\Lam}_h}}
\newcommand{\Lambht}{{\Lamb_h^t}}
\newcommand{\Ltwo}{{L^2\LRp{\Omega}}}
\newcommand{\Ltwoe}{{L^2\LRp{\e}}}
\newcommand{\Ltw}{{L^2\LRp{\K}}}
\newcommand{\VhK}{{\V_h\LRp{\K}}}
\newcommand{\ThK}{{\Ts_h\LRp{\K}}}
\newcommand{\VbhK}{{\Vb_h\LRp{\K}}}
\newcommand{\Lamhe}{{\Lam_h\LRp{\e}}}
\newcommand{\Lambhe}{{\Lamb_h\LRp{\e}}}
\renewcommand{\L}{\mc{L}}
\renewcommand{\D}{\mc{D}}
\newcommand{\T}{\mc{T}}
\newcommand{\Tt}{\tilde{\T}}
\newcommand{\A}{\mb{A}}
\newcommand{\B}{\mb{B}}
\renewcommand{\D}{\mb{D}}
\newcommand{\Am}{\A^-}
\newcommand{\Ap}{\A^+}
\newcommand{\Aa}{\snor{\A}}
\newcommand{\Bb}{\mb{B}}
\newcommand{\Bbk}{\Bb_{\mb{k}}}
\newcommand{\Rb}{\mb{R}}
\newcommand{\C}{\mb{C}}
\renewcommand{\S}{\mb{S}}
\newcommand{\Sm}{\S^{-}}
\newcommand{\Sp}{\S^{+}}
\newcommand{\Spm}{\S^{\pm}}

\newcommand{\Lte}{{L^2\LRp{\Gh}}}
\newcommand{\n}{\mb{n}}
\newcommand{\nm}{\n^-}
\newcommand{\np}{\n^+}
\newcommand{\uh}{{\hat{\u}}}
\newcommand{\uhk}{{\uh}^k}
\newcommand{\uhkp}{{\uh}^{k+1}}
\newcommand{\us}{{\u^*}}
\newcommand{\uhn}{{\uh_h}}
\newcommand{\ubh}{{\hat{\ub}}}
\newcommand{\ubs}{{\ub^*}}
\newcommand{\ubht}{{\tilde{\ubh}}}
\newcommand{\uht}{{\tilde{\uh}}}
\newcommand{\uhtn}{{\uht_h}}
\newcommand{\uhd}{{\dot{\uh}}}
\newcommand{\ubhk}{\ubh^k}
\newcommand{\ubhkp}{\ubh^{k+1}}
\newcommand{\sigbh}{{\hat{\sigb}}}
\newcommand{\Fh}{{\hat{\F}}}
\newcommand{\fh}{{\bar{\f}}}
\newcommand{\Fs}{\F^*}
\renewcommand{\c}{{c}}

\newcommand{\zb}{\mb{z}}
\renewcommand{\sb}{\mb{s}}
\newcommand{\rb}{\mb{r}}
\newcommand{\betab}{\bs{\beta}}
\newcommand{\veps}{{\varepsilon}}
\newcommand{\vepsb}{\bs{\veps}}
\newcommand{\Hb}{\mb{H}}
\newcommand{\Hbt}{\Hb^t}
\newcommand\E{\mathcal{E}}
\newcommand\Eh{\mathcal{E}_h}
\newcommand\Ehi{\mathcal{E}_{h,i}}
\newcommand\Ehiint{\mathcal{E}_{h,i}^\mathrm{int}}
\newcommand\EhG{\mathcal{E}_{h}^{\Gamma}}
\newcommand\EhiG{\mathcal{E}_{h,i}^{\Gamma}}
\newcommand\EhijG{\mathcal{E}_{h,i,j}^{\Gamma}}
\newcommand\EhjiG{\mathcal{E}_{h,j,i}^{\Gamma}}
\newcommand{\Eb}{\mb{E}}
\newcommand{\Ebb}{\mbb{E}}
\newcommand{\Ebt}{\Eb^t}
\newcommand{\Ebk}{\Eb_{\mb{k}}}
\newcommand{\hb}{\mb{h}}
\newcommand{\eb}{\mb{e}}
\newcommand{\Hbh}{\hat{\mb{H}}}
\newcommand{\Ebh}{\hat{\mb{E}}}
\newcommand{\Hbs}{{\mb{H}^*}}
\newcommand{\Ebs}{{\mb{E}^*}}
\newcommand{\Ebht}{\Ebh^t}
\newcommand{\Hbht}{\Hbh^t}
\newcommand{\Lb}{\mb{L}}
\newcommand{\Gb}{\mb{G}}
\newcommand{\Lbh}{\hat{\Lb}}
\newcommand{\Lbs}{\Lb^*}
\newcommand{\Ib}{\mb{I}}
\newcommand{\I}{I}
\newcommand{\Q}{Q}
\newcommand{\Sigb}{\bs{\Sigma}}
\newcommand{\Piu}{\Pi_\u}
\newcommand{\Pisigb}{\Pi_\sigb}
\newcommand{\Z}{Z}
\newcommand{\Y}{Y}
\newcommand{\rhou}{\rho\u}
\newcommand{\rhov}{\rho\v}
\newcommand{\rhoe}{\rho_e}
\newcommand{\rhoa}{\tilde\rho}
\newcommand{\rhoaa}{\hat\rho}
\renewcommand{\P}{P}
\newcommand{\h}{H}

\newcommand{\mygreen}[1]{{\color[rgb]{0,.65,0} #1}}
\newcommand{\mywhite}[1]{{\color[rgb]{1.0,1.0,1.0} #1}}
\newcommand{\myred}[1]{{\color[rgb]{0.65,0.0,0.0} #1}}
\newcommand{\myblue}[1]{{\color[rgb]{0.0,0.0,1.0} #1}}
\newcommand{\mygray}[1]{{\color[rgb]{.15,.35,.60} #1}}
\newcommand{\mythemec}[1]{{\color[RGB]{51,108,121} #1}}
\newcommand{\mydarkred}[1]{{\color[rgb]{.6,.1,.1} #1}}
\newcommand{\mydarkblue}[1]{{\color[rgb]{.1,.1,.9} #1}}
\newcommand{\mygreenback}[1]{{\color[rgb]{.75,1,.75} #1}}
\newcommand{\myorange}[1]{{\color[rgb]{.76,.39,.13} #1}}
\newcommand{\mygrass}[1]{{\color[rgb]{.19,.64,.13} #1}}
\newcommand{\mysierp}[1]{{\color[RGB]{209,28,209} #1}}

\newcommand{\tanbui}[2]{\textcolor{blue}{#1} \textcolor{red}{#2}}

\newcommand{\vel}{\bs{\vartheta}}
\newcommand{\vels}{\vel^*}
\newcommand{\vele}{\vel^e}
\newcommand{\velh}{\hat{\vel}}
\newcommand{\velk}{\vel^k}
\newcommand{\velkp}{\vel^{k+1}}
\newcommand{\vhk}{\hat\v^{k}}

\newcommand{\phis}{\phi^*}
\newcommand{\phie}{\phi^e}
\newcommand{\phih}{\hat{\phi}}
\newcommand{\phihk}{\phih^k}
\newcommand{\phik}{\phi^k}
\newcommand{\phikp}{\phi^{k+1}}

\begin{frontmatter}

\title{NEOPIC: A Neural Operator Framework for Particle-based Kinetic Plasma Simulations}

    \author[Sriram]{Sriramkrishnan Muralikrishnan}
    \author[Sriram]{Chelsea Maria John}
    \author[Jeffrey1,Jeffrey2]{Jeffrey Kelling}
    \author[Jeffrey1,Michael1]{Michael Bussmann}
    \author[Sriram]{Robert Speck}
    \author[Sriram]{Paul Gibbon}
    \address[Sriram]{J\"ulich Supercomputing Centre, Forschungszentrum J\"ulich GmbH, J\"ulich, Germany}
    \address[Jeffrey1]{Helmholtz-Zentrum Dresden-Rossendorf, Dresden, Germany}
    \address[Jeffrey2]{Chemnitz University of Technology, Chemnitz, Germany}
    \address[Michael1]{Center for Advanced Systems Understanding (CASUS), G\"orlitz, Germany}
\begin{abstract}
We introduce NEOPIC, a framework for learning the field operator in particle-based methods applied to kinetic plasma simulations. Its key novelty is a neural operator that acts directly
on the empirical measure of the particles, rather than on a grid, and is
evaluated at the particle positions. No mesh or tree is needed and, unlike
grid-based neural operators, no deposition onto a mesh takes place, so the
particle noise that dominates the inputs of particle simulations is not
aliased. We develop NEOPIC for the self-consistent electric field in electrostatic kinetic
plasma simulations. Since it replaces only the field solve, it serves as a
plug-and-play field solver for particle-based methods: it can be trained on
data from particle-in-cell, particle-in-Fourier or tree codes, and combined
with any time integrator, including implicit schemes, for which the Jacobians
needed by the nonlinear solver are obtained exactly by automatic
differentiation. We prove that the discrete evaluation coincides with the
continuous operator at any number of particles (representation equivalence)
and derive an error analysis that separates sampling, learning and
time-stepping errors: up to the learning error, NEOPIC has the error
structure of standard Monte Carlo particle methods. We show that the particle
noise in the training data acts as an implicit regularizer, which gives
guidance on the number of training particles in relation to the particle
numbers used at inference. A complexity analysis shows that NEOPIC is an
$\mathcal{O}(N_p)$ method whose constants are set by the architecture alone,
independently of any mesh or tree. Numerical experiments on electrostatic
benchmarks in two and three dimensions, including Landau damping, two-stream
and bump-on-tail instabilities, a cyclotron mini-app (a magnetized electron
cloud with open boundaries) and a Penning trap, confirm the theoretical findings. A single
model trained on several mini-apps generated by different reference methods
captures all of them and remains stable up to twice the training time window, which points towards foundation models for fields in
particle-based methods.
\end{abstract}


\end{frontmatter}

\newcommand{\apptocnumberline}{\def\numberline##1{##1\enspace}}
{\small\setcounter{tocdepth}{2}\tableofcontents}
\newpage

\section{Introduction}
\seclab{intro}

Particle-based methods are the most widely used class of numerical methods
for kinetic plasma simulations. A large fraction of them is based on the
particle-in-cell (PIC)
formulation~\cite{dawson1983particle,birdsall2004plasma,hockney2021computer},
which comes in many flavours: the standard finite-difference and fast Fourier
transform (FFT) based PIC schemes~\cite{birdsall2004plasma,hockney2021computer},
finite element PIC~\cite{sonnendrucker1995finite,bettencourt2021empire},
structure-preserving PIC schemes based on a variational
formulation~\cite{squire2012geometric,jianyuan2018structure,campos2022variational,campos2024variational}
or on a discretization of the underlying Hamiltonian
structure~\cite{kraus2017gempic,he2016hamiltonian}, energy-conserving fully
implicit and semi-implicit PIC
schemes~\cite{lewis1970energy,chen2011energy,markidis2011energy,lapenta2017exactly,chen2020semi,pinto2022semi},
and PIC schemes with noise
reduction~\cite{denton1995deltaf,ricketson2016sparse,wang2011particle}.
Mesh-free methods based on the Barnes--Hut algorithm~\cite{barnes1986hierarchical},
the fast multipole method (FMM)~\cite{greengard1987fast} and the
particle-in-Fourier (PIF) scheme~\cite{evstatiev2013variational,mitchell2019efficient}
form the other class of particle-based methods used in kinetic plasma
simulations~\cite{pfalzner19943d,ambrosiano1988fast,christlieb2006grid,gibbon2010progress,masek2010mesh,siddi2017mesh,shen2024particle,muralikrishnan2025parapif}.
The particle-particle particle-mesh (P$^3$M)
algorithm~\cite{hockney2021computer} resolves the short-range interactions in
strongly coupled plasmas, and binary Monte Carlo
collisions~\cite{takizuka1977binary,nanbu1997theory} and Langevin
approaches~\cite{manheimer1997langevin,cadjan1999langevin} add collisions to
particle-based simulations.

An important advantage of particle-based methods is that, unlike grid-based
continuum approaches, they do not suffer from the curse of dimensionality.
The price is statistical noise, which decreases only slowly, as
$\mc{O}(N_p^{-1/2})$ with the total number of particles $N_p$. Particle
methods are also well suited to large-scale parallelization, and most
production codes for kinetic plasma simulations use one of the methods listed
above~\cite{winkel2012massively,vay2018warp,burau2010picongpu,muralikrishnan2024scaling}.

In standard PIC schemes, the particles live in continuous phase space while
the fields are represented on a grid. Modes that the grid does not resolve
are therefore aliased onto the resolved low-frequency modes, which leads to
the finite grid instability and to a loss of energy conservation when the
Debye length is not
resolved~\cite{langdon1970theory,huang2016finite,birdsall2004plasma}. The
statistical noise decreases as $\mc{O}(P_c^{-1/2})$, where $P_c$ is the
number of particles per cell (typically between $\mc{O}(10)$ and
$\mc{O}(1000)$). Together with the requirement that the mesh resolves the
Debye length, this leads to large numbers of grid points and particles and
makes high-fidelity simulations expensive. Advanced PIC schemes, such as
energy-conserving~\cite{lewis1970energy},
structure-preserving~\cite{squire2012geometric,jianyuan2018structure,campos2022variational,campos2024variational,kraus2017gempic,he2016hamiltonian},
semi-implicit and fully
implicit~\cite{chen2011energy,markidis2011energy,lapenta2017exactly,chen2020semi,pinto2022semi}
and sparse-grid PIC schemes~\cite{ricketson2016sparse}, mitigate these
effects and can reduce the total cost in certain regimes by reducing the
number of particles and/or grid points. Many of them, however, depart
considerably from the simplicity of the standard PIC algorithm, which makes
them harder to adopt in applications and existing codes.

Among the mesh-free methods, PIF retains the simplicity of standard PIC but
replaces the mesh-based scatter and gather operations by a spectral scatter
from the particles to Fourier space and a gather back to the particles, both
computed with nonuniform FFTs (NUFFTs)~\cite{dutt1993fast,barnett2019parallel}.
The result is a structure-preserving scheme with excellent stability and
conservation properties~\cite{evstatiev2013variational,mitchell2019efficient},
and PIF has recently been extended to non-periodic boundary
conditions~\cite{shen2024particle} and to massively parallel
simulations~\cite{muralikrishnan2025parapif,mayani2026comparison}. It remains
expensive, however, when many Fourier modes are needed, and it cannot handle
boundary conditions in complex domains.

Tree-based methods such as Barnes--Hut and the FMM involve no grid and
therefore do not alias the particle noise, and they are adaptive by
construction, which uniform, structured PIC schemes are not. Although their asymptotic cost is
comparable to ($\mc{O}(N_p\log N_p)$ for Barnes--Hut) or even lower than
($\mc{O}(N_p)$ for the FMM) that of PIC and PIF schemes, the constant in
front of $N_p$ is much larger, so they remain slower than PIC at the particle
numbers typical of kinetic plasma simulations. Parallelizing tree
structures is also more involved than parallelizing standard PIC schemes.

Each of these schemes has its strengths and limitations and outperforms the
others in certain regions of the parameter space (number of particles, type
of distribution, required accuracy) and in certain applications. Consider
now an abstract particle--field method, in which the particles are advanced
by ordinary differential equations with fields computed by a field solver.
All of the methods above differ only in how the fields at the particles are
computed from the particle attributes, that is, in the field operator
$\mathbf{E}_h[\mu_N]$ (Section~\secref{abstract}), which is specific to each
method. In this work we ask \emph{whether a field operator for this abstract
particle--field method can be learned from the data of any or all of
these methods}, each preferably run in the part of the parameter space where
it performs best, so that the data contribute complementary information.
Such a learned field solver can then be coupled to any production particle
code, with or without a mesh, in a plug-and-play fashion: at each time step
the code passes the particle attributes to the solver, which returns the
self-consistent fields at the particle positions.

To this end, we present NEOPIC, a neural operator (``NEO'') framework for the
field operator in particle-based methods. Although no mesh is involved and
the framework can be used with any particle-based method, the suffix ``PIC''
refers to the PIC community at which it is aimed, since most kinetic plasma
simulations are performed with PIC codes. In this work we focus on
collisionless electrostatic kinetic plasma simulations; the framework is
general and will be extended to other kinetic models (Darwin, electromagnetic
and collisional) in future work. NEOPIC is trained on data from a variety of
particle-based methods, such as PIC, PIF and Barnes--Hut, and the learned
operator can be used with any time integrator, including implicit ones, for
which the exact Jacobian needed by the nonlinear solver is obtained by
automatic differentiation.

Neural operators~\cite{patel2021physics,kovachki2023neural,lu2021learning} are a natural choice
for this task, since the object to be learned is an operator. Since PIC, PIF
and tree codes are discretizations of the same continuum field operator,
their data can be used to train a single neural operator, which is defined at
the continuum level; where the teachers differ, the model learns a weighted
combination of their responses (Section~\secref{noise}). The discretization
invariance of neural operators further allows different numbers of particles
in training and inference, and since the particle attributes at each time
step act as the input of the operator, the model must generalize to the
inputs generated by the time loop of the particle method. Neural
operators, however, typically act on function spaces with much more
regularity than the empirical particle measures that NEOPIC must handle. A
key novelty of our work is therefore a nonuniform Fourier neural operator
(NUFNO) that acts entirely on particle measures: the inputs are Monte Carlo
particle positions, i.e.\ sums of Dirac measures, the output is the electric
field at these particles, and the input enters the network only through the
particle measure. The layers are those of~\cite{lingsch2023beyond}, but the
data they act on, and hence the properties of the operator, are different
(Section~\ref{sec:related}). We show in Section~\secref{reno} that the NUFNO
applied to these particle measures is the continuous operator evaluated on
the same input, without any aliasing, a property known as representation
equivalence~\cite{bartolucci2023representation}. This is critical in our
context: without it, the high frequencies of the particle noise could be
aliased onto the low frequencies retained in the FNO, an effect that the
nonlinear activation functions make worse. This could cause numerical
instabilities similar to the finite grid instability and grid heating of
standard PIC schemes and would partly defeat the purpose of the approach.
Our further contributions are the following.
\begin{enumerate}
\item An error analysis of the particle method with the NEOPIC field operator
  (Section~\secref{error}), which shows that the total error consists of a
  Monte Carlo sampling error, a time-stepping error and a model error. When
  the model error is small, the total error has the same structure as that of
  other Monte Carlo particle methods.
\item An analysis showing that the noise in the training data acts as an
  implicit regularizer that drives the trained model to behave like its
  teacher (Section~\secref{noise}). The noise term leads to an
  accuracy--robustness trade-off, which gives guidance on the number of
  particles to train with, given the particle numbers expected at inference.
\item A complexity analysis (Section~\secref{complexity}) showing that NEOPIC
  is an $\mc{O}(N_p)$ particle method, which is optimal in its asymptotic
  complexity, with a constant determined by the architecture of the operator
  and independent of the number of particles and of any grid size.
\item A validation of the theoretical findings on a number of benchmark
  kinetic plasma simulations in 2D and 3D (Section~\secref{results}). We
  also show that a single model can be trained on data from different methods
  and benchmarks and remains stable up to twice the training time horizon,
  with errors that only grow moderately beyond the end of the training data. This encourages the development of a foundation model for
  fields in particle-based methods, which could be fine-tuned with a small
  number of application-specific samples, as has been done for grid-based
  neural operators~\cite{herde2024poseidon,subramanian2023towards}; we leave
  this to future work.
\end{enumerate}

\subsection{Related work}
\label{sec:related}
Deep learning has been actively studied both for solving partial differential
equations (PDEs) and as a component of hybrid numerical methods. One reason
for this shift is that emerging hardware architectures support deep learning
workloads better than conventional numerical methods. Since solving a PDE
amounts to learning its solution operator, neural
operators~\cite{li2020fourier,kovachki2023neural,fanaskov2023spectral,lu2021learning},
which act on function spaces rather than on finite-dimensional vectors, have
gained traction in recent years, with applications in many
domains~\cite{pathak2022fourcastnet,azizzadenesheli2024neural}. As noted by
McGreivy and Hakim~\cite{mcgreivy2024weak}, however, many such approaches
have been compared against weak baselines, so that reported gains have to be
assessed case by case.

For kinetic plasma simulations, Aguilar and Markidis~\cite{aguilar2021deep}
used a multi-layer perceptron and a convolutional neural network (CNN) to
compute the electric field at the grid points from the particle phase space in
a PIC scheme. In studies of the 1D-1V two-stream instability, the deep
learning based PIC simulation reproduced the growth rate of the standard PIC
scheme but did not conserve energy or momentum; it was, however, stable
against the cold-beam instability that affects standard PIC schemes. No
speedup over the standard PIC scheme was observed, and since standard neural
networks were used, training and inference must use the same resolution and
the network must be retrained when the parameters change. Kube et
al.~\cite{kube2021machine} used a CNN to accelerate implicit PIC schemes by
augmenting the search space of the iterative linear solver with the output
of the CNN, which reduced the number of solver iterations by $25\%$ for 1D-1V
electron plasma oscillations; a net speedup was left to future work.

Lyu and Zhong~\cite{lyu2026latent} recently replaced the grid-based field
solve of PIC by a latent residual-closure FNO: an autoencoder compresses the
particle-deposited charge and current densities, a coarse FNO predicts the
dominant field response and a second FNO corrects it on the full grid. Used
as a stand-alone field solver and as an initial guess for an iterative
solver, it reduces the field-solve time of a bounded scrape-off-layer model by
factors of $33$ and $18$, respectively, measured against a finite-difference
iterative solver. Like other grid-based surrogates, it operates on the
deposited sources, which carry the particle noise and require a dedicated
compression stage. NEOPIC instead acts directly on the particles: no charge
is deposited on a mesh, so the particle noise is not aliased, and the
deposition, field solve and interpolation of PIC are replaced together.

Graph neural networks and graph neural operators offer another route to
particle and point-cloud data. Graph network simulators learn the complete
update of particle systems, such as fluids and granular media, by message
passing between neighbouring particles~\cite{sanchez2020learning}, and of
mesh-based simulations~\cite{pfaff2020learning,brandstetter2022message};
Carvalho et al.~\cite{carvalho2024learning} applied this approach to a
one-dimensional electrostatic sheet model of a plasma. Graph neural
operators~\cite{li2020neural} approximate integral kernels by sums over
neighbourhoods in a graph, and the geometry-informed neural operator
(GINO)~\cite{li2023geometry} maps point clouds with a graph neural operator
onto a latent regular grid, on which an FNO is applied. Message passing over
neighbourhoods captures long-range interactions such as the Coulomb field
only through many layers or large radii, and the map onto a latent grid in
GINO acts like a learned deposition, which samples the particle noise on a
grid. NEOPIC instead learns only the field, couples all particles through the
Fourier modes of the particle measure, and leaves the particle update to the
time integrator of the particle method.

The architectures closest to ours are the direct spectral evaluation of
Lingsch et al.~\cite{lingsch2023beyond} and the long-range convolutional (LRC)
layer of Peng et al.~\cite{peng2023efficient}. Lingsch et al.\ evaluate the
Fourier transforms of the FNO layers directly at nonuniform points, and
NEOPIC uses the same layer. There, however, the points are sampling locations
of continuum input and output functions, the sums over the points approximate
integrals, and two discretizations of the same function differ by a
quadrature error in every layer (Remark~\ref{rem:measure_input}). In NEOPIC
the point measure is itself the state of the particle method, and the
discrete evaluation is the continuous operator applied to it
(Theorem~\ref{thm:reno}). Peng et al.\ treat a point cloud as a weighted sum
of Dirac measures and apply a trainable Fourier multiplier with a NUFFT, so
that, as in NEOPIC, the convolution with the point measure is computed
without a mesh. Their LRC layer is a single linear convolution with a
parameterized, Yukawa-type multiplier, whose output is passed to pointwise
networks and combined with short-range descriptors to learn the energies and
forces of model $N$-body potentials. NEOPIC instead composes several Fourier
layers with general learned multipliers and pointwise nonlinearities, which
allows nonlinear and non-translation-invariant field operators such as those
of PIC and Barnes--Hut teachers, predicts the field directly, and is trained
on noisy particle data and used inside the time loop of kinetic simulations.
Neither work analyses the role of particle noise in the training data, the
dependence on the number of particles, or the error of the coupled
particle--field evolution, which are central here.

The remainder of the paper is organized as follows.
Section~\secref{vp} recalls the Vlasov--Poisson system and the common
structure of particle--field methods, and Section~\secref{neopic} introduces
NEOPIC. Section~\secref{architecture} defines the NUFNO on particle measures
and proves its representation equivalence. Section~\secref{error} contains
the error analysis, Section~\secref{noise} the role of the particle noise in
the training data, and Section~\secref{complexity} the complexity analysis.
Numerical results are presented in Section~\secref{results}, and conclusions
are drawn in Section~\secref{conclusions}. The appendices collect further
remarks on the architecture and the analysis, the diagnostics of the
cyclotron mini-app, a convergence study in the time step, a demonstration of
implicit time integration, and the dependence of the cost on the
architecture.
\section{The Vlasov--Poisson system and particle--field methods}
\seclab{vp}

\subsection{The Vlasov--Poisson system}

We consider a non-relativistic, collisionless electron plasma in an
electrostatic setting, with possibly external electric and magnetic fields.
The electron phase-space distribution $f(t,\mathbf{x},\mathbf{v})$,
$\mathbf{x}\in\Omega\subseteq\mathbb{R}^d$, $\mathbf{v}\in\mathbb{R}^d$, evolves
according to the Vlasov equation
\begin{equation}
  \partial_tf+\mathbf{v}\cdot\nabla_{\mathbf{x}}f
  +\frac{q_e}{m_e}\big(\mathbf{E}+\mathbf{v}\times\mathbf{B}_{\mathrm{ext}}\big)\cdot\nabla_{\mathbf{v}}f=0,
  \label{eq:vlasov}
\end{equation}
where $q_e$ and $m_e$ are the electron charge and mass,
$\mathbf{E}=\mathbf{E}_{\mathrm{sc}}+\mathbf{E}_{\mathrm{ext}}$ is the total
electric field, and $\mathbf{E}_{\mathrm{ext}}$ and $\mathbf{B}_{\mathrm{ext}}$
are given external fields. The self-consistent (space-charge) field is
obtained from the Poisson equation for the electrostatic potential $\phi$,
\begin{equation}
  -\Delta\phi=\frac{\rho_e+\rho_i}{\varepsilon_0},\qquad
  \mathbf{E}_{\mathrm{sc}}=-\nabla\phi,\qquad
  \rho_e(t,\mathbf{x})=q_e\int_{\mathbb{R}^d}f(t,\mathbf{x},\mathbf{v})\,d\mathbf{v},
  \label{eq:poisson}
\end{equation}
with the permittivity of free space $\varepsilon_0$. Two kinds of boundary
conditions are considered. On a periodic box
$\Omega=\prod_{i=1}^d[0,L_i)$ the electrons are immersed in a uniform,
immobile, neutralizing ion background with charge density
$\rho_i=-Q_{\mathrm{tot}}/|\Omega|$, where
$Q_{\mathrm{tot}}=\int_\Omega\rho_e\,d\mathbf{x}$ is the total electron
charge, so that the right-hand side of~\eqref{eq:poisson} has zero mean. On $\Omega=\mathbb{R}^d$ (open, or
free-space, boundary conditions) there is no background, $\rho_i=0$, and
$\phi$ decays at infinity; this is the setting of non-neutral plasmas, such
as an electron cloud confined by an external magnetic field. External fields
are also used on periodic domains, for example for the Penning trap of
Section~\secref{multi3d}. In both cases
\eqref{eq:poisson} defines the space-charge field as a linear, nonlocal
functional of the charge density,
\begin{equation}
  \mathbf{E}_{\mathrm{sc}}[\rho](\mathbf{x})
  =-\frac{1}{\varepsilon_0}\int_\Omega\nabla G(\mathbf{x}-\mathbf{y})\,
   \big(\rho_e(\mathbf{y})+\rho_i\big)\,d\mathbf{y},
  \label{eq:field_green}
\end{equation}
where $G$ is the Green's function of $-\Delta$: the periodic Green's function
on the torus, or $G(\mathbf{x})=-\frac{1}{2\pi}\ln|\mathbf{x}|$ for $d=2$ and
$G(\mathbf{x})=\frac{1}{4\pi|\mathbf{x}|}$ for $d=3$ in free space. On the
periodic box the Fourier representation of~\eqref{eq:field_green} is the
operator $\mathcal{E}$ of~\eqref{eq:exact_field_op}, which we use in the
error analysis in normalized form. The characteristics of~\eqref{eq:vlasov}
are the trajectories of individual electrons,
\begin{equation}
  \frac{d\mathbf{x}}{dt}=\mathbf{v},\qquad
  \frac{d\mathbf{v}}{dt}=\frac{q_e}{m_e}\big(\mathbf{E}(t,\mathbf{x})+\mathbf{v}\times\mathbf{B}_{\mathrm{ext}}(t,\mathbf{x})\big),
  \label{eq:characteristics}
\end{equation}
along which $f$ is constant. The system conserves the total charge and the
total energy (including the potential energy of static external fields) and,
in the absence of external fields, the total momentum; in a uniform magnetic
field in free space, the canonical angular momentum about the field axis is
conserved as well. In the numerical experiments we use normalized units with
$\varepsilon_0=1$ and $q_e/m_e=-1$.

\subsection{Particle discretization}

Particle methods discretize $f$ in a Lagrangian way by $N_p$ computational
particles (macro-particles) with positions $\mathbf{x}_p$, velocities
$\mathbf{v}_p$ and weights $w_p>0$, $\sum_pw_p=1$,
\begin{equation}
  f_N(t,\mathbf{x},\mathbf{v})=\frac{Q_{\mathrm{tot}}}{q_e}\sum_{p=1}^{N_p}w_p\,
  \delta\big(\mathbf{x}-\mathbf{x}_p(t)\big)\,\delta\big(\mathbf{v}-\mathbf{v}_p(t)\big),
  \qquad
  \mu_N(t)=\sum_{p=1}^{N_p}w_p\,\delta_{\mathbf{x}_p(t)},
  \label{eq:particle_measure}
\end{equation}
so that particle $p$ carries the charge $Q_{\mathrm{tot}}w_p$ and the charge
density is $Q_{\mathrm{tot}}\mu_N$. At $t=0$ the particles are sampled from
$f(0,\cdot,\cdot)$, and the particles then follow the
characteristics~\eqref{eq:characteristics} with the field evaluated at their
positions,
\begin{equation}
  \begin{gathered}
  \frac{d\mathbf{x}_p}{dt}=\mathbf{v}_p,\qquad
  \frac{d\mathbf{v}_p}{dt}=\frac{q_e}{m_e}\big(\mathbf{E}_p+\mathbf{v}_p\times\mathbf{B}_{\mathrm{ext}}(\mathbf{x}_p)\big),\\
  \mathbf{E}_p=\mathbf{E}_{h}[\mu_N](\mathbf{x}_p)+\mathbf{E}_{\mathrm{ext}}(\mathbf{x}_p),
  \end{gathered}
  \label{eq:particle_odes}
\end{equation}
where $\mathbf{E}_h[\mu_N]$ is an approximation of the space-charge
field~\eqref{eq:field_green} for the charge measure of the particles. The
weights are carried along the trajectories, so that the total charge is
conserved exactly, and $f_N$ is a weak solution of~\eqref{eq:vlasov} with the
field $\mathbf{E}_p$. The ODEs~\eqref{eq:particle_odes} are typically integrated with an explicit
one-step method, in this work the kick--drift--kick leapfrog
scheme~\eqref{eq:kdk} or, with a magnetic field, the Boris
scheme~\cite{birdsall2004plasma}, and each time step then requires one
evaluation of the field at all particles. Implicit schemes, such as the
implicit midpoint (Crank--Nicolson) rule~\eqref{eq:implicit_midpoint} of
\secref{implicit1d}, couple
the particles and the field in a nonlinear system at every time
step~\cite{chen2011energy,markidis2011energy}, so that each step requires
several field evaluations, one per nonlinear iteration and, in Newton--Krylov
solvers, one per Jacobian--vector product. Note that the field of a point measure is
singular at the particles. Every particle method therefore regularizes the
interaction, by a shape function, a truncation of the Fourier series or a
softened kernel, and $\mathbf{E}_h$ is an approximation of the field of a
smoothed charge distribution.

\subsection{An abstract particle--field method}
\seclab{abstract}

Particle methods for~\eqref{eq:vlasov}--\eqref{eq:poisson} differ mostly in
how $\mathbf{E}_h[\mu_N]$ is computed. They share a common structure, which
is also found in particle methods for other particle--field systems, such as
the Vlasov--Maxwell system, vortex methods and immersed boundary
methods~\cite{cottet2000vortex,peskin2002immersed}: particles evolve by ODEs,
fields by PDEs, and the two are coupled by mappings between them. In the
general form, with particle quantities $\mathbf{p}$, a field $\mathbf{f}_l$
that acts on the particles and a field $\mathbf{f}_r$ that the particles
generate,
\begin{equation}
  \frac{d\mathbf{p}}{dt}=\mathbf{g}(\mathbf{f}_l,\mathbf{p}),\qquad
  \mathscr{L}(\mathbf{f}_l)=\mathbf{h}(\mathbf{f}_r),\qquad
  \mathbf{f}_r=\mathscr{R}_{\mathbf{p}\to\mathbf{f}_r}(\mathbf{p}),\qquad
  \mathbf{f}_l|_{\mathbf{p}}=\mathscr{R}^*_{\mathbf{f}_l\to\mathbf{p}}(\mathbf{f}_l),
  \label{eq:abstract}
\end{equation}
where $\mathscr{L}$ is a (possibly nonlinear) differential or
differential-algebraic operator, $\mathscr{R}_{\mathbf{p}\to\mathbf{f}_r}$ maps
the particles to a field (deposition, or scatter) and
$\mathscr{R}^*_{\mathbf{f}_l\to\mathbf{p}}$ maps a field to the particles
(interpolation, or gather). For the Vlasov--Poisson system,
$\mathbf{p}=(\mathbf{x}_p,\mathbf{v}_p,w_p)_p$, $\mathbf{f}_r=\rho_e$,
$\mathbf{f}_l=(\phi,\mathbf{E}_{\mathrm{sc}})$,
$\mathscr{L}=-\varepsilon_0\Delta$ with $\mathbf{E}_{\mathrm{sc}}=-\nabla\phi$,
$\mathbf{h}(\rho_e)=\rho_e+\rho_i$, and $\mathbf{g}$ is the right-hand side
of~\eqref{eq:particle_odes}.

\paragraph{The kernel form} Here $\mathscr{L}$ is linear, $\mathbf{h}$ is
affine and $\mathscr{R}$ is linear in the weights, so that the field can be
eliminated from~\eqref{eq:abstract}: the particles interact through the
kernel of the composition
$\mathscr{R}^*\mathscr{L}^{-1}\mathbf{h}\,\mathscr{R}$, i.e.\ the field at
one particle generated by another particle of unit weight. By~\eqref{eq:field_green}
with $\rho_e=Q_{\mathrm{tot}}\mu_N$, and the zero-mean Green's function on
periodic domains, which accounts for the background,
\begin{equation}
  \mathbf{E}_{\mathrm{sc}}[\mu_N](\mathbf{x})=\int_\Omega K(\mathbf{x},\mathbf{y})\,d\mu_N(\mathbf{y}),
  \qquad
  K(\mathbf{x},\mathbf{y})=-\frac{Q_{\mathrm{tot}}}{\varepsilon_0}\nabla G(\mathbf{x}-\mathbf{y}).
  \label{eq:exact_kernel}
\end{equation}
The kernel is singular at $\mathbf{x}=\mathbf{y}$, translation invariant and
odd. Every particle method replaces it by a regularized, discrete kernel
$K_h$, which may depend on the whole particle configuration $X$, and
evaluates, for $p=1,\dots,N_p$,
\begin{equation}
  \mathbf{E}_h[\mu_N](\mathbf{x}_p)
  =\int_\Omega K_h(\mathbf{x}_p,\mathbf{y};X)\,d\mu_N(\mathbf{y})
  =\sum_{q=1}^{N_p}w_q\,K_h(\mathbf{x}_p,\mathbf{x}_q;X)
  \label{eq:kernel_form}
\end{equation}
The methods differ in how $K_h$ is defined and how the sum is evaluated; the
field is linear in the weights in all of them, and the dependence on the
positions is nonlinear. Table~\tabref{instantiations} lists the kernels of the
methods considered here, and Fig.~\figref{neopic_loop}(a) shows the resulting
particle--field cycle. If $K_h(\mathbf{x},\mathbf{y};X)=-K_h(\mathbf{y},\mathbf{x};X)$,
the self-forces of the particles sum to zero,
$\sum_pw_p\mathbf{E}_h(\mathbf{x}_p)=\sum_{p,q}w_pw_qK_h(\mathbf{x}_p,\mathbf{x}_q;X)=0$,
and the scheme conserves momentum in the absence of external
fields~\cite{birdsall2004plasma,hockney2021computer}.

\paragraph{Mesh- and basis-based methods} In particle-in-cell and
particle-in-Fourier methods the field is represented in a finite-dimensional
space $V_h$ that does not depend on the particles, with deposition functions
$\psi_j$ and interpolation functions $\psi^*_i$. The
deposition
$\mathscr{R}:\mu_N\mapsto(\int\psi_j\,d\mu_N)_j$, the discrete field solve
$\mathscr{L}_h^{-1}$ on $V_h$, which maps a charge density to the field, and
the interpolation $\mathscr{R}^*$ are separate linear maps; the gradient is
part of $\mathscr{L}_h^{-1}$ or, as for finite elements, of the
interpolation functions. The kernel factorizes,
\begin{equation}
  K_h(\mathbf{x},\mathbf{y})=Q_{\mathrm{tot}}\sum_{i,j}\psi^*_i(\mathbf{x})\,
  \big(\mathsf{L}_h^{-1}\big)_{ij}\,\psi_j(\mathbf{y}),
  \qquad
  \mathbf{E}_h=\mathscr{R}^*\,\mathscr{L}_h^{-1}\,\mathscr{R}\,\mu_N .
  \label{eq:factorized_kernel}
\end{equation}
In matrix form $\mathbf{E}=\mathsf{R}^*\mathsf{L}_h^{-1}\mathsf{R}\,\mathbf{q}$,
with matrices that depend on the positions and the vector $\mathbf{q}$ of
particle charges~\cite{muralikrishnan2025parapif}. This kernel does not depend
on $X$ beyond the evaluation points. It is odd if the interpolation functions
are the complex conjugates of the deposition functions,
$\psi^*_i=\overline{\psi_i}$, and $\mathsf{L}_h^{-1}$,
which then contains the discrete gradient, is antisymmetric and translation
invariant; variational
derivations lead to energy-conserving schemes~\cite{evstatiev2013variational}.

\paragraph{Particle-in-cell} The PIC method~\cite{birdsall2004plasma,hockney2021computer,dawson1983particle}
uses a mesh: $\psi_j$ are shape functions or basis functions associated with
the mesh, and $\mathsf{L}_h$ is a discretization of the Poisson equation,
solved with FFTs on periodic boxes, with FFT-based free-space
solvers~\cite{vico2016fast,mayani2025Vico}, or with sparse solvers for finite
difference, finite element or discontinuous Galerkin discretizations on
general geometries~\cite{jacobs2006high,pfeiffer2019particle}.

\paragraph{Particle-in-Fourier} The PIF
method~\cite{langdon1970theory,evstatiev2013variational,huang2016finite,mitchell2019efficient}
uses a truncated Fourier basis, $\psi_{\mathbf{k}}(\mathbf{y})=\hat S_{\mathbf{k}}e^{-\mathrm{i}\mathbf{k}\cdot\mathbf{y}}$
and $\psi^*_{\mathbf{k}}=\overline{\psi_{\mathbf{k}}}$, for which
$\mathsf{L}_h^{-1}$ is the diagonal multiplier
$-\mathrm{i}\mathbf{k}/(\varepsilon_0|\Omega||\mathbf{k}|^2)$ and the kernel is a
truncated, smoothed Fourier series of~\eqref{eq:exact_kernel}
(Table~\tabref{instantiations}). The deposition and interpolation are
type-1 and type-2 nonuniform Fourier transforms, computed with
NUFFTs~\cite{dutt1993fast,barnett2019parallel,mitchell2019efficient}, so that
the cost is comparable to PIC with a larger constant. 

\paragraph{Tree codes and fast multipole methods} Direct summation
evaluates~\eqref{eq:kernel_form} with a softened kernel
$K_h=-\frac{Q_{\mathrm{tot}}}{\varepsilon_0}\nabla G_\delta(\mathbf{x}-\mathbf{y})$
at a cost of $\mathcal{O}(N_p^2)$. Tree codes~\cite{barnes1986hierarchical}
reduce the cost by grouping the source particles into the cells of a
hierarchical tree built on the particles. For each target, cells that appear
small enough from it (opening angle $\vartheta$) contribute through a truncated
multipole expansion about their centre, and the remaining particles, nearby,
contribute directly. The kernel $K_h(\mathbf{x}_p,\mathbf{y};X)$ is therefore
the softened kernel in the near field and its multipole approximation in the
far field, and it depends on the configuration $X$ through the tree and on the
target through its interaction list. It is in general not odd, so
Barnes--Hut does not conserve momentum exactly. The fast multipole method~\cite{greengard1987fast}
also expands the field about the targets, in local expansions, and translates
multipole into local expansions between well-separated cells. Its far-field
kernel is then separable cell pair by cell pair, a hierarchical low-rank
version of~\eqref{eq:factorized_kernel} with the multipole-to-local
translations in the role of $\mathsf{L}_h^{-1}$. In both, no field equation is
discretized and no mesh is used, which makes them a natural choice for open
boundaries and non-uniform particle distributions; but they can also be used on
periodic domains.
 Periodicity can be treated by a tree walk over
the box and its nearest images and an Ewald summation of the multipole
expansion of the box for the remaining images.

\begin{table}[!htbp]
  \centering
  \caption{Discrete kernels $K_h$ in~\eqref{eq:kernel_form} (up to the factor
  $Q_{\mathrm{tot}}/\varepsilon_0$). $\psi_j,\psi^*_i$: deposition and
  interpolation functions; $\mathsf{L}_h$: discrete field operator;
  $\hat S_{\mathbf{k}}$: Fourier transform of the shape function;
  $\mathcal{K}$: retained Fourier modes; $G_\delta$: softened Green's function;
  $T_P$: multipole expansion of order $P$ about the cell centre $\mathbf{x}_c$;
  $\mathcal{F}(\mathbf{x})$, $\mathcal{N}(\mathbf{x})$: far-field cells and
  near-field particles of target $\mathbf{x}$ for opening angle $\vartheta$;
  $\varepsilon$: NUFFT tolerance. Structure: whether the kernel factorizes
  through a space that does not depend on the particles,
  as in~\eqref{eq:factorized_kernel}, whether it is translation invariant, and
  whether it depends on the particle configuration $X$. Costs per field
  evaluation, with $C_h$ the cost of the mesh solve.}
  \tablab{instantiations}
  \small
  \renewcommand{\arraystretch}{1.3}
  \begin{tabular}{@{}>{\raggedright\arraybackslash}p{0.17\linewidth}>{\raggedright\arraybackslash}p{0.46\linewidth}>{\raggedright\arraybackslash}p{0.3\linewidth}@{}}
    \toprule
    Method & Kernel $K_h(\mathbf{x},\mathbf{y};X)$ & Structure, cost \\
    \midrule
    PIC (mesh)
      & $\sum_{i,j}\psi^*_i(\mathbf{x})\,(\mathsf{L}_h^{-1})_{ij}\,\psi_j(\mathbf{y})$
      & factorized~\eqref{eq:factorized_kernel}; $\mathcal{O}(N_p)+C_h$ \\
    PIF
      & $\frac{1}{|\Omega|}\sum_{\mathbf{k}\in\mathcal{K}\setminus\{0\}}|\hat S_{\mathbf{k}}|^2\,\frac{-\mathrm{i}\mathbf{k}}{|\mathbf{k}|^2}\,e^{\mathrm{i}\mathbf{k}\cdot(\mathbf{x}-\mathbf{y})}$
      & factorized~\eqref{eq:factorized_kernel}, translation invariant; $\mathcal{O}\big((|\log\varepsilon|+1)^dN_p+|\mathcal{K}|\log|\mathcal{K}|\big)$ \\
    Direct sum
      & $-\nabla G_\delta(\mathbf{x}-\mathbf{y})$
      & translation invariant; $\mathcal{O}(N_p^2)$ \\
    Barnes--Hut
      & $-\nabla G_\delta(\mathbf{x}-\mathbf{y})$ for $\mathbf{y}\in\mathcal{N}(\mathbf{x})$;
        $T_P[-\nabla G](\mathbf{x}-\mathbf{x}_c;\mathbf{y}-\mathbf{x}_c)$ for $\mathbf{y}\in c\in\mathcal{F}(\mathbf{x})$
      & depends on $X$ and on the target; $\mathcal{O}(N_p\log N_p)$ \\
    FMM
      & near field as above; far field separable between well-separated cells
        (multipole-to-local translations)
      & hierarchical low rank, depends on $X$; $\mathcal{O}(N_p)$ \\
    \bottomrule
  \end{tabular}
\end{table}

\paragraph{The common interface} All these methods share the particle
part~\eqref{eq:particle_odes} and its time integration, and differ only in the
kernel $K_h$ and in how~\eqref{eq:kernel_form} is evaluated. Each has its own
discretization parameters (mesh size and shape function, Fourier modes and
NUFFT tolerance, opening angle, expansion order and softening), its own errors
(aliasing, truncation, multipole truncation) and its own cost, and each
approximates the same operator~\eqref{eq:exact_kernel}. At the level of the
particle scheme, a field solver is therefore a map
\begin{equation}
  \{(\mathbf{x}_p,w_p)\}_{p=1}^{N_p}\ \longmapsto\ \{\mathbf{E}_p\}_{p=1}^{N_p},
  \label{eq:interface}
\end{equation}
which depends on the particles only through $\mu_N$, takes any number of
particles, and returns the field at the particles. This interface is what
NEOPIC learns.

\section{NEOPIC: learning the field operator of particle methods}
\seclab{neopic}

\begin{figure}[!htbp]
  \centering
  \includegraphics[width=\linewidth]{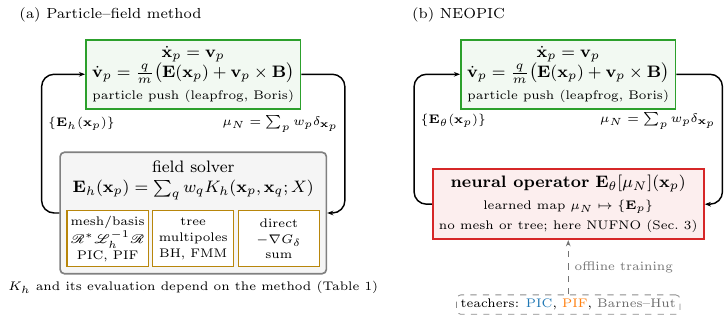}
  \caption{(a) The particle--field cycle of a particle method for the
  Vlasov--Poisson system. The particles are advanced by the time integrator
  (green); the field solver (grey box) evaluates the field at the particles
  with a discrete kernel $K_h$~\eqref{eq:kernel_form}, whose form depends on
  the method (Table~\tabref{instantiations}): a mesh or basis
  ($\mathscr{R}^*\mathscr{L}_h^{-1}\mathscr{R}$), a tree, or direct
  summation. (b) NEOPIC keeps the particles and the time integrator and
  replaces the field solver by a neural operator $\mathbf{E}_\theta$, which
  maps the charge measure of the particles directly to the field at the
  particles. It is trained offline on data from particle methods; in this work
  PIC, PIF and Barnes--Hut, and the neural operator is the NUFNO of
  Section~\secref{architecture}. With an implicit time integrator
  (\secref{implicit1d}) the cycle is the same, but the field is
  evaluated several times per step within the nonlinear solve.}
  \figlab{neopic_loop}
\end{figure}

\subsection{Learning the particle--field--particle map}

NEOPIC replaces the field solver, i.e.\ the evaluation of~\eqref{eq:kernel_form}
in~\eqref{eq:particle_odes}, by a neural operator $\mathbf{E}_\theta$,
\begin{equation}
  \mathbf{E}_p=\mathbf{E}_\theta[\mu_N](\mathbf{x}_p)+\mathbf{E}_{\mathrm{ext}}(\mathbf{x}_p),
  \qquad p=1,\dots,N_p,
  \label{eq:neopic_field}
\end{equation}
which is trained on data generated by particle methods
(Fig.~\figref{neopic_loop}(b)). Everything else is left unchanged: the
particles are sampled and weighted as before, they follow the
characteristics~\eqref{eq:particle_odes}, and they are advanced by the same
time integrator. NEOPIC accompanied by a time integrator is thus a particle method in its own right, with
a learned instance of the interface~\eqref{eq:interface}, and its learned
operator can replace the field solver of an existing particle code without
changes to the rest of the code. Two choices define the approach: what is
learned, and what the learned operator acts on.

\paragraph{Learning the field, not the dynamics} A common approach to data-driven
surrogates of time-dependent problems is to learn the evolution of the whole
state, i.e.\ to map the state at one time to the state at a later time,
either directly or autoregressively~\cite{mousavi2025rigno,wen2025geometry}.
For kinetic plasmas the state is the distribution function in the
$2d$-dimensional phase space, or its particle representation, and its
evolution combines transport along the characteristics, which is
nonlinear and multi-scale, with the self-consistent field. NEOPIC learns only
the map from the particles to the field at a fixed time, which is the
instantaneous solution operator of an elliptic problem. This operator is much
simpler than the flow map: it does not depend on time or on the velocities,
it is nonlocal but smooth, and for the Vlasov--Poisson system it is linear in
the charge density. NEOPIC itself does not rely on this linearity: the map
from the positions to the field at the positions is nonlinear anyway, and the
same approach applies to particle methods with nonlinear field equations,
although parts of the analysis in Section~\secref{noise} use the linearity of
the teacher.
The particle part, on the other hand, is the least expensive and most
scalable part of a particle method: it is pointwise and embarrassingly
parallel. Keeping it has three consequences. First, the particles follow the
characteristics of the system with the learned field, so that charge is
conserved exactly and the errors of the scheme enter only through the field.
This is the basis of the error analysis of Section~\secref{error}, which
follows the analysis of PIC schemes. Second, structure-preserving time
integrators, such as the symplectic leapfrog scheme or the Boris scheme for
magnetized plasmas, can be used as in any particle method, and so can
implicit schemes, for which the Jacobian of the learned field with respect to
the positions is available exactly by automatic differentiation. In
\secref{implicit1d} we demonstrate this with the implicit midpoint rule in one
dimension, where NEOPIC remains stable and nearly energy conserving well
beyond the stability limit of the leapfrog scheme. Third, the learned
operator depends neither on time nor on the time step: time stepping beyond
the training horizon, or with another time step, uses the same field solver
at every step, and the dynamics is integrated by the scheme itself rather
than extrapolated by a learned flow map. Only its inputs, the particle configurations, may move away from
the training distribution (Section~\secref{results}).

\paragraph{Learning the whole map, not the field solve} Neural operators are
usually trained to approximate solution operators between functions on a
mesh. Applied to a particle method this would mean learning
$\mathscr{L}_h^{-1}$ on a mesh, and keeping the deposition $\mathscr{R}$ and
the interpolation $\mathscr{R}^*$ of PIC similar to \cite{lyu2026latent}. The mesh would then bring back the
constraints of PIC: (i) the mesh must resolve the Debye length to avoid the
finite grid instability and grid
heating~\cite{langdon1970theory,birdsall2004plasma}; (ii) the noise of the
deposited density is controlled by the number of particles per cell, so that
the number of particles grows with the number of mesh cells; (iii) unless the
mesh is adaptive, the adaptivity of the particles to the distribution is lost
in the field computation; and (iv) in electromagnetic simulations, the
Courant--Friedrichs--Lewy condition of the field solver on the mesh requires
a time step much smaller than that needed by the particles. NEOPIC instead
approximates the whole map~\eqref{eq:kernel_form}, for mesh-based methods the
whole composition $\mathscr{R}^*\mathscr{L}_h^{-1}\mathscr{R}$, which takes
particle data as input and returns fields at the particles. No mesh enters the
scheme, and data from methods without a field discretization, such as tree
codes, can be used for training on the same footing as data from PIC and
PIF.

\paragraph{Operators on measures} The input of this map is not a function
sampled at given points but the measure $\mu_N$, an ensemble of random
Monte Carlo samples of the charge distribution. The operator must
approximate the continuum field operator~\eqref{eq:field_green} from
input--output pairs defined over such measure-valued data, and the same
operator must apply to measures with any number of particles and in the
limit $N_p\to\infty$. This departs from the classical neural operator
framework~\cite{kovachki2023neural}, which is formulated for functions with
some regularity and their discretizations on grids or point clouds. Neural
operators on point clouds, such as the direct spectral evaluation
of~\cite{lingsch2023beyond} and the graph-based operators
of~\cite{mousavi2025rigno,wen2025geometry}, treat the points as sampling
locations of an input function. In NEOPIC the points, with their weights,
are the input itself. The distinction is made precise in
Section~\secref{architecture} (Remark~\ref{rem:measure_input}).

\subsection{Training on several particle methods}

The training data are pairs of particle configurations and fields at the
particles, $\big(\mu^n_N,\{\mathbf{E}^{\mathrm{ref}}_p(t_n)\}_p\big)$, collected
along the trajectories of reference runs with a particle method, which we call
the teacher. Each teacher evaluates~\eqref{eq:kernel_form} with its own
kernel $K_h$, an approximation of the same kernel~\eqref{eq:exact_kernel}
with errors of a different nature. A single operator can therefore be trained on data from several
teachers and several physical scenarios (``mini-apps''). This also allows
each teacher's data to be used where it is accurate: for every mini-app, the
training data can be generated with the method best suited to it. The trained
operator inherits the error of its teacher in addition to its own learning
error (Remark~\ref{rem:resolution_mismatch}). Within one training set all
mini-apps use the same number of particles. The operator itself does not
depend on this number, nor on the discretization of the teacher, and NEOPIC
is evaluated with particle numbers different from those used in training
(Section~\secref{landau2d}). The training procedure and the data of the
individual experiments are given in Section~\secref{results}.

\subsection{Requirements on the learned operator}
\seclab{requirements}

From the role of $\mathbf{E}_\theta$ in~\eqref{eq:neopic_field} and the
interface~\eqref{eq:interface}, we require that the learned operator
\begin{enumerate}
  \item[(R1)] takes as input a weighted point measure with any number of
  particles at arbitrary positions, and depends on the particle data only
  through the measure, i.e.\ is invariant under permutations of the particles;
  \item[(R2)] returns the field at the particle positions, and can be
  evaluated at any other point;
  \item[(R3)] is one operator on measures with unit total charge, so that its
  outputs for different $N_p$ are consistent and its limit $N_p\to\infty$ is
  the same operator applied to the continuous density;
  \item[(R4)] does not transfer the particle data to a mesh, so that the
  particle noise is not aliased onto the resolved scales;
  \item[(R5)] is global, as the Coulomb interaction, and costs
  $\mathcal{O}(N_p)$ operations per evaluation for a fixed resolution of the
  operator, which is the minimum complexity needed by any particle-based method;
  \item[(R6)] is Lipschitz continuous in the particle positions, as required
  by the error analysis of Section~\secref{error}.
\end{enumerate}
A Fourier neural operator (FNO)~\cite{li2020fourier} on a uniform grid
violates (R1) and (R4), since the particle data must first be deposited on
the grid; operators with fixed sensor locations such as
DeepONet~\cite{lu2021learning} violate (R1). Graph-based
operators~\cite{mousavi2025rigno,wen2025geometry} can be made to satisfy
(R1)--(R5), with aggregations that are sums weighted by the particle weights
and a deterministic graph that is rebuilt at every time step as the particles
move. They propagate the long-range interaction through local message passing
on multiscale graphs, or through attention on latent tokens, and with hard
neighbourhoods (fixed radius or nearest neighbours) their output jumps when an
edge appears or disappears, which violates (R6). The FNO with the Fourier
transforms evaluated directly at the points~\cite{lingsch2023beyond}, written
with the transforms as integrals against the input measure, satisfies all six:
its spectral layers are global, the nonuniform transforms of type 1 and 2 map
the particles to Fourier modes and back without a mesh, with NUFFTs they cost
$\mathcal{O}(N_p)$ operations for a fixed number of modes, and they are smooth
in the positions. Moreover, its layers are exact integrals against the measure
at every $N_p$, and its spectral layers match the structure of the Poisson
operator, which is diagonal in Fourier space.
We call this operator the nonuniform Fourier neural operator (NUFNO). Its
structure is shown in Fig.~\figref{nufno_schematic}: the particle measure is
normalized to the torus, lifted pointwise, processed by Fourier layers whose
transforms are nonuniform transforms between the particles and a fixed set
of modes, and projected pointwise to the field at the particles. The
nonuniform transforms play the role of $\mathscr{R}$ and $\mathscr{R}^*$ in
PIF, and the spectral weights and nonlinearities replace the fixed multiplier
$-\mathrm{i}\mathbf{k}/|\mathbf{k}|^2$; NUFNO can thus be seen as a learned
generalization of PIF that is trained on data from any particle method.
NEOPIC itself is not tied to this choice: any neural operator that satisfies
(R1)--(R6) can be used in~\eqref{eq:neopic_field}, and the NUFNO is the
instance used in this work. The
next section defines the NUFNO, shows that it satisfies (R1)--(R4), and that it
is a representation equivalent neural operator; (R6) is
Lemma~\ref{lem:lipschitz}, for a fixed normalization of the positions
(\ref{rem:normalization}) and a periodic lifting
(Remark~\ref{rem:periodic_lifting}), and (R5) is shown in
Section~\secref{complexity}. In the remainder of the paper
we use NEOPIC both for the whole particle scheme including time integration with a learned field operator and, where no confusion arises, only for its learned 
field operator.

\begin{figure}[!htbp]
  \centering
  \includegraphics[width=\linewidth]{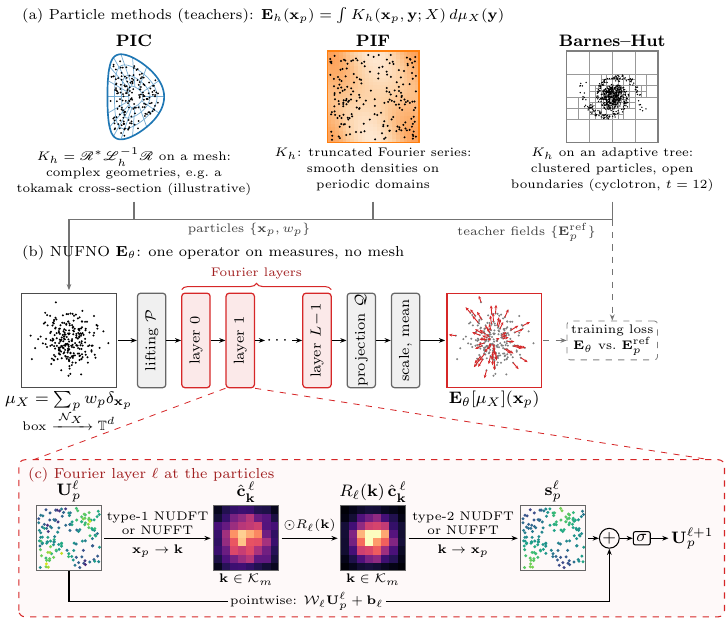}
  \caption{The nonuniform Fourier neural operator (NUFNO) of NEOPIC.
  (a) The training data come from particle methods with different
  kernels $K_h$ (Table~\tabref{instantiations}): PIC uses a mesh, PIF a
  truncated Fourier basis, and Barnes--Hut multipole expansions on the cells
  of a tree. The examples illustrate their complementary strengths: PIF for
  smooth densities on periodic domains, Barnes--Hut for clustered particles
  and open boundaries (particles sampled from the density of the cyclotron
  mini-app at $t=12$), and PIC for complex geometries (a tokamak
  cross-section, for illustration only; the PIC teachers of this work use
  uniform periodic meshes, Section~\secref{results}). All of them provide the
  particles $\{\mathbf{x}_p,w_p\}$ and the field at the particles. (b) The NUFNO maps the charge measure $\mu_X$,
  normalized from its bounding box to the torus by $\mathcal{N}_X$,
  Eq.~\eqref{eq:minmax}, through a pointwise lifting $\mathcal{P}$, $L$ Fourier
  layers and a pointwise projection $\mathcal{Q}$ to the field at the
  particles, followed by the scaling and mean subtraction
  of~\eqref{eq:output}. (c) A Fourier layer: a type-1 nonuniform transform (NUDFT or NUFFT)
  from the particle features $\mathbf{U}^\ell_p$ to the coefficients on the
  modes $\mathcal{K}_m$~\eqref{eq:forward}, multiplication by the spectral
  weights, a type-2 nonuniform transform back to the
  particles~\eqref{eq:inverse}, and the pointwise update~\eqref{eq:disc_layer}.
  No mesh is used at any stage. The point clouds and the spectra are
  illustrations; in the PIF panel the colour shows the smooth density from
  which the particles are sampled.}
  \figlab{nufno_schematic}
\end{figure}

\section{Nonuniform Fourier neural operator}
\seclab{architecture}
In this section we define the nonuniform Fourier neural operator and its discrete evaluation,
and show that, unlike a grid-based FNO, it is a representation equivalent
neural operator (ReNO) in the sense of~\cite{bartolucci2023representation}:
exactly when the transforms are evaluated directly, and up to the NUFFT
tolerance when they are evaluated with NUFFTs. This property is the reason
why the learned operator does not alias the particle noise.

\subsection{Notation}
\seclab{notation}

Throughout the paper, $|\mathbf{a}|$ is the Euclidean norm of a vector
$\mathbf{a}\in\mathbb{R}^n$ or $\mathbb{C}^n$, e.g.\ of the $d$ components of a
position or of the $d_v$ channels of a feature vector, $|\mathbf{a}|_1$ the
sum of the absolute values of its components, and $\|A\|$ the matrix norm
induced by $|\cdot|$, i.e.\ the largest singular value of $A$. On the torus
$\mathbb{T}^d$, $|\mathbf{x}-\mathbf{y}|$ denotes the periodic (minimum-image)
distance. For a scalar, $|a|$ is the absolute value; for a finite set, such as
the set of modes $\mathcal{K}_m$, $|\mathcal{K}_m|$ is the number of its
elements; and for a domain, $|\Omega|$ is its volume. For functions, $\|\cdot\|_{L^p}$ is the $L^p$ norm over the domain
indicated, and $L^2(\mu)$ is the space of functions that are square integrable
with respect to a measure $\mu$, with inner product
$\langle\mathbf{f},\mathbf{g}\rangle=\int\mathbf{f}\cdot\overline{\mathbf{g}}\,d\mu$; in a
general Hilbert space, $\langle\cdot,\cdot\rangle$ and $\|\cdot\|$ are its inner
product and norm, and $\mathrm{dist}(\mathbf{a},V)=\inf_{\mathbf{b}\in V}\|\mathbf{a}-\mathbf{b}\|$ is the distance of an element
$\mathbf{a}$ from a set $V$. A vector or sequence with entries $a_p$ indexed by
$p$ is written $(a_p)_p$, or $(a_p)_{p=1}^{N_p}$, and $(a_{p,i})_{p,i}$ for two
indices; for example $(\mathbf{f}(\mathbf{x}_p))_p$ is the vector of the values of
$\mathbf{f}$ at the particles. $\delta_{\mathbf{x}}$ is the Dirac measure at
$\mathbf{x}$, i.e.\ $\int\varphi\,d\delta_{\mathbf{x}}=\varphi(\mathbf{x})$;
$X\sim\rho$ means that the random variable $X$ has the probability density
$\rho$; iid stands for independent and identically distributed; and
$\mathbb{E}$ is the expectation. $\mathcal{M}_1(\Omega)$ is the set of
probability measures on $\Omega$. A map $F$ is Lipschitz with constant $L_F$ if
$|F(\mathbf{a})-F(\mathbf{b})|\le L_F|\mathbf{a}-\mathbf{b}|$.

\subsection{Particles as a measure}

Positions are normalized per sample, i.e.\ at every time step, by a min--max
normalization based on the particle cloud, without padding, for all
mini-apps. With $\mathbf{a}_X$ and $\mathbf{b}_X$ the componentwise
minimum and maximum of the positions, the affine map
\begin{equation}
  \mathcal{N}_X(\mathbf{x})=-\pi\mathbf{1}+2\pi\,(\mathbf{x}-\mathbf{a}_X)\oslash(\mathbf{b}_X-\mathbf{a}_X),
  \label{eq:minmax}
\end{equation}
with componentwise division $\oslash$, maps the bounding box of the cloud onto
the torus $\mathbb{T}^d=[-\pi,\pi)^d$, $d\in\{1,2,3\}$. In the following,
$\mathbf{x}_p$ denotes the normalized positions; the consequences of the
normalization are discussed in \ref{rem:normalization}. A configuration of $N_p$
particles with positions $X=(\mathbf{x}_p)_{p=1}^{N_p}$ and weights $w_p>0$,
$\sum_pw_p=1$, defines the charge measure
\begin{equation}
  \mu_X=\sum_{p=1}^{N_p}w_p\,\delta_{\mathbf{x}_p},
  \qquad
  \hat\mu_{\mathbf{k}}=\int_{\mathbb{T}^d}e^{-\mathrm{i}\mathbf{k}\cdot\mathbf{y}}\,d\mu(\mathbf{y})
  =\sum_{p=1}^{N_p}w_p\,e^{-\mathrm{i}\mathbf{k}\cdot\mathbf{x}_p},
  \label{eq:charge_measure}
\end{equation}
normalized to unit total charge. In all our experiments the particles carry
equal charges, $w_p=1/N_p$, and we then write $\mu_N=\mu_X$; general weights
cover importance-sampled initial data. In the limit $N_p\to\infty$ the empirical
measure approaches the absolutely continuous measure $\rho\,d\mathbf{x}$,
defined by $(\rho\,d\mathbf{x})(A)=\int_A\rho\,d\mathbf{x}$ for every region
$A$, where $\rho$ is the normalized charge density (unit total charge;
$\rho_e=Q_{\mathrm{tot}}\rho$ in the notation of Section~\secref{vp}). The
field operator of the Vlasov--Poisson system, \eqref{eq:exact_field_op} below,
is defined for both, and so is NEOPIC. Let
$\mathcal{K}_m=\{-m,\dots,m-1\}^d$ be the set of retained wave vectors, with
$M=|\mathcal{K}_m|=(2m)^d$.

\subsection{The operator}

For a probability measure $\mu$ on $\mathbb{T}^d$, NEOPIC is defined by a
lifting, $L$ Fourier layers and a projection,
\begin{align}
  \mathbf{u}^0_\mu(\mathbf{x}) &= \mathcal{P}(\mathbf{x}),
  \label{eq:err_lift}\\
  \mathbf{u}^{\ell+1}_\mu(\mathbf{x}) &= \sigma\Big(
     \mathcal{W}_\ell\,\mathbf{u}^\ell_\mu(\mathbf{x}) + \mathbf{b}_\ell
     + \mathrm{Re}\!\sum_{\mathbf{k}\in\mathcal{K}_m}
       R_\ell(\mathbf{k})\,\hat{\mathbf{c}}^{\,\ell}_\mu(\mathbf{k})\,
       e^{\mathrm{i}\mathbf{k}\cdot\mathbf{x}} \Big),
  \label{eq:err_layer}\\
  \hat{\mathbf{c}}^{\,\ell}_\mu(\mathbf{k})
   &= \int_{\mathbb{T}^d}\mathbf{u}^\ell_\mu(\mathbf{y})\,
     e^{-\mathrm{i}\mathbf{k}\cdot\mathbf{y}}\,d\mu(\mathbf{y}),
  \label{eq:err_coef}\\
  \mathcal{G}_\theta[\mu](\mathbf{x}) &= \mathcal{Q}\big(\mathbf{u}^L_\mu(\mathbf{x})\big),
  \label{eq:err_proj}
\end{align}
for $\ell=0,\dots,L-1$ and $\mathbf{x}\in\mathbb{T}^d$. Here
$\mathbf{u}^\ell_\mu(\mathbf{x})\in\mathbb{R}^{d_v}$, the lifting
$\mathcal{P}:[-\pi,\pi)^d\to\mathbb{R}^{d_v}$, defined on the fundamental
domain of $\mathbb{T}^d$, the skip weights
$\mathcal{W}_\ell\in\mathbb{R}^{d_v\times d_v}$ and the projection
$\mathcal{Q}:\mathbb{R}^{d_v}\to\mathbb{R}^d$ act pointwise, $\sigma$ is applied
componentwise, and $R_\ell(\mathbf{k})\in\mathbb{C}^{d_v\times d_v}$ are the
spectral weights. The real part is needed because $\mathcal{K}_m$ is not
symmetric under $\mathbf{k}\mapsto-\mathbf{k}$. As in~\cite{li2020fourier},
the activation is omitted in the last Fourier layer: for $\ell=L-1$, $\sigma$
in~\eqref{eq:err_layer} is replaced by the identity. In this work $\mathcal{P}$ is a
single linear layer applied to the normalized positions, which is not
periodic (Remark~\ref{rem:periodic_lifting}),
$\mathcal{W}_\ell$ a convolution with kernel size one, $\mathcal{Q}$ a two-layer
MLP with hidden width $h$, and $\sigma$ is the GELU, which is smooth and
Lipschitz with $L_\sigma\approx1.13$, as required in
Sections~\secref{error} and~\secref{noise}. The field used to push the
particles is
\begin{equation}
  \mathbf{E}_\theta[\mu](\mathbf{x})=s\Big(\mathcal{G}_\theta[\mu](\mathbf{x})
  -c_{\mathrm{per}}\int_{\mathbb{T}^d}\mathcal{G}_\theta[\mu]\,d\mu\Big),
  \label{eq:output}
\end{equation}
where $s$ is the constant that restores physical units
($Q_{\mathrm{tot}}/|\Omega|^{(d-1)/d}$, see Section~\secref{results}), and
$c_{\mathrm{per}}=1$ for periodic and $c_{\mathrm{per}}=0$ for open boundaries; other field
boundary conditions are discussed in Remark~\ref{rem:field_bc}. On periodic domains
the subtraction of the mean makes the sum of the forces on the particles
vanish, so that momentum is conserved. In open domains no correction is
applied; the self-forces of the exact free-space field also sum to zero, but
the learned field is not constrained to do so.

Three properties of~\eqref{eq:err_lift}--\eqref{eq:output} are used throughout
the paper. First, the measure enters only through the coefficients
$\hat{\mathbf{c}}^{\,\ell}_\mu(\mathbf{k})$, which are integrals against $\mu$;
all other operations are pointwise. Second, the functions
$\mathbf{u}^\ell_\mu$ are defined for every $\mathbf{x}\in\mathbb{T}^d$, not
only at the particles, so the field can be evaluated anywhere. Third, the
definition refers neither to the number of particles nor to a grid. It
defines one operator on probability measures, which can be applied to an
empirical measure with any number of particles as well as to the measure
$\rho\,d\mathbf{x}$ of a continuous density. The normalization of the weights to unit total
charge, i.e.\ a charge $1/N_p$ per particle, is what makes this possible.

\paragraph{Constants of the network}
\seclab{constants}
The bounds in this and the following sections use the following constants.
$L_\sigma$, $L_{\mathcal{P}}$ and $L_{\mathcal{Q}}$ are Lipschitz constants of
$\sigma$, of $\mathcal{P}$ (with respect to the periodic distance) and of
$\mathcal{Q}$, and $\sigma(0)$ is the value of the activation at zero, which
vanishes for ReLU and GELU. For the last Fourier layer, $\sigma$ is the
identity, with $L_\sigma=1$ and $\sigma(0)=0$. The spectral weights of layer
$\ell$ enter through
\begin{equation}
  r^\ell_0=\sum_{\mathbf{k}\in\mathcal{K}_m}\|R_\ell(\mathbf{k})\|,
  \qquad
  r^\ell_1=\sum_{\mathbf{k}\in\mathcal{K}_m}|\mathbf{k}|\,\|R_\ell(\mathbf{k})\|,
  \label{eq:r_constants}
\end{equation}
the features are bounded by the constants
\begin{equation}
  M_0=\sup_{\mathbf{x}}|\mathcal{P}(\mathbf{x})|,\qquad
  M_{\ell+1}=|\sigma(0)|+L_\sigma\big((\|\mathcal{W}_\ell\|+r^\ell_0)M_\ell+|\mathbf{b}_\ell|\big)
  \label{eq:M_bound}
\end{equation}
(Lemma~\ref{lem:lipschitz}), and
$\Gamma_\ell=L_\sigma\big(\|\mathcal{W}_\ell\|+r^\ell_0\big)$ is a Lipschitz
constant of layer $\ell$ with respect to the maximum over the particles of the
features.

\subsection{Discrete evaluation}

For a particle measure $\mu_X$ let $\mathbf{U}^\ell_p\in\mathbb{R}^{d_v}$ be the
features of particle $p$ in layer $\ell$, with $\mathbf{U}^0_p=\mathcal{P}(\mathbf{x}_p)$.
A Fourier layer consists of a forward transform, a spectral product, an
inverse transform and the pointwise update,
\begin{align}
  \hat{\mathbf{c}}^{\,\ell}_{\mathbf{k}} &= \sum_{p=1}^{N_p}w_p\,\mathbf{U}^\ell_p\,
     e^{-\mathrm{i}\mathbf{k}\cdot\mathbf{x}_p},\qquad\mathbf{k}\in\mathcal{K}_m,
  \label{eq:forward}\\
  \mathbf{s}^\ell_p &= \sum_{\mathbf{k}\in\mathcal{K}_m}R_\ell(\mathbf{k})\,
     \hat{\mathbf{c}}^{\,\ell}_{\mathbf{k}}\,e^{\mathrm{i}\mathbf{k}\cdot\mathbf{x}_p},
     \qquad p=1,\dots,N_p,
  \label{eq:inverse}\\
  \mathbf{U}^{\ell+1}_p &= \sigma\big(\mathcal{W}_\ell\,\mathbf{U}^\ell_p
     +\mathbf{b}_\ell+\mathrm{Re}\,\mathbf{s}^\ell_p\big),
  \label{eq:disc_layer}
\end{align}
with $\sigma$ replaced by the identity for $\ell=L-1$, and the output is
$\mathbf{E}_p=s\big(\mathcal{Q}(\mathbf{U}^L_p)-\sum_qw_q\mathcal{Q}(\mathbf{U}^L_q)\big)$.
With the nonuniform Vandermonde matrix
$V_{p\mathbf{k}}=e^{\mathrm{-i}\mathbf{k}\cdot\mathbf{x}_p}\in\mathbb{C}^{N_p\times M}$
and $D_w=\mathrm{diag}(w_p)$, \eqref{eq:forward} is $\hat{\mathbf{c}}=VD_w\mathbf{U}$
(a type-1 transform) and \eqref{eq:inverse} is $\mathbf{s}=V^*(R\odot\hat{\mathbf{c}})$
(a type-2 transform), where $(R\odot\hat{\mathbf{c}})_{\mathbf{k}}=R_\ell(\mathbf{k})\,\hat{\mathbf{c}}_{\mathbf{k}}$ is the product with the spectral weights mode by mode; $V$ and $V^*$ act on each of the $d_v$ channels. They are evaluated
either directly (NUDFT), as in the direct spectral evaluation
of~\cite{lingsch2023beyond}, or with NUFFTs~\cite{barnett2019parallel,shih2021cufinufft},
which compute the same sums to a prescribed tolerance $\varepsilon$ at a cost
that is linear in $N_p$ for a fixed number of modes
(Section~\secref{complexity}). This is the architecture
of~\cite{lingsch2023beyond}, in which the weights $w_p$ make the forward
transform an integral against the charge measure. In our implementation,
following~\cite{lingsch2023beyond}, each transform is scaled by
$\sqrt{d/N_p}$, so that their composition carries the weights $d/N_p$. This
is equivalent to $w_p=1/N_p$ with $R_\ell(\mathbf{k})$ replaced by
$d\,R_\ell(\mathbf{k})$, and all statements below apply with these effective
spectral weights.

\begin{lemma}[Exactness of the particle evaluation]
\label{lem:exact_eval}
If~\eqref{eq:forward}--\eqref{eq:disc_layer} are evaluated exactly (NUDFT),
then $\mathbf{U}^\ell_p=\mathbf{u}^\ell_{\mu_X}(\mathbf{x}_p)$ for all $\ell$
and $p$, and $\mathbf{E}_p=\mathbf{E}_\theta[\mu_X](\mathbf{x}_p)$. In
particular, the output depends on the particle data $(X,w)$ only through the
measure $\mu_X$: it is invariant under permutations of the particles, under
splitting a particle into several particles of smaller weight at the same
position, and under merging particles at the same position.
\end{lemma}

\begin{proof}
For $\ell=0$ the claim is the definition. If
$\mathbf{U}^\ell_q=\mathbf{u}^\ell_{\mu_X}(\mathbf{x}_q)$ for all $q$, then
\eqref{eq:forward} is the integral~\eqref{eq:err_coef} for $\mu=\mu_X$, and
\eqref{eq:inverse}--\eqref{eq:disc_layer} is~\eqref{eq:err_layer} at
$\mathbf{x}=\mathbf{x}_p$. The mean in~\eqref{eq:output} is likewise the sum
over the particles. The operations in~\eqref{eq:err_lift}--\eqref{eq:output}
depend on $\mu$ only, which gives the invariances.
\end{proof}

The lemma is elementary, but it is the property that distinguishes the
nonuniform evaluation from a grid-based one. \emph{The discrete computation is not
an approximation of the operator~\eqref{eq:err_lift}--\eqref{eq:output} but
its exact evaluation for the input that the particle method provides, at
every particle number.}

\subsection{Representation equivalence}
\seclab{reno}

We recall the framework of~\cite{bartolucci2023representation}. Let
$\mathcal{H}$ be a separable Hilbert space. A countable family
$\Psi=\{\psi_i\}_{i\in I}\subset\mathcal{H}$ is a frame sequence if it is a
frame for its closed span $\mathcal{M}_\Psi$, i.e.\ if there are constants
$0<A_\Psi\le B_\Psi$ with
$A_\Psi\|f\|^2\le\sum_i|\langle f,\psi_i\rangle|^2\le B_\Psi\|f\|^2$ for all
$f\in\mathcal{M}_\Psi$, where $\langle\cdot,\cdot\rangle$ and $\|\cdot\|$ are the
inner product and the norm of $\mathcal{H}$. Its synthesis operator is
$T_\Psi:\ell^2(I)\to\mathcal{H}$, $T_\Psi\mathbf{c}=\sum_ic_i\psi_i$, which maps
coefficients to a function. Its adjoint, the analysis operator
$T^*_\Psi f=(\langle f,\psi_i\rangle)_{i\in I}$, maps a function to its inner
products with the frame elements, and $S=T_\Psi T^*_\Psi$ is the frame
operator, which is invertible on $\mathcal{M}_\Psi$. The pseudo-inverse
$T_\Psi^\dagger f=(\langle f,S^{-1}\psi_i\rangle)_{i\in I}$ maps a function to its
discrete representation, the coefficients with respect to the canonical dual
frame, and $T_\Psi T_\Psi^\dagger$ is the orthogonal projection onto $\mathcal{M}_\Psi$.
The implementation of an operator $U:\mathcal{H}\to\mathcal{Y}$ is performed through a 
discrete map $u:\ell^2(I)\to\ell^2(J)$ for frame sequences $\Psi$ of
$\mathcal{H}$ and $\Phi$ of $\mathcal{Y}$.

\begin{definition}[Operator aliasing and ReNO~{\cite[Def.~3.1, Def.~3.4, Sec.~3.4]{bartolucci2023representation}}]
\label{def:reno}
The aliasing error of $(U,u)$ is
$\varepsilon(U,u,\Psi,\Phi)=U-T_\Phi\circ u\circ T_\Psi^\dagger$.
The pair $(U,u)$ is a representation equivalent neural operator (ReNO) if $\varepsilon(U,u,\Psi,\Phi)=0$ for every pair of
frame sequences with $\mathrm{Dom}\,U\subseteq\mathcal{M}_\Psi$ and
$\mathrm{Ran}\,U\subseteq\mathcal{M}_\Phi$, and an $\varepsilon$-ReNO if
$\|\varepsilon(U,u,\Psi,\Phi)\|\le\varepsilon$ for all such pairs. For
linear $U$, $\|\varepsilon(U,u,\Psi,\Phi)\|$ is the operator norm. For
nonlinear $U$ the aliasing error is a nonlinear map, for which no operator
norm is defined, and we measure instead the largest error over a bounded set
$B$ of inputs,
$\|\varepsilon\|_B=\sup_{f\in B}\big\|Uf-T_\Phi\, u(T_\Psi^\dagger f)\big\|$.
\end{definition}

A ReNO has a discrete representation $u(\Psi,\Phi)=T_\Phi^\dagger\circ U\circ T_\Psi$
for each admissible discretization, and any two of them are equivalent: they
are representations of the same operator~\cite[Prop.~3.6]{bartolucci2023representation}.
ReNOs are stable under composition~\cite[Prop.~3.7]{bartolucci2023representation},
so it suffices to examine the layers. Pointwise linear maps, such as the
lifting and the skip connections, and the spectral product are ReNOs. 

\paragraph{The hidden spaces of NEOPIC} The transforms
in~\eqref{eq:err_coef} are integrals against the input measure, so a layer
needs the hidden functions only on the support of $\mu$, and only the values
at the particles are ever used. We therefore take as hidden spaces
\begin{equation}
  \mathcal{H}_\mu=L^2\big(\mathbb{T}^d,\mu;\mathbb{R}^{d_v}\big),
  \label{eq:hidden_space}
\end{equation}
and as layer operators $U_\ell:\mathcal{H}_\mu\to\mathcal{H}_\mu$ the maps
of~\eqref{eq:err_layer},
\begin{equation}
  (U_\ell\mathbf{f})(\mathbf{x})=\sigma\Big(\mathcal{W}_\ell\mathbf{f}(\mathbf{x})+\mathbf{b}_\ell
  +\mathrm{Re}\sum_{\mathbf{k}\in\mathcal{K}_m}R_\ell(\mathbf{k})
  \Big(\int_{\mathbb{T}^d}\mathbf{f}(\mathbf{y})e^{-\mathrm{i}\mathbf{k}\cdot\mathbf{y}}\,d\mu(\mathbf{y})\Big)
  e^{\mathrm{i}\mathbf{k}\cdot\mathbf{x}}\Big).
  \label{eq:layer_operator}
\end{equation}
These are well defined on $\mu$-equivalence classes,
and they map bounded functions to bounded functions. The lifting provides the
input $\mathbf{u}^0_\mu=\mathcal{P}|_{\mathrm{supp}\,\mu}\in\mathcal{H}_\mu$, and
the output map $U_{\mathcal{Q}}:\mathcal{H}_\mu\to L^2(\mathbb{T}^d,\mu;\mathbb{R}^d)$,
$\mathbf{f}\mapsto s\big(\mathcal{Q}\circ\mathbf{f}-\int\mathcal{Q}\circ\mathbf{f}\,d\mu\big)$,
is~\eqref{eq:output}. For a particle measure, the natural discretization is
by the values at the particles: let
\begin{equation}
  \Psi_X=\big\{\chi_{\{\mathbf{x}_p\}}\,\mathbf{e}_i:\ p=1,\dots,N_p,\ i=1,\dots,d_v\big\},
  \label{eq:particle_frame}
\end{equation}
where $\mathbf{e}_i$ are the unit vectors of $\mathbb{R}^{d_v}$ and
$\chi_{\{\mathbf{x}_p\}}$ is the indicator function of the point
$\mathbf{x}_p$. We assume the positions $\mathbf{x}_p$ to be distinct, which
holds with probability one for random particles and otherwise can be
achieved by merging coincident particles (Lemma~\ref{lem:exact_eval}).

\begin{theorem}[NEOPIC with NUDFTs is a ReNO]
\label{thm:reno}
Let $\mu_X$ be a particle measure~\eqref{eq:charge_measure} with distinct
positions. Then:
\begin{enumerate}
  \item[(i)] $\Psi_X$ is an orthogonal basis of $\mathcal{H}_{\mu_X}$ with frame
  bounds $A=\min_pw_p$ and $B=\max_pw_p$ (a tight frame for equal weights).
  Its pseudo-inverse is the evaluation at the particles,
  $T_{\Psi_X}^\dagger\mathbf{f}=\big(\mathbf{f}(\mathbf{x}_p)\big)_p$, and its
  synthesis operator maps particle values to the function with these values on
  $\mathrm{supp}\,\mu_X$.
  \item[(ii)] For every $\ell$, the discrete map
  $u_\ell=T_{\Psi_X}^\dagger\circ U_\ell\circ T_{\Psi_X}$ is the Fourier
  layer~\eqref{eq:forward}--\eqref{eq:disc_layer} evaluated with NUDFTs, and
  $\varepsilon(U_\ell,u_\ell,\Psi_X,\Psi_X)=0$. The same holds for the output
  map $U_{\mathcal{Q}}$ and its discrete form.
  \item[(iii)] Every pair of frames $(\Psi,\Phi)$ of $\mathcal{H}_{\mu_X}$
  satisfies the conditions of Definition~\ref{def:reno}, and
  $u_\ell(\Psi,\Phi)=T^\dagger_\Phi\circ U_\ell\circ T_\Psi$ has zero aliasing
  error. Hence each layer is a ReNO, and so is the composition
  $U_{\mathcal{Q}}\circ U_{L-1}\circ\dots\circ U_0$, whose discrete
  representation in the particle frame is NEOPIC evaluated with NUDFTs.
\end{enumerate}
\end{theorem}

\begin{proof}
(i) Since $\mu_X$ is supported on the $N_p$ distinct points, two functions are
equal in $\mathcal{H}_{\mu_X}$ if and only if they agree at the particles, so
$\mathcal{H}_{\mu_X}$ has dimension $N_pd_v$ and
$\|\mathbf{f}\|^2=\sum_pw_p|\mathbf{f}(\mathbf{x}_p)|^2$. The elements
of~\eqref{eq:particle_frame} are orthogonal, with
$\|\chi_{\{\mathbf{x}_p\}}\mathbf{e}_i\|^2=w_p$ and
$\langle\mathbf{f},\chi_{\{\mathbf{x}_p\}}\mathbf{e}_i\rangle=w_pf_i(\mathbf{x}_p)$.
Hence
$\sum_{p,i}|\langle\mathbf{f},\chi_{\{\mathbf{x}_p\}}\mathbf{e}_i\rangle|^2=\sum_pw_p^2|\mathbf{f}(\mathbf{x}_p)|^2$
lies between $\min_pw_p\|\mathbf{f}\|^2$ and $\max_pw_p\|\mathbf{f}\|^2$. The
analysis operator is
$T^*_{\Psi_X}\mathbf{f}=\big(w_pf_i(\mathbf{x}_p)\big)_{p,i}$, and the frame
operator $S=T_{\Psi_X}T^*_{\Psi_X}$ multiplies the value at $\mathbf{x}_p$ by
$w_p$, so
    $T_{\Psi_X}^\dagger\mathbf{f}=\big(\langle\mathbf{f},S^{-1}\chi_{\{\mathbf{x}_p\}}\mathbf{e}_i\rangle\big)_{p,i}=\big(f_i(\mathbf{x}_p)\big)_{p,i}$, the vector of the values at the particles, indexed by the
particle $p$ and the channel $i$. Conversely, the synthesis operator maps such
a vector $\mathbf{U}=(U_{p,i})_{p,i}$ to
$T_{\Psi_X}\mathbf{U}=\sum_{p,i}U_{p,i}\,\chi_{\{\mathbf{x}_p\}}\mathbf{e}_i$, the
function with the values $\mathbf{U}_p$ at the particles.
(ii) By (i), $T_{\Psi_X}\mathbf{U}$ is the function with values
$\mathbf{U}_p$ at the particles, and $U_\ell$ applied to it, evaluated at
$\mathbf{x}_p$, is~\eqref{eq:forward}--\eqref{eq:disc_layer}; this is the
induction step of Lemma~\ref{lem:exact_eval}. Since $\Psi_X$ spans
$\mathcal{H}_{\mu_X}$, $T_{\Psi_X}T_{\Psi_X}^\dagger$ is the identity, and
$T_{\Psi_X}\circ u_\ell\circ T_{\Psi_X}^\dagger
=(T_{\Psi_X}T_{\Psi_X}^\dagger)\circ U_\ell\circ(T_{\Psi_X}T_{\Psi_X}^\dagger)=U_\ell$.
The output map is treated in the same way.
(iii) A frame of the finite-dimensional space $\mathcal{H}_{\mu_X}$ spans it,
so $\mathrm{Dom}\,U_\ell=\mathcal{H}_{\mu_X}=\mathcal{M}_\Psi$ and
$\mathrm{Ran}\,U_\ell\subseteq\mathcal{H}_{\mu_X}=\mathcal{M}_\Phi$, and the
computation in (ii) applies with $T_\Psi T_\Psi^\dagger=T_\Phi T_\Phi^\dagger=I$.
The statement for the composition is~\cite[Prop.~3.7]{bartolucci2023representation}.
\end{proof}

The frame computation is elementary; the content of the theorem lies in the
choice of spaces that the nonuniform evaluation permits.   
For the standard grid-based FNO \cite{li2020fourier}, the hidden spaces are spaces of band-limited functions, discretized by
their values on a uniform grid, and the pointwise activation is the
obstruction: $\sigma$ applied to a band-limited function is not band-limited,
so the next DFT folds the new high frequencies onto the retained
ones~\cite[Sec.~4]{bartolucci2023representation} as quantified in
Proposition~\ref{prop:alias}(i). Thus the activation fails to be a ReNO because the space spanned by the grid values 
is smaller than the space into which $\sigma$
maps. With the transforms written as integrals against the input measure, the
space spanned by the particle values is the whole hidden
space~\eqref{eq:hidden_space}, so no pointwise operation can leave it, and
the NUDFT evaluates the integrals~\eqref{eq:err_coef} exactly. This holds for
every number and distribution of the particles and for every activation. For
a grid-based FNO applied to particle data the corresponding choice is not
available, because the grid is not the support of the input: the charge must
first be transferred from $\mu_X$ to the grid, and the integrals of the layers
are then computed against the grid instead of $\mu_X$. Both steps alias
(Proposition~\ref{prop:alias}).

Theorem~\ref{thm:reno} compares discretizations of the same measure, which is
the setting of~\cite{bartolucci2023representation}: two particle data sets
$(X,w)$ and $(X',w')$ with $\mu_X=\mu_{X'}$ give identical outputs, i.e.\ a zero
representation equivalence error~\cite[Def.~3.3]{bartolucci2023representation}.
Two particle sets with different $N_p$ drawn from the same density $\rho$,
however, are different measures. They carry different sampling noise, and
their outputs differ by the sampling error of Lemma~\ref{lem:sampling},
$\mathcal{O}(N_p^{-1/2})$ with a constant that does not depend on any grid.
This difference is a property of the input, i.e.\ of the particle
discretization of the Vlasov equation, and it is shared by every particle
method, including the teacher; the operator itself, $\mathbf{E}_\theta$, is
the same for all $N_p$. No evaluation with finitely many particles can be an
exact representation of $\mathbf{E}_\theta[\rho\,d\mathbf{x}]$ for the measure
of a continuous density, but the resulting error is of a different nature
than grid aliasing (Proposition~\ref{prop:alias} in
\secref{app_alias}). On a grid, the frequencies of the integrand above
the grid resolution are folded onto the retained modes coherently and with
full weight, in every layer of a grid-based FNO, and this error does not
decrease with the number of particles. For particles, they contribute only
through the sampling fluctuations, exactly as for the continuous operator
applied to $\mu_N$: the error is unbiased, decreases as $N_p^{-1/2}$, and is
bounded independently of the smoothness of the integrand.
\begin{remark}[When is the input a measure?]
\label{rem:measure_input}

Lemma~\ref{lem:exact_eval} and Theorem~\ref{thm:reno} hold for any weighted
point measure, for example the nodes of an unstructured mesh with quadrature
weights. Their meaning depends on what the measure represents. In a particle
method the point measure is the state of the scheme. The particles carry their
weights along the characteristics, so that
$f_N(t)=\sum_pw_p\delta_{\mathbf{z}_p(t)}$ is a weak solution of the Vlasov
equation with the field computed from it, which is the classical mean-field
interpretation of particle methods. The field operator, a functional of the
charge measure, is evaluated on this state exactly. How well $f_N$
approximates $f$ is a property of the particle discretization, shared with
every particle method including the teacher. It enters through the initial
sampling, and through remapping steps where these are used, and not in each
evaluation of the operator. For Monte Carlo sampling it is controlled by
Lemma~\ref{lem:sampling} and Theorem~\ref{thm:trajectory}. For structured
initial loadings it is a deterministic quadrature error, which converges
faster but depends on the smoothness of the integrands and may alias
coherently (Proposition~\ref{prop:alias}(i)); Monte Carlo gives the slower
error of Proposition~\ref{prop:alias}(iii), which is unbiased and independent
of the smoothness. The situation is different when the state of the scheme
is not the point measure, for example nodal values of a function on an
Eulerian mesh. The points then only serve to evaluate integrals of a
continuum function, the sums over the points replace integrals that the
scheme defines otherwise, and two discretizations of the same function
differ by a quadrature error in every layer. The nonuniform transforms then
give an $\varepsilon$-ReNO with $\varepsilon$ set by this error. This is the
situation of~\cite{lingsch2023beyond}, where the input functions enter
through their values at the points.

Additional particle attributes keep the input a measure. Charges are
carried by the weights: for several species $\mu_X$ is a signed measure, and
$\sum_p|w_p|$ replaces $\sum_pw_p$ in the bounds. Velocities, or other
per-particle coordinates, enter through a lifting
$\mathcal{P}(\mathbf{x},\mathbf{v})$. The input is then the phase-space
measure
\begin{equation*}
  \nu=\sum_pw_p\,\delta_{(\mathbf{x}_p,\mathbf{v}_p)},
\end{equation*}
the hidden functions are elements of $L^2(\nu)$, and the forward transforms compute
spatial Fourier coefficients of velocity moments, e.g.\ of the current
density for a lifting that is linear in $\mathbf{v}$.
Lemma~\ref{lem:exact_eval} and Theorems~\ref{thm:reno}
and~\ref{thm:eps_reno} hold with $\mu$ replaced by $\nu$. The statements
become approximate only if a feature is itself a discretized approximation
of a continuum function, such as a mesh field interpolated to the particles;
its interpolation error then enters as above.
\end{remark}

\paragraph{NUFFT evaluation} The NUFFT computes~\eqref{eq:forward}
and~\eqref{eq:inverse} by spreading onto an oversampled uniform grid with a
smooth kernel, an FFT, and a division by the Fourier transform of the kernel
(and the reverse for type 2)~\cite{dutt1993fast,potts2001fast,barnett2019parallel}.
The spreading step is a deposition in the sense
of~\eqref{eq:deposit_alias}, and the error of the NUFFT is exactly the
aliasing of the spreading kernel, which the kernel width and the upsampling
factor make smaller than the tolerance~\cite{barnett2021aliasing}. We use
error bounds of the following form, which is that of the available rigorous
estimates~\cite{barnett2021aliasing}.

\paragraph{Assumption}

\begin{enumerate}
  \item[(A0)] \textit{NUFFT accuracy.} For strengths
  $(g_p)_p\in\mathbb{C}^{N_p}$ and coefficients
  $(a_{\mathbf{k}})_{\mathbf{k}}\in\mathbb{C}^M$, the computed type-1 and type-2
  sums satisfy, for all $\mathbf{k}\in\mathcal{K}_m$ and all $p$,
  \begin{equation}
    \Big|\tilde c_{\mathbf{k}}-\sum_qg_qe^{-\mathrm{i}\mathbf{k}\cdot\mathbf{x}_q}\Big|\le\varepsilon\sum_q|g_q|,
    \qquad
    \Big|\tilde s_p-\sum_{\mathbf{k}}a_{\mathbf{k}}e^{\mathrm{i}\mathbf{k}\cdot\mathbf{x}_p}\Big|\le\varepsilon\sum_{\mathbf{k}}|a_{\mathbf{k}}|.
    \label{eq:nufft_accuracy}
  \end{equation}
\end{enumerate}
The tolerance requested from NUFFT libraries like FINUFFT \cite{barnett2019parallel} refers to the relative $\ell^2$ error,
which is of the same order in practice~\cite{barnett2019parallel}. Let
$\tilde u_\ell$ denote the Fourier layer~\eqref{eq:forward}--\eqref{eq:disc_layer}
with both transforms computed by the NUFFT, $\tilde{\mathbf{E}}_\theta[\mu_X]$
the resulting output, and $M_\ell$, $r^\ell_0$ and $\Gamma_\ell$ the constants of
Section~\secref{constants}.

\begin{theorem}[NEOPIC with NUFFTs is an $\varepsilon$-ReNO]
\label{thm:eps_reno}
Assume (A0), and that $\sigma$ and $\mathcal{Q}$ are Lipschitz with constants
$L_\sigma$ and $L_{\mathcal{Q}}$. Let $\mu_X$ be any particle
measure~\eqref{eq:charge_measure}.
\begin{enumerate}
  \item[(i)] For every $\ell$ and every $\mathbf{f}\in\mathcal{H}_{\mu_X}$,
  \begin{equation}
    \begin{aligned}
    \big\|U_\ell\mathbf{f}-T_{\Psi_X}\tilde u_\ell T^\dagger_{\Psi_X}\mathbf{f}\big\|_{\mathcal{H}_{\mu_X}}
    &\le\max_p\big|(U_\ell\mathbf{f})(\mathbf{x}_p)-(\tilde u_\ell T^\dagger_{\Psi_X}\mathbf{f})_p\big|\\
    &\le\varepsilon\,(2+\varepsilon)\,L_\sigma\,r^\ell_0\,\max_p|\mathbf{f}(\mathbf{x}_p)|.
    \end{aligned}
    \label{eq:eps_layer}
  \end{equation}
  \item[(ii)] For the whole network,
  \begin{equation}
    \begin{gathered}
    \max_p\big|\tilde{\mathbf{E}}_\theta[\mu_X](\mathbf{x}_p)-\mathbf{E}_\theta[\mu_X](\mathbf{x}_p)\big|
    \le C_{\varepsilon}\,\varepsilon,\\
    C_{\varepsilon}=2|s|L_{\mathcal{Q}}(2+\varepsilon)\sum_{\ell=0}^{L-1}
      \Big(\prod_{j=\ell+1}^{L-1}\Gamma_j\Big)L_\sigma\,r^\ell_0\,\tilde M_\ell,
    \end{gathered}
    \label{eq:eps_network}
  \end{equation}
  where $\tilde M_0=M_0$ and
  $\tilde M_{\ell+1}=|\sigma(0)|+L_\sigma\big((\|\mathcal{W}_\ell\|+(1+\varepsilon)^2r^\ell_0)\tilde M_\ell+|\mathbf{b}_\ell|\big)$,
  so that $\tilde M_\ell\to M_\ell$ as $\varepsilon\to0$.
\end{enumerate}
The constants do not depend on $N_p$, on the positions or on the weights.
Hence each layer, and the network, is an $\varepsilon'$-ReNO on bounded sets
of $\mathcal{H}_{\mu_X}$ with $\varepsilon'$ proportional to $\varepsilon$,
uniformly over all particle discretizations.
\end{theorem}

\begin{proof}
Throughout, $|\cdot|$ is the Euclidean norm on $\mathbb{C}^{d_v}$, applied to
vectors of channels, and $\|\mathbf{U}\|_\infty=\max_p|\mathbf{U}_p|$. If a
bound of the form~\eqref{eq:nufft_accuracy} holds for each channel, then by
Minkowski's inequality it holds for the vectors of channels, with $|g_q|$ and
$|a_{\mathbf{k}}|$ the Euclidean norms.

(i) The synthesis operator returns the function with the given values at
the particles (Theorem~\ref{thm:reno}(i)), so that
$\big(T_{\Psi_X}\tilde u_\ell T^\dagger_{\Psi_X}\mathbf{f}\big)(\mathbf{x}_p)=(\tilde u_\ell T^\dagger_{\Psi_X}\mathbf{f})_p$,
and in $\mathcal{H}_{\mu_X}$ a function is determined by these values. The
first inequality then holds because
$\|\mathbf{g}\|^2_{\mathcal{H}_{\mu_X}}=\sum_pw_p|\mathbf{g}(\mathbf{x}_p)|^2$
is at most $\max_p|\mathbf{g}(\mathbf{x}_p)|^2$.
Let $\mathbf{U}=T^\dagger_{\Psi_X}\mathbf{f}$ be the particle values. The exact
coefficients~\eqref{eq:forward} satisfy
$|\hat{\mathbf{c}}_{\mathbf{k}}|\le\sum_pw_p|\mathbf{U}_p|\le\|\mathbf{U}\|_\infty$,
and by (A0) the computed ones are
$\tilde{\mathbf{c}}_{\mathbf{k}}=\hat{\mathbf{c}}_{\mathbf{k}}+\boldsymbol\delta_{\mathbf{k}}$
with $|\boldsymbol\delta_{\mathbf{k}}|\le\varepsilon\|\mathbf{U}\|_\infty$. The
type-2 transform of $\mathbf{a}_{\mathbf{k}}=R_\ell(\mathbf{k})\tilde{\mathbf{c}}_{\mathbf{k}}$
is computed with an error $\boldsymbol\eta_p$,
$|\boldsymbol\eta_p|\le\varepsilon\sum_{\mathbf{k}}\|R_\ell(\mathbf{k})\||\tilde{\mathbf{c}}_{\mathbf{k}}|\le\varepsilon(1+\varepsilon)r^\ell_0\|\mathbf{U}\|_\infty$.
The pre-activations of the exact and the computed layer therefore differ at
particle $p$ by
\begin{equation*}
  \Big|\mathrm{Re}\Big(\sum_{\mathbf{k}}R_\ell(\mathbf{k})\boldsymbol\delta_{\mathbf{k}}e^{\mathrm{i}\mathbf{k}\cdot\mathbf{x}_p}+\boldsymbol\eta_p\Big)\Big|
  \le\varepsilon\,r^\ell_0\|\mathbf{U}\|_\infty+\varepsilon(1+\varepsilon)r^\ell_0\|\mathbf{U}\|_\infty
  =\varepsilon(2+\varepsilon)\,r^\ell_0\|\mathbf{U}\|_\infty,
\end{equation*}
and the Lipschitz bound of $\sigma$ gives~\eqref{eq:eps_layer}, since the exact
layer is $U_\ell$ by Theorem~\ref{thm:reno}(ii).

(ii) Let $\mathbf{U}^\ell$ and $\tilde{\mathbf{U}}^\ell$ be the features of the
exact and the computed network, with $\mathbf{U}^0=\tilde{\mathbf{U}}^0$, and
$D_\ell=\|\tilde{\mathbf{U}}^\ell-\mathbf{U}^\ell\|_\infty$. The same estimates as
in (i) show that $\|\tilde{\mathbf{U}}^\ell\|_\infty\le\tilde M_\ell$. The exact
layer is Lipschitz in $\|\cdot\|_\infty$ with constant $\Gamma_\ell$, because
$|\hat{\mathbf{c}}_{\mathbf{k}}(\mathbf{U})-\hat{\mathbf{c}}_{\mathbf{k}}(\mathbf{U}')|\le\|\mathbf{U}-\mathbf{U}'\|_\infty$.
Splitting
$\tilde u_\ell(\tilde{\mathbf{U}}^\ell)-u_\ell(\mathbf{U}^\ell)
=\big(\tilde u_\ell(\tilde{\mathbf{U}}^\ell)-u_\ell(\tilde{\mathbf{U}}^\ell)\big)
+\big(u_\ell(\tilde{\mathbf{U}}^\ell)-u_\ell(\mathbf{U}^\ell)\big)$
and using (i) for the first term,
\begin{equation*}
  D_{\ell+1}\le\Gamma_\ell D_\ell+\varepsilon(2+\varepsilon)L_\sigma r^\ell_0\tilde M_\ell,
  \qquad D_0=0,
\end{equation*}
so that
$D_L\le\varepsilon(2+\varepsilon)\sum_{\ell<L}\big(\prod_{j>\ell}\Gamma_j\big)L_\sigma r^\ell_0\tilde M_\ell$.
The output~\eqref{eq:output} changes by at most $|s|L_{\mathcal{Q}}D_L$ through
the pointwise term and, for $c_{\mathrm{per}}=1$, by at most the same amount through the mean, which
gives~\eqref{eq:eps_network}. The last statement follows from (i) and (ii) with
Definition~\ref{def:reno}, since the frame $\Psi_X$ and the bounds are
available for every particle measure.
\end{proof}

\begin{remark}[What the theorems do and do not say]
\label{rem:reno_scope}
(a) The bound~\eqref{eq:eps_network} does not grow with $N_p$, because the
type-1 error in~\eqref{eq:nufft_accuracy} is relative to the total
$\sum_pw_p|\mathbf{U}_p|$, which the normalization of the weights keeps
bounded. The NUFFT-based network is thus a grid-based Fourier network whose
aliasing is controlled by the tolerance, in contrast to an FNO on a grid
whose aliasing is fixed by the grid.
(b) Representation equivalence concerns the consistency between the operator
and its discrete evaluation. It says nothing about the approximation error
$\varepsilon_{\mathrm{NN}}$ of the trained operator, about the bias of the
teacher, or about the stability of the scheme, which are treated in
Section~\secref{error}. The particle noise
reaches the field of NEOPIC only as it reaches the continuous operator applied
to the particle measure, through the coefficients~\eqref{eq:err_coef}, whose
fluctuations are controlled by Lemma~\ref{lem:sampling} and, in training, by
Section~\secref{noise}. We do not claim that this rules out every numerical heating mechanism: the learned operator is not exactly energy conserving, and
the stability of the scheme is governed by its Lipschitz constant
$L_\theta$ (Theorem~\ref{thm:trajectory}). It removes one of the important reasons for numerical heating which is the aliasing mechanism.
Consistent with this, the energy errors of the mini-apps in
Section~\secref{results} do not grow in the extrapolation windows.
(c) Translation equivariance of the operator, and the role of the lifting
for it, are discussed in \secref{app_equivariance}.
\end{remark}

\section{Error analysis}
\seclab{error}

In this section we bound the error of the NEOPIC scheme with respect to the
exact solution of the Vlasov--Poisson system. The analysis follows the
strategy of~\cite{tranquilli2022deterministic} for PIC schemes (see also~\cite{riva2017methodology} for the corresponding error ansatz): we first bound
the error in the particle trajectories and in the electric field, and then
show that the error in the distribution function is controlled by the error in
the field. Two features of NEOPIC lead to differences from the PIC analysis.
First, the field operator is learned, so an approximation error
$\varepsilon_{\mathrm{NN}}$ appears that does not vanish as $N_p\to\infty$ and
$\Delta t\to0$. Second, the field operator acts on the particles through a
fixed, finite set of Fourier modes, which makes the sampling error
independent of any grid resolution. We also avoid the assumption of
approximately independent particles used in~\cite{tranquilli2022deterministic}
by coupling the numerical particles to an auxiliary set of independent
particles that are transported by the exact field. Further remarks, on the
comparison with PIC, other time integrators, the periodicity of the lifting,
other boundary conditions, the $L^1$ training loss and the scope of the
noise analysis, are collected in \secref{app_error}.

\subsection{Setting and notation}

We consider the Vlasov--Poisson system~\eqref{eq:vlasov}--\eqref{eq:poisson}
without external fields on the periodic domain $\mathbb{T}^d=[-\pi,\pi)^d$,
$d\in\{1,2,3\}$, with a uniform neutralizing background, in normalized form:
$f(t,\cdot,\cdot)$ is a probability density on
$\mathbb{T}^d\times\mathbb{R}^d$ (unit total charge, as in our normalization of
the network inputs), so that $\rho(t,\mathbf{x})=\int_{\mathbb{R}^d}f(t,\mathbf{x},\mathbf{v})\,d\mathbf{v}$
is the normalized charge density, related to $\rho_e$ of
Section~\secref{vp} by $\rho_e=Q_{\mathrm{tot}}\rho$, $q/m$ denotes the
charge-to-mass ratio, and the field is $\mathbf{E}=\mathcal{E}[\rho]$, where
$\mathcal{E}$ is the exact field operator,
\begin{equation}
  \mathcal{E}[\mu](\mathbf{x})
  = \frac{1}{(2\pi)^d}\sum_{\mathbf{k}\in\mathbb{Z}^d\setminus\{0\}}
    \frac{-\mathrm{i}\mathbf{k}}{|\mathbf{k}|^2}\,\hat\mu_{\mathbf{k}}\,
    e^{\mathrm{i}\mathbf{k}\cdot\mathbf{x}},
  \qquad
  \hat\mu_{\mathbf{k}} = \int_{\mathbb{T}^d} e^{-\mathrm{i}\mathbf{k}\cdot\mathbf{y}}\,d\mu(\mathbf{y}),
  \label{eq:exact_field_op}
\end{equation}
defined for any probability measure $\mu$ on $\mathbb{T}^d$, i.e.\ for the
measure $\rho\,d\mathbf{x}$ of a continuous density as well as for empirical
measures of particles. We write $\mathbf{z}=(\mathbf{x},\mathbf{v})$ for points in phase
space; norms are as in Section~\secref{notation}, with
$|\mathbf{x}-\mathbf{y}|$ the periodic distance on $\mathbb{T}^d$.

\paragraph{NEOPIC as an operator on measures} We use the operator
$\mathcal{G}_\theta$ of Section~\secref{architecture},
\eqref{eq:err_lift}--\eqref{eq:err_proj}, which is defined for every
probability measure $\mu$ on $\mathbb{T}^d$. By Lemma~\ref{lem:exact_eval},
its evaluation at the particles for the empirical measure
$\mu_N=\frac1{N_p}\sum_p\delta_{\mathbf{x}_p}$ is the NEOPIC field, and the
same operator applied to the measure $\rho\,d\mathbf{x}$ is the
infinite-particle limit of NEOPIC. For brevity we state the analysis for
$\mathcal{G}_\theta$, without the mean subtraction and the scaling
in~\eqref{eq:output}. Including them changes only the constants: the mean
$\int\mathcal{G}_\theta[\mu]\,d\mu$ is Lipschitz in the particle positions with
constant $2L_\theta$ (in the sense of~\eqref{eq:lip}), and its sampling error
is bounded by $\big(C_S+\sup|\mathcal{G}_\theta[\rho]|\big)N_p^{-1/2}$ where $C_S$ is given by Lemma \ref{lem:sampling}.

\paragraph{The scheme} Particles $\mathbf{z}^n_p=(\mathbf{x}^n_p,\mathbf{v}^n_p)$,
$p=1,\dots,N_p$, are advanced with the kick--drift--kick leapfrog method,
\begin{equation}
  \begin{aligned}
  \mathbf{v}^{n+\frac12}_p &= \mathbf{v}^n_p + \tfrac{\Delta t}{2}\tfrac{q}{m}\mathbf{E}^n_\theta(\mathbf{x}^n_p),
  \qquad
  \mathbf{x}^{n+1}_p = \mathbf{x}^n_p + \Delta t\,\mathbf{v}^{n+\frac12}_p,\\
  \mathbf{v}^{n+1}_p &= \mathbf{v}^{n+\frac12}_p + \tfrac{\Delta t}{2}\tfrac{q}{m}\mathbf{E}^{n+1}_\theta(\mathbf{x}^{n+1}_p),
  \end{aligned}
  \label{eq:kdk}
\end{equation}
with $\mathbf{E}^n_\theta=\mathcal{G}_\theta[\mu^n_N]$ and
$\mu^n_N=\frac1{N_p}\sum_p\delta_{\mathbf{x}^n_p}$, so that one field
evaluation is required per step. Other one-step integrators are discussed in
Remark~\ref{rem:integrators}.

\paragraph{Assumptions}
\begin{enumerate}
  \item[(A1)] \textit{Regularity.} The exact solution $f$ is a classical
  solution on $[0,T]$, the initial datum $f_0$ has compact support in
  $\mathbf{v}$, and $\mathbf{E}\in C^2([0,T]\times\mathbb{T}^d)$. For smooth,
  compactly supported data such solutions exist globally for
  $d\le3$~\cite{batt1991global,pfaffelmoser1992global}. The velocity support
  then stays bounded on $[0,T]$, and we denote by $C_\tau$ the constant in
  the local error of~\eqref{eq:kdk}, which depends on $\|\mathbf{E}\|_{C^2}$
  and on the velocity bound.
  \item[(A2)] \textit{Sampling.} The initial particles are independent and
  identically distributed (iid) samples of $f_0$ (``noisy start''), with
  equal weights. Initial data drawn by importance sampling, with weights
  $w_p=f_0/p^0$, are treated in the same way~\cite{tranquilli2022deterministic}.
  \item[(A3)] \textit{Architecture.} The maps $\mathcal{P}$ and $\mathcal{Q}$
  are Lipschitz with constants $L_{\mathcal{P}}$ and $L_{\mathcal{Q}}$ with
  respect to the periodic distance, the activation $\sigma$ is Lipschitz with
  constant $L_\sigma\ge1$, and the weights are bounded. Since the last
  Fourier layer has no activation (Section~\secref{architecture}), the
  recursions below hold for that layer with $L_\sigma$ replaced by $1$ and
  $|\sigma(0)|$ by $0$; as $L_\sigma\ge1$, the constants as stated remain
  valid upper bounds. The constants $r^\ell_0$, $r^\ell_1$ and $M_\ell$ are those
  of~\eqref{eq:r_constants}--\eqref{eq:M_bound}.
  \item[(A4)] \textit{Model error.} The approximation error of the trained
  operator on the exact solution,
  \begin{equation}
    \varepsilon_{\mathrm{NN}}
    = \sup_{t\in[0,T]}\big\|\mathcal{G}_\theta[\rho(t)] - \mathcal{E}[\rho(t)]\big\|_{L^\infty(\mathbb{T}^d)},
    \label{eq:eps_nn}
  \end{equation}
  is finite.
\end{enumerate}

\subsection{Properties of the learned field operator}

\begin{lemma}[Boundedness and Lipschitz continuity]
\label{lem:lipschitz}
Under (A3), let $M_\ell$ be the constants~\eqref{eq:M_bound}. Then $|\mathbf{u}^\ell_\mu(\mathbf{x})|\le M_\ell$ for every probability
measure $\mu$. Moreover, for any two particle configurations
$X=(\mathbf{x}_p)_p$ and $Y=(\mathbf{y}_p)_p$ with empirical measures $\mu_X$,
$\mu_Y$, and any $\mathbf{x},\mathbf{y}\in\mathbb{T}^d$,
\begin{equation}
  \big|\mathcal{G}_\theta[\mu_X](\mathbf{x})-\mathcal{G}_\theta[\mu_Y](\mathbf{y})\big|
  \le L_\theta\Big(|\mathbf{x}-\mathbf{y}| + \frac1{N_p}\sum_{p=1}^{N_p}|\mathbf{x}_p-\mathbf{y}_p|\Big),
  \label{eq:lip}
\end{equation}
with $L_\theta=L_{\mathcal{Q}}\max(A_L,B_L)$, where $A_0=L_{\mathcal{P}}$,
$B_0=0$ and
\begin{equation}
  \begin{aligned}
  A_{\ell+1}&=L_\sigma\big(\|\mathcal{W}_\ell\|A_\ell + r^\ell_1M_\ell\big),\\
  B_{\ell+1}&=L_\sigma\big(\|\mathcal{W}_\ell\|B_\ell + r^\ell_0(A_\ell+B_\ell) + r^\ell_1M_\ell\big).
  \end{aligned}
  \label{eq:lip_recursion}
\end{equation}
\end{lemma}

\begin{proof}
The bound on $\mathbf{u}^\ell_\mu$ follows by induction, since
$|\hat{\mathbf{c}}^{\,\ell}_\mu(\mathbf{k})|\le M_\ell$ for a probability
measure. For~\eqref{eq:lip}, write
$\delta=\frac1{N_p}\sum_p|\mathbf{x}_p-\mathbf{y}_p|$ and suppose that
$|\mathbf{u}^\ell_{\mu_X}(\mathbf{x})-\mathbf{u}^\ell_{\mu_Y}(\mathbf{y})|\le A_\ell|\mathbf{x}-\mathbf{y}|+B_\ell\delta$
for all $\mathbf{x},\mathbf{y}$, which holds for $\ell=0$. Using
$|e^{-\mathrm{i}\mathbf{k}\cdot\mathbf{x}_p}-e^{-\mathrm{i}\mathbf{k}\cdot\mathbf{y}_p}|\le|\mathbf{k}||\mathbf{x}_p-\mathbf{y}_p|$,
\begin{align*}
  \big|\hat{\mathbf{c}}^{\,\ell}_{\mu_X}(\mathbf{k})-\hat{\mathbf{c}}^{\,\ell}_{\mu_Y}(\mathbf{k})\big|
  &\le \frac1{N_p}\sum_p\Big(\big|\mathbf{u}^\ell_{\mu_X}(\mathbf{x}_p)-\mathbf{u}^\ell_{\mu_Y}(\mathbf{y}_p)\big|
  + M_\ell|\mathbf{k}||\mathbf{x}_p-\mathbf{y}_p|\Big)\\
  &\le (A_\ell+B_\ell+M_\ell|\mathbf{k}|)\,\delta .
\end{align*}
Hence the Fourier term in~\eqref{eq:err_layer} changes by at most
$\sum_{\mathbf{k}}\|R_\ell(\mathbf{k})\|\big(|\hat{\mathbf{c}}_X-\hat{\mathbf{c}}_Y|+M_\ell|\mathbf{k}||\mathbf{x}-\mathbf{y}|\big)
\le r^\ell_0(A_\ell+B_\ell)\delta + r^\ell_1M_\ell(\delta+|\mathbf{x}-\mathbf{y}|)$,
and the pointwise term by at most
$\|\mathcal{W}_\ell\|(A_\ell|\mathbf{x}-\mathbf{y}|+B_\ell\delta)$. Applying
the Lipschitz bound of $\sigma$ gives~\eqref{eq:lip_recursion}, and finally
that of $\mathcal{Q}$ gives~\eqref{eq:lip}.
\end{proof}

\begin{lemma}[Sampling error]
\label{lem:sampling}
Let $\rho$ be a probability density on $\mathbb{T}^d$, let
$\bar{\mathbf{x}}_1,\dots,\bar{\mathbf{x}}_{N_p}$ be iid samples of $\rho$, and
let $\bar\mu_N$ be their empirical measure. Under (A3),
\begin{equation}
  \Big(\mathbb{E}\,\sup_{\mathbf{x}\in\mathbb{T}^d}
  \big|\mathcal{G}_\theta[\bar\mu_N](\mathbf{x})-\mathcal{G}_\theta[\rho](\mathbf{x})\big|^2\Big)^{1/2}
  \le \frac{C_S}{\sqrt{N_p}},
  \label{eq:sampling}
\end{equation}
with $C_S=L_{\mathcal{Q}}\,c_L$, $c_0=0$ and
$c_{\ell+1}=L_\sigma\big((\|\mathcal{W}_\ell\|+r^\ell_0)c_\ell + r^\ell_0M_\ell\big)$.
The constant depends on the architecture (in particular on the number of
modes through $r^\ell_0$) but not on $N_p$ or on any grid resolution, unlike
the sampling error of PIC (Remark~\ref{rem:pic_comparison}).
\end{lemma}

\begin{proof}
Let $D_\ell=\sup_{\mathbf{x}}|\mathbf{u}^\ell_{\bar\mu_N}(\mathbf{x})-\mathbf{u}^\ell_\rho(\mathbf{x})|$,
so that $D_0=0$ because $\mathcal{P}$ does not depend on the measure. We split
\begin{align*}
  \hat{\mathbf{c}}^{\,\ell}_{\bar\mu_N}(\mathbf{k})-\hat{\mathbf{c}}^{\,\ell}_\rho(\mathbf{k})
  &= \frac1{N_p}\sum_p\big(\mathbf{u}^\ell_{\bar\mu_N}-\mathbf{u}^\ell_\rho\big)(\bar{\mathbf{x}}_p)\,e^{-\mathrm{i}\mathbf{k}\cdot\bar{\mathbf{x}}_p}
  + \mathbf{S}_\ell(\mathbf{k}),\\
  \mathbf{S}_\ell(\mathbf{k})&=\frac1{N_p}\sum_p \mathbf{g}_{\mathbf{k}}(\bar{\mathbf{x}}_p)-\mathbb{E}\,\mathbf{g}_{\mathbf{k}}(\bar{\mathbf{x}}_1),
\end{align*}
with $\mathbf{g}_{\mathbf{k}}=\mathbf{u}^\ell_\rho\,e^{-\mathrm{i}\mathbf{k}\cdot\mathbf{x}}$.
The first term is bounded by $D_\ell$. Since $\mathbf{u}^\ell_\rho$ is
deterministic and bounded by $M_\ell$, $\mathbf{S}_\ell(\mathbf{k})$ is a
centred average of iid terms and
$(\mathbb{E}|\mathbf{S}_\ell(\mathbf{k})|^2)^{1/2}\le M_\ell/\sqrt{N_p}$.
From~\eqref{eq:err_layer},
$D_{\ell+1}\le L_\sigma\big((\|\mathcal{W}_\ell\|+r^\ell_0)D_\ell+\sum_{\mathbf{k}}\|R_\ell(\mathbf{k})\||\mathbf{S}_\ell(\mathbf{k})|\big)$,
and Minkowski's inequality in $L^2(\Omega)$ gives
$(\mathbb{E}D_{\ell+1}^2)^{1/2}\le c_{\ell+1}/\sqrt{N_p}$ by induction. The
Lipschitz bound of $\mathcal{Q}$ completes the proof.
\end{proof}

\begin{proposition}[PIF is a special case of NEOPIC]
\label{prop:pif}
Let $\sigma$ satisfy $\sigma(z)-\sigma(-z)=z$ and $\sigma(1)\neq0$, which holds
for ReLU, GELU and SiLU, and let the last Fourier layer be linear
(Section~\secref{architecture}). If $d_v\ge d$ and the hidden width of
$\mathcal{Q}$ is at least $2d$, then for every $L\ge1$ there are parameters
$\theta^\ast$ such that, for every probability measure $\mu$,
\begin{equation}
  \mathcal{G}_{\theta^\ast}[\mu]=\Pi_m\mathcal{E}[\mu],
\end{equation}
where $\Pi_m$ is the Fourier truncation to the set
$\mathcal{K}_m'\subset\mathcal{K}_m$ of modes that are not on the Nyquist
boundary, i.e.\ that have no component equal to $-m$. In particular, applied to particles,
$\mathcal{G}_{\theta^\ast}$ is the PIF field solve with the same modes, and
\begin{equation}
  \inf_\theta\varepsilon_{\mathrm{NN}}(\theta)\le
  \varepsilon_{\mathrm{NN}}(\theta^\ast)
  = \sup_{t\in[0,T]}\big\|(I-\Pi_m)\mathcal{E}[\rho(t)]\big\|_{L^\infty}
  \le \frac{1}{(2\pi)^d}\sup_{t\in[0,T]}\sum_{\mathbf{k}\notin\mathcal{K}_m'}\frac{|\hat\rho_{\mathbf{k}}(t)|}{|\mathbf{k}|}.
  \label{eq:pif_bound}
\end{equation}
\end{proposition}

\begin{proof}
Take $\mathcal{P}\equiv\mathbf{e}_0$ (zero weights, unit bias in channel $0$).
In the layers $\ell<L-1$ set $\mathcal{W}_\ell=0$, $R_\ell=0$ and
$\mathbf{b}_\ell=\mathbf{e}_0$, so that $\mathbf{u}^{L-1}_\mu\equiv c\,\mathbf{e}_0$
with $c=1$ for $L=1$ and $c=\sigma(1)$ otherwise (for activations with
$\sigma(0)\neq0$ the other channels hold constants, which are not used below);
then $\hat{\mathbf{c}}^{\,L-1}_\mu(\mathbf{k})$ has the entry $c\,\hat\mu_{\mathbf{k}}$
in channel $0$. In the last, linear layer set $\mathcal{W}_{L-1}=0$,
$\mathbf{b}_{L-1}=0$, $R_{L-1}(\mathbf{k})=0$ for $\mathbf{k}\notin\mathcal{K}_m'$
and for $\mathbf{k}=0$, and otherwise let $R_{L-1}(\mathbf{k})$ map channel $0$
to channel $j-1$ with entry $c^{-1}(2\pi)^{-d}(-\mathrm{i}k_j/|\mathbf{k}|^2)$,
$j=1,\dots,d$; the output may overwrite channel $0$, so $d_v\ge d$ channels
suffice. Since $\mathcal{K}_m'$ is symmetric under $\mathbf{k}\mapsto-\mathbf{k}$
and the multiplier is Hermitian, the real part is exact and channel $j-1$ of
$\mathbf{u}^L_\mu$ is $z_j=(\Pi_m\mathcal{E}[\mu])_j$.
The first layer of $\mathcal{Q}$ forms the pairs $(z_j,-z_j)$, and its second
layer takes the difference of the activated pair, which is $z_j$ because
$\sigma(z)-\sigma(-z)=z$ (for a linear $\mathcal{Q}$ this step is immediate).
The bound in~\eqref{eq:pif_bound} follows from~\eqref{eq:exact_field_op}.
\end{proof}

Proposition~\ref{prop:pif} shows that NEOPIC is at least as expressive as PIF
with the same modes, and that the best achievable model error is bounded by
the Fourier truncation error, which decays rapidly for smooth densities. This
provides a constructive bound and does not rely on a universal approximation
theorem. For general targets, which may be nonlinear in the measure and
depend on the teacher, no explicit construction is available. Universal
approximation results for neural
operators~\cite{kovachki2021universal,kovachki2023neural,lanthaler2023nonlocal}
are stated for inputs that are functions; for NEOPIC, whose layers integrate
against the input measure, the corresponding statement follows from the
Stone--Weierstrass theorem. For compact $\Omega$, the set
$\mathcal{M}_1(\Omega)$ of probability measures with the weak topology, and
hence $\mathcal{M}_1(\Omega)\times\Omega$, is compact. The functions
$F\big(\int\varphi_1\,d\mu,\dots,\int\varphi_n\,d\mu,\mathbf{x}\big)$ with
continuous $\varphi_i$ and $F$ form an algebra that contains the constants
and separates points, and are therefore dense in
$C(\mathcal{M}_1(\Omega)\times\Omega)$. NEOPIC represents them
approximately: pointwise layers approximate the $\varphi_i$, the mode
$\mathbf{k}=0$ of the next layer forms their integrals against $\mu$, and the
following pointwise layers approximate $F$, as for the averaging neural
operator of~\cite{lanthaler2023nonlocal}. Hence every target that is
continuous in $(\mu,\mathbf{x})$, such as the field operators of PIF and of
PIC with continuous shape functions, can be approximated uniformly on all
probability measures, and in particular on particle measures. The
Barnes--Hut field is not continuous in the particle positions; since
continuous functions are dense in $L^1$, parameters with a small expected
error with respect to the data distribution still exist for such targets.
Neither result guarantees that training finds such parameters, nor does it
control $L_\theta$.

\paragraph{Teacher bias} NEOPIC is trained on data generated by a numerical
method, not on the exact field. Denoting by $\mathcal{E}_h$ the field operator
of the teacher (PIC with mesh size $h$, PIF with a given number of modes,
Barnes--Hut with opening angle $\vartheta$), the triangle inequality gives
\begin{equation}
  \varepsilon_{\mathrm{NN}}
  \le \underbrace{\sup_{t}\big\|\mathcal{G}_\theta[\rho(t)]-\mathcal{E}_h[\rho(t)]\big\|_{L^\infty}}_{\varepsilon_{\mathrm{learn}}}
  + \underbrace{\sup_{t}\big\|\mathcal{E}_h[\rho(t)]-\mathcal{E}[\rho(t)]\big\|_{L^\infty}}_{\varepsilon_{\mathrm{teacher}}},
  \label{eq:eps_split}
\end{equation}
where $\varepsilon_{\mathrm{teacher}}$ is the discretization error of the
teacher (e.g.\ $\mathcal{O}(h^\alpha)$ for PIC with shape functions of order
$\alpha$) and $\varepsilon_{\mathrm{learn}}$ is the error that training
controls. When the model is trained on several teachers,
$\varepsilon_{\mathrm{teacher}}$ is the largest of their biases on the data
the model sees.

\subsection{Convergence of trajectories and fields}

The key device is a set of reference particles. Let
$\bar{\mathbf{z}}_p(t)=(\bar{\mathbf{x}}_p(t),\bar{\mathbf{v}}_p(t))$ be the
exact characteristics of the Vlasov--Poisson system, started from the same initial
samples as the numerical particles,
$\bar{\mathbf{z}}_p(0)=\mathbf{z}^0_p$, and transported by the exact field
$\mathbf{E}(t,\cdot)$. Because this field is a given function and does not
depend on the particles, the reference particles remain independent and, by
Liouville's theorem, $\{\bar{\mathbf{z}}_p(t)\}$ are iid samples of $f(t)$ for
every $t$~\cite[Lemma~2]{tranquilli2022deterministic}. In particular
$\bar{\mathbf{x}}_p(t)$ are iid samples of $\rho(t)$, and
Lemma~\ref{lem:sampling} applies to their empirical measure $\bar\mu_N(t)$.
The numerical particles are compared with the reference particles path by
path, so no independence of the numerical particles is required.

\begin{theorem}[Trajectory error]
\label{thm:trajectory}
Assume (A1)--(A4), and let
\begin{equation*}
  e^n=\|\mathbf{x}^n-\bar{\mathbf{x}}(t_n)\|_N+\|\mathbf{v}^n-\bar{\mathbf{v}}(t_n)\|_N,
  \qquad
  \|\mathbf{a}\|_N=\Big(\frac1{N_p}\sum_p|\mathbf{a}_p|^2\Big)^{1/2}.
\end{equation*}
Let $\Lambda=2\big(1+2|q/m|L_\theta\big)$. Then for
$\Delta t\le\big(1+2|q/m|L_\theta\big)^{-1}$ and $\Delta t\le t_n\le T$,
\begin{equation}
  \mathbb{E}\,e^n
  \le e^{\Lambda t_n}\Big(4|q/m|\,t_n\big(C_S N_p^{-1/2}+\varepsilon_{\mathrm{NN}}\big)
  + C_\tau\,t_n\,\Delta t^2\Big).
  \label{eq:traj_bound}
\end{equation}
\end{theorem}

\begin{proof}
Let $\mathbf{e}^n_{x,p}=\mathbf{x}^n_p-\bar{\mathbf{x}}_p(t_n)$,
$\mathbf{e}^n_{v,p}=\mathbf{v}^n_p-\bar{\mathbf{v}}_p(t_n)$, and let
$\boldsymbol\eta^n_p=\frac{q}{m}\big(\mathbf{E}^n_\theta(\mathbf{x}^n_p)-\mathbf{E}(t_n,\bar{\mathbf{x}}_p(t_n))\big)$
be the error in the acceleration. Inserting the reference field configuration
$\bar\mu_N(t_n)$ and the exact density $\rho(t_n)$,
\begin{align*}
  \frac{|\boldsymbol\eta^n_p|}{|q/m|}
  &\le \big|\mathcal{G}_\theta[\mu^n_N](\mathbf{x}^n_p)-\mathcal{G}_\theta[\bar\mu_N(t_n)](\bar{\mathbf{x}}_p)\big|\\
  &\quad + \big|\mathcal{G}_\theta[\bar\mu_N(t_n)](\bar{\mathbf{x}}_p)-\mathcal{G}_\theta[\rho(t_n)](\bar{\mathbf{x}}_p)\big|\\
  &\quad + \big|\mathcal{G}_\theta[\rho(t_n)](\bar{\mathbf{x}}_p)-\mathcal{E}[\rho(t_n)](\bar{\mathbf{x}}_p)\big|\\
  &\le L_\theta\Big(|\mathbf{e}^n_{x,p}|+\frac1{N_p}\sum_q|\mathbf{e}^n_{x,q}|\Big) + s^n + \varepsilon_{\mathrm{NN}},
\end{align*}
by Lemma~\ref{lem:lipschitz} and (A4), where
$s^n=\sup_{\mathbf{x}}|\mathcal{G}_\theta[\bar\mu_N(t_n)]-\mathcal{G}_\theta[\rho(t_n)]|$.
Hence $\|\boldsymbol\eta^n\|_N\le\kappa\|\mathbf{e}^n_x\|_N+\xi^n$ with
$\kappa=2|q/m|L_\theta$ and $\xi^n=|q/m|(s^n+\varepsilon_{\mathrm{NN}})$.

By (A1), applying one step of~\eqref{eq:kdk} with the exact field to
$\bar{\mathbf{z}}_p(t_n)$ reproduces $\bar{\mathbf{z}}_p(t_{n+1})$ up to a
local error bounded by $C_\tau\Delta t^3$ (the kick--drift--kick method is
second order). Subtracting this from~\eqref{eq:kdk} gives, with
$a^n=\|\mathbf{e}^n_x\|_N$ and $b^n=\|\mathbf{e}^n_v\|_N$,
\begin{align*}
  b^{n+\frac12}&\le b^n+\tfrac{\Delta t}{2}(\kappa a^n+\xi^n)+C_\tau\Delta t^3,\qquad
  a^{n+1}\le a^n+\Delta t\,b^{n+\frac12}+C_\tau\Delta t^3,\\
  b^{n+1}&\le b^{n+\frac12}+\tfrac{\Delta t}{2}(\kappa a^{n+1}+\xi^{n+1})+C_\tau\Delta t^3 .
\end{align*}
Eliminating $b^{n+1/2}$ and $a^{n+1}$ on the right-hand side and using
$\Delta t\le(1+\kappa)^{-1}$ yields
$e^{n+1}\le(1+\Lambda\Delta t)\,e^n+\Delta t(\xi^n+\xi^{n+1})+C_\tau\Delta t^3$,
where the constant multiplying $\Delta t^3$ has been redefined. Since $e^0=0$,
the discrete Gronwall inequality gives, path by path,
\begin{equation*}
  e^n\le e^{\Lambda t_n}\Big(\Delta t\sum_{k=0}^{n}2\xi^k+C_\tau t_n\Delta t^2\Big).
\end{equation*}
Taking expectations, using Lemma~\ref{lem:sampling} at each $t_k$, i.e.\
$\mathbb{E}\,s^k\le(\mathbb{E}(s^k)^2)^{1/2}\le C_SN_p^{-1/2}$, and
$(n+1)\Delta t\le2t_n$ gives~\eqref{eq:traj_bound}. Only the marginal law of the reference particles
at each fixed time enters, so the correlation between different times plays
no role.
\end{proof}

\begin{corollary}[Field error]
\label{cor:field}
Under the assumptions of Theorem~\ref{thm:trajectory},
\begin{equation}
  \mathbb{E}\,\big\|\mathcal{G}_\theta[\mu^n_N]-\mathbf{E}(t_n)\big\|_{L^\infty(\mathbb{T}^d)}
  \le C_1N_p^{-1/2}+C_2\Delta t^2+C_3\,\varepsilon_{\mathrm{NN}},
  \label{eq:field_bound}
\end{equation}
with
\begin{equation*}
  C_1=C_S\big(1+4|q/m|L_\theta t_ne^{\Lambda t_n}\big),\quad
  C_2=L_\theta C_\tau t_ne^{\Lambda t_n},\quad
  C_3=1+4|q/m|L_\theta t_ne^{\Lambda t_n}.
\end{equation*}
\end{corollary}

\begin{proof}
We apply the splitting used for the acceleration error in the proof of
Theorem~\ref{thm:trajectory}, now with both fields evaluated at a common point
$\mathbf{x}\in\mathbb{T}^d$ instead of at $\mathbf{x}^n_p$ and
$\bar{\mathbf{x}}_p$. By Lemma~\ref{lem:lipschitz} with $\mathbf{x}=\mathbf{y}$,
for which only the change of the configuration remains,
Lemma~\ref{lem:sampling} and (A4),
$|\mathcal{G}_\theta[\mu^n_N](\mathbf{x})-\mathbf{E}(t_n,\mathbf{x})|
\le L_\theta\frac1{N_p}\sum_p|\mathbf{e}^n_{x,p}| + s^n + \varepsilon_{\mathrm{NN}}
\le L_\theta e^n+s^n+\varepsilon_{\mathrm{NN}}$,
where the last step uses
$\frac1{N_p}\sum_p|\mathbf{e}^n_{x,p}|\le\|\mathbf{e}^n_x\|_N\le e^n$.
Taking the supremum over $\mathbf{x}$, then expectations, and
using~\eqref{eq:traj_bound} proves the claim.
\end{proof}

\subsection{Convergence of the distribution function}

The particles approximate $f$ in the sense of measures. The following result
covers moments and quantities of interest computed from the particles, the
phase-space empirical measure, and, following~\cite{tranquilli2022deterministic},
the continuous distribution function transported by the NEOPIC field.

\begin{theorem}[Distribution function]
\label{thm:distribution}
Under the assumptions of Theorem~\ref{thm:trajectory}, let
$\nu^n_N=\frac1{N_p}\sum_p\delta_{\mathbf{z}^n_p}$ be the phase-space empirical
measure.
\begin{enumerate}
  \item[(i)] For every bounded, Lipschitz observable $\varphi$ on phase space,
  \begin{equation}
    \mathbb{E}\,\Big|\frac1{N_p}\sum_p\varphi(\mathbf{z}^n_p)-\int\varphi\,f(t_n)\,d\mathbf{x}\,d\mathbf{v}\Big|
    \le \mathrm{Lip}(\varphi)\,\mathbb{E}\,e^n + \|\varphi\|_{L^\infty}N_p^{-1/2}.
    \label{eq:observable}
  \end{equation}
  \item[(ii)] In the Wasserstein-1 distance,
  $\mathbb{E}\,W_1\big(\nu^n_N,f(t_n)\big)\le\mathbb{E}\,e^n+C_W\,r_{2d}(N_p)$,
  where $r_{2d}(N)=N^{-1/(2d)}$ for $d\ge2$ and $r_2(N)=N^{-1/2}\log(1+N)$
  for $d=1$~\cite{fournier2015rate}.
  \item[(iii)] Let $\tilde{\mathbf{E}}(t,\mathbf{x})$ be the piecewise linear
  interpolant in time of $\mathcal{G}_\theta[\mu^n_N](\mathbf{x})$ and $\tilde f$
  the solution of the Vlasov equation~\eqref{eq:vlasov} with the field
  $\tilde{\mathbf{E}}$ instead of $\mathbf{E}$ and $\tilde f(0)=f_0$. Then
  \begin{equation}
    \begin{aligned}
    \mathbb{E}\,\|f(t_n)-\tilde f(t_n)\|_{L^2}
    &\le |q/m|\,K_f\,(2\pi)^{d/2}\,t_n\\
    &\quad\times\Big(C_1N_p^{-1/2}+(C_2+C_E)\Delta t^2+C_3\,\varepsilon_{\mathrm{NN}}\Big),
    \end{aligned}
    \label{eq:f_bound}
  \end{equation}
  with $K_f=\sup_{t\le T,\,\mathbf{x}}\|\nabla_{\mathbf{v}}f(t,\mathbf{x},\cdot)\|_{L^2(\mathbb{R}^d)}$
  and $C_E$ the constant of linear interpolation of $\mathbf{E}$ in time.
\end{enumerate}
\end{theorem}

\begin{proof}
(i) Split the error at the reference particles:
\begin{equation*}
  \Big|\frac1{N_p}\sum_p\big(\varphi(\mathbf{z}^n_p)-\varphi(\bar{\mathbf{z}}_p(t_n))\big)\Big|
  \le\mathrm{Lip}(\varphi)\frac1{N_p}\sum_p|\mathbf{z}^n_p-\bar{\mathbf{z}}_p(t_n)|\le\mathrm{Lip}(\varphi)\,e^n,
\end{equation*}
and the remaining term is a centred average of iid variables bounded by
$\|\varphi\|_{L^\infty}$, whose mean absolute value is at most
$\|\varphi\|_{L^\infty}N_p^{-1/2}$.
(ii) By the same coupling,
$W_1(\nu^n_N,\bar\nu_N(t_n))\le\frac1{N_p}\sum_p|\mathbf{z}^n_p-\bar{\mathbf{z}}_p(t_n)|\le e^n$,
and $\mathbb{E}\,W_1(\bar\nu_N(t_n),f(t_n))\le C_Wr_{2d}(N_p)$ for iid samples
of a compactly supported density in dimension $2d$~\cite{fournier2015rate}.
(iii) As in~\cite[Lemma~1]{tranquilli2022deterministic}, the difference
$\delta f=f-\tilde f$ satisfies
\begin{equation*}
  \partial_t\delta f+\mathbf{v}\cdot\nabla_{\mathbf{x}}\delta f+\tfrac{q}{m}\tilde{\mathbf{E}}\cdot\nabla_{\mathbf{v}}\delta f
  =-\tfrac{q}{m}(\mathbf{E}-\tilde{\mathbf{E}})\cdot\nabla_{\mathbf{v}}f .
\end{equation*}
The transport operator is skew-adjoint in $L^2$ on the torus, so
$\frac{d}{dt}\|\delta f\|_{L^2}\le|q/m|\,K_f\,\|\mathbf{E}-\tilde{\mathbf{E}}\|_{L^2(\mathbb{T}^d)}$
and, integrating in time,
$\|\delta f(t_n)\|_{L^2}\le|q/m|K_f\int_0^{t_n}\|\mathbf{E}-\tilde{\mathbf{E}}\|_{L^2}\,ds$.
On each interval, at $s=t_k+\theta\Delta t$ with $\theta\in[0,1]$, both
$\tilde{\mathbf{E}}$ and the linear interpolant of $\mathbf{E}$ in time are
convex combinations of their values at $t_k$ and $t_{k+1}$ with the same
weights, so that
$\|\mathbf{E}-\tilde{\mathbf{E}}\|_{L^2}\le(2\pi)^{d/2}\big((1-\theta)\|X_k\|_{L^\infty}+\theta\|X_{k+1}\|_{L^\infty}+C_E\Delta t^2\big)$
with $X_j=\mathbf{E}(t_j)-\mathcal{G}_\theta[\mu^j_N]$. Taking expectations,
which commute with the convex combination and with the time integral, and
using Corollary~\ref{cor:field}, whose constants are nondecreasing in $t_n$,
bounds the expectation of the integrand by the same quantity for all
$s\le t_n$; integrating over $[0,t_n]$ gives~\eqref{eq:f_bound}.
\end{proof}

Theorem~\ref{thm:distribution}(iii) is the analogue of the statement
of~\cite{tranquilli2022deterministic} that convergence of the distribution
function is no slower than convergence of the electric field. As there, the
particles follow the characteristics of $\tilde f$ but are not independent
samples of it, because $\tilde{\mathbf{E}}$ depends on the particles;
statements (i) and (ii), which concern the particles themselves, are rigorous
without this caveat.

\begin{remark}[Interpretation]
\label{rem:interpretation}
Corollary~\ref{cor:field} refines the heuristic decomposition
\textit{field error $\lesssim$ sampling + time stepping + model error} in three
ways. (a) The model error is amplified by the self-consistent dynamics through
$C_3$, which grows with $t_ne^{\Lambda t_n}$. (b) The growth rate $\Lambda$
depends on the Lipschitz constant $L_\theta$ of the learned operator
with respect to the particle positions, not on that of the exact field
operator. Controlling $L_\theta$, for example through spectral normalization
of the weights or a penalty on the weights of high modes, therefore directly
improves the stability of NEOPIC. (c) The error does not vanish as
$N_p\to\infty$ and $\Delta t\to0$; it saturates at $C_3\varepsilon_{\mathrm{NN}}$.
NEOPIC is convergent up to the model error, and by~\eqref{eq:eps_split} this
floor contains the bias of the teacher.
\end{remark}

\begin{remark}[Teacher-forced and in-loop errors]
\label{rem:teacher_forced}
If the field is evaluated on the reference particles, i.e.\ with
$\mathbf{x}^n_p=\bar{\mathbf{x}}_p(t_n)$ (teacher forcing), then $e^n=0$ and the
proof of Corollary~\ref{cor:field} gives
$\mathbb{E}\|\mathcal{G}_\theta[\bar\mu_N(t_n)]-\mathbf{E}(t_n)\|_{L^\infty}\le C_SN_p^{-1/2}+\varepsilon_{\mathrm{NN}}$,
without amplification. The difference between the teacher-forced and the
in-loop errors observed in the numerical experiments is thus the accumulation
captured by Theorem~\ref{thm:trajectory}.
\end{remark}

\begin{remark}[Learning and teacher errors at finite particle number]
\label{rem:resolution_mismatch}

In the proof of Theorem~\ref{thm:trajectory}, the acceleration error can also
be split through the teacher,
$\mathcal{G}_\theta[\bar\mu_{N_p}]-\mathcal{E}[\rho]
=(\mathcal{G}_\theta[\bar\mu_{N_p}]-\mathcal{E}_h[\bar\mu_{N_p}])
+(\mathcal{E}_h[\bar\mu_{N_p}]-\mathcal{E}[\rho])$,
instead of through $\mathcal{G}_\theta[\rho]$. The term
$C_SN_p^{-1/2}+\varepsilon_{\mathrm{NN}}$ in~\eqref{eq:traj_bound} and in
Corollary~\ref{cor:field} is then replaced, with the same constants, by
$\varepsilon_{\mathrm{learn}}(N_{\mathrm{tr}};N_p)+\varepsilon_{\mathrm{teacher}}(N_p)$,
the suprema over $t_n\le T$ of the expected $L^\infty$ norms of the two
differences. The first is the teacher-forced error at $N_p$ particles of a
model trained with $N_{\mathrm{tr}}$ particles, and the second is the error of
the teacher itself with $N_p$ particles. On noise-free inputs they reduce to
the quantities in~\eqref{eq:eps_split}. We state
Theorem~\ref{thm:trajectory} with the first splitting because it requires no
sampling estimate for the teacher, which is not available for a tree code.
The second splitting can be much sharper: if the network responds to sampling
noise like the teacher (small $Q$ in Proposition~\ref{prop:noise_reg}), the
noise cancels in $\mathcal{G}_\theta[\bar\mu_{N_p}]-\mathcal{E}_h[\bar\mu_{N_p}]$.
It also matches the experiments. Against the teacher started from the same
particles, $\varepsilon_{\mathrm{teacher}}$ drops out
(Proposition~\ref{prop:inloop_teacher}). Against a high-fidelity reference,
the in-loop error contains $\varepsilon_{\mathrm{teacher}}(N_p)$, amplified by
the Gronwall factor; its sampling part is of order $N_p^{-1/2}$ and is the
same for all models.

The dependence of $\varepsilon_{\mathrm{learn}}(N_{\mathrm{tr}};N_p)$ on both
particle numbers differs from the resolution dependence of grid-based neural
operators. By Theorem~\ref{thm:reno}, the operator is the same at every $N_p$,
but its accuracy need not be. The Fourier coefficients that enter the first
layer are
$\hat{\mathbf{c}}^{\,0}_{\bar\mu_{N_p}}(\mathbf{k})=\hat{\mathbf{c}}^{\,0}_{\rho}(\mathbf{k})+\boldsymbol{\eta}_{N_p}(\mathbf{k})$,
with sampling fluctuations $\boldsymbol{\eta}_{N_p}$ of variance
$\mathcal{O}(N_p^{-1})$. A change of $N_p$ therefore changes the noise level
of the input, i.e.\ it shifts the input distribution, and nothing in the
training guarantees accuracy outside the training distribution. For
$N_p<N_{\mathrm{tr}}$ the inputs are noisier than any training sample, and
the deviation is bounded only through the Lipschitz estimate~\eqref{eq:lip},
with a constant $L_\theta$ that training does not control. The effect of the
training noise level on both sides of $N_{\mathrm{tr}}$ is analysed in
Section~\ref{sect:noise}.
\end{remark}

\paragraph{Error norms}
The experiments report particle-averaged relative $L^1$ errors. For any
particle quantity $\mathbf{a}=(\mathbf{a}_p)_p$,
\begin{equation}
  \frac1{N_p}\sum_{p=1}^{N_p}|\mathbf{a}_p|
  \le\Big(\frac1{N_p}\sum_{p=1}^{N_p}|\mathbf{a}_p|^2\Big)^{1/2}
  \le\max_p|\mathbf{a}_p|,
  \label{eq:norm_chain}
\end{equation}
so that the bounds of Theorem~\ref{thm:trajectory} and Corollary~\ref{cor:field} hold
for these errors with the same rates, as long as the reference norm used for
the normalization is bounded away from zero. Measured against the teacher, a
teacher-forced error is in this norm a lower estimate of the $L^\infty$
learning error $\varepsilon_{\mathrm{learn}}$ of~\eqref{eq:eps_split} at the
given particle number.
\paragraph{Sharpness of the constants}
The constants are explicit but pessimistic: $e^{\Lambda t_n}$ is the
worst-case Gronwall growth, which does not reflect, e.g., the phase mixing
that saturates errors in Landau damping, and $K_f$ grows with time by
filamentation, so that statements (i)--(ii) of
Theorem~\ref{thm:distribution} are the more natural measures of convergence
for the distribution function. The bounds establish the structure and the
rates of the error, which the experiments confirm.
\section{Sampling noise as an implicit regularizer}
\seclab{noise}
Adding noise to the inputs during training is a well-known
regularizer~\cite{bishop1995training,an1996effects}. In operator learning for
dynamical systems it improves long-time stability, and the noise level must be
tuned to the problem~\cite{nguyen2022mctangent,nguyen2025dgnet}. Neural
operators on grids are typically trained on deterministic solver data, so
noise has to be added on purpose. NEOPIC differs in three ways. First, its
inputs are empirical measures of particles, so sampling noise is present by
construction. Second, the noise level is not a free parameter but is set by
the number of training particles $N_{\mathrm{tr}}$, and its covariance is set
by the physics. Third, the labels are computed by the teacher from the same
noisy particles. This section makes these points quantitative and draws
consequences for the choice of $N_{\mathrm{tr}}$.

\paragraph{Training loss} Let $\mathcal{E}_h$ be the field operator of the
teacher (see~\eqref{eq:eps_split}) and $r_\theta[\mu]=\mathcal{G}_\theta[\mu]-\mathcal{E}_h[\mu]$
the residual. The training data are snapshots $\mu^{h,j}_N$, $j\in\mathcal{J}$,
of teacher runs with $N$ particles, labelled with the teacher field at the
particles. We consider the particle-averaged squared error
\begin{equation}
  J(\theta;\mu)=\int_{\mathbb{T}^d}\big|r_\theta[\mu](\mathbf{x})\big|^2\,d\mu(\mathbf{x}),
  \qquad
  \mathcal{L}_N(\theta)=\frac1{|\mathcal{J}|}\sum_{j\in\mathcal{J}}J\big(\theta;\mu^{h,j}_N\big),
  \label{eq:train_loss}
\end{equation}
and the population loss $\bar{\mathcal{L}}_N(\theta)=\mathbb{E}\,\mathcal{L}_N(\theta)$,
where the expectation is over the particle realizations. Other weightings
of the loss change only the term $P$ below.

The common teachers are linear in the charge distribution: the field is
\begin{equation}
  \mathcal{E}_h[\mu](\mathbf{x})=\int_{\mathbb{T}^d}\mathbf{K}_h(\mathbf{x},\mathbf{y})\,d\mu(\mathbf{y}).
  \label{eq:linear_teacher}
\end{equation}
For PIF, $\mathbf{K}_h$ is the truncated Fourier sum~\eqref{eq:exact_field_op}.
For PIC, $\mathbf{K}_h$ composes charge deposition, the discrete Poisson solve
and field interpolation. Barnes--Hut is linear for a fixed tree. We assume
that $\mathbf{K}_h$ is bounded, which holds for PIF, for PIC, and for
softened tree codes.

\paragraph{Assumptions}
\begin{enumerate}
  \item[(A5)] \textit{Fluctuations.} For each snapshot $j$, the particle
  positions are iid samples of a density $\rho_j$, the limit of
  $\mu^{h,j}_N$ as $N\to\infty$.
  \item[(A6)] \textit{Smoothness.} For every $\mathbf{x}$, the map
  $\mu\mapsto\mathcal{G}_\theta[\mu](\mathbf{x})$ has bounded linear functional
  derivatives up to third order~\cite{chassagneux2022weak}. This holds for
  smooth activations such as GELU, because by~\eqref{eq:err_layer}--\eqref{eq:err_coef}
  the measure enters only through integrals of smooth functions.
\end{enumerate}
We write $D\mathcal{G}_\theta[\rho]\nu$ and $D^2\mathcal{G}_\theta[\rho](\nu,\nu)$
for the first and second derivatives in the direction of a signed measure
$\nu$, and $Dr_\theta[\rho]=D\mathcal{G}_\theta[\rho]-\mathcal{E}_h$; the
teacher contributes $\mathcal{E}_h$ itself because it is linear. Finally,
$\boldsymbol\xi=\delta_X-\rho$, where $X\sim\rho$ is a random particle
position drawn from $\rho$ and $\delta_X$ the Dirac measure at it, denotes the fluctuation of a
single particle, so that $\mu_N-\rho=\frac1N\sum_p\boldsymbol\xi_p$.
Assumption (A5) is exact at $t=0$ and for the reference particles of
Section~\ref{sect:error}. For the interacting particles of a teacher run it
holds approximately (propagation of chaos), and we discuss the correction
in Remark~\ref{rem:noise_caveats}.

\begin{proposition}[Noise-induced regularization]
\label{prop:noise_reg}
Under (A5)--(A6),
\begin{equation}
  \bar{\mathcal{L}}_N(\theta)=\mathcal{L}_\infty(\theta)
  +\frac1N\,\frac1{|\mathcal{J}|}\sum_{j\in\mathcal{J}}\big(Q(\theta;\rho_j)+P(\theta;\rho_j)\big)
  +\mathcal{O}(N^{-2}),
  \label{eq:noise_expansion}
\end{equation}
where $\mathcal{L}_\infty(\theta)=\frac1{|\mathcal{J}|}\sum_jJ(\theta;\rho_j)$
is the loss on noise-free inputs,
\begin{equation}
  Q(\theta;\rho)=\mathbb{E}\int_{\mathbb{T}^d}\big|Dr_\theta[\rho]\,\boldsymbol\xi\big|^2\,d\rho\;\ge0,
  \label{eq:Q}
\end{equation}
and
\begin{equation}
  \begin{aligned}
  P(\theta;\rho)&=2\,\mathbb{E}\int_{\mathbb{T}^d} r_\theta[\rho]\cdot\big(Dr_\theta[\rho]\,\boldsymbol\xi\big)\,d\boldsymbol\xi\\
  &\quad+\mathbb{E}\int_{\mathbb{T}^d} r_\theta[\rho]\cdot D^2\mathcal{G}_\theta[\rho](\boldsymbol\xi,\boldsymbol\xi)\,d\rho
  \end{aligned}
  \label{eq:P}
\end{equation}
is linear in the residual on noise-free inputs, so that
$|P(\theta;\rho)|\le C_P\|r_\theta[\rho]\|_{L^\infty}$.
\end{proposition}

\begin{proof}
Fix $j$ and write $F(\mu)=J(\theta;\mu)$ and $\rho=\rho_j$. For iid samples
and $F$ regular in the sense of (A6), the expansion of~\cite{chassagneux2022weak}
gives
$\mathbb{E}\,F(\mu_N)=F(\rho)+\frac1{2N}\,\mathbb{E}\,D^2F[\rho](\boldsymbol\xi,\boldsymbol\xi)+\mathcal{O}(N^{-2})$.
There is no first-order term because $\mathbb{E}\,\boldsymbol\xi=0$, and the
cross terms $p\ne q$ vanish by independence. Differentiating
$s\mapsto F(\rho+s\nu)=\int|r_\theta[\rho+s\nu]|^2\,d(\rho+s\nu)$ twice at
$s=0$, and using $D^2\mathcal{E}_h=0$, gives
\begin{equation*}
  \begin{aligned}
  \tfrac12D^2F[\rho](\nu,\nu)
  &=\int|Dr_\theta[\rho]\nu|^2\,d\rho
  +\int r_\theta[\rho]\cdot D^2\mathcal{G}_\theta[\rho](\nu,\nu)\,d\rho\\
  &\quad+2\int r_\theta[\rho]\cdot(Dr_\theta[\rho]\nu)\,d\nu .
  \end{aligned}
\end{equation*}
Setting $\nu=\boldsymbol\xi$, taking expectations and averaging over $j$
proves~\eqref{eq:noise_expansion}. The bound on $P$ follows from the
boundedness of $\mathbf{K}_h$ and of the derivatives in (A6).
\end{proof}

The proposition is stated for the squared loss; the models of
Section~\secref{results} are trained with the $L^1$ loss, for which the same
conclusions hold in the regime discussed in Remark~\ref{rem:l1_loss}.
Proposition~\ref{prop:noise_reg} is also the particle analogue of the
randomized-loss expansions of~\cite[Theorem~2.6]{nguyen2022mctangent}
and~\cite[Theorem~1]{nguyen2025dgnet}. The added noise variance $\delta^2$ is
replaced by $1/N_{\mathrm{tr}}$, and the isotropic noise is replaced by the
fluctuation $\boldsymbol\xi$, whose covariance is fixed by the density:
$\mathrm{Cov}\big(\varphi(X),\psi(X)\big)$ for test functions $\varphi$ and
$\psi$. For a Fourier mode, $\mathrm{Var}\,e^{-\mathrm{i}\mathbf{k}\cdot X}=1-|\hat\rho_{\mathbf{k}}|^2$.
There is one structural difference. In~\cite{nguyen2022mctangent} only the
inputs are randomized, and the data-misfit term produces a penalty on the
network Jacobian itself, which pushes the network towards insensitivity. In
NEOPIC the label $\mathcal{E}_h[\mu_N]$ is computed from the same noisy
particles. Only the mismatch $D\mathcal{G}_\theta-\mathcal{E}_h$ is penalized,
i.e.\ the regularizer has the form of the model-constrained term
of~\cite{nguyen2022mctangent}. Sampling noise thus drives the linear response
of the network to density fluctuations towards that of the teacher, rather
than towards zero. In particular, if the teacher is realizable,
$r_\theta\equiv0$ for some $\theta$ (for example a PIF teacher,
Proposition~\ref{prop:pif}), then $Q=P=0$ there and the noise does not bias
the minimizer.

\begin{proposition}[Structure of the regularizer]
\label{prop:noise_structure}
Assume (A5)--(A6) and let $\rho$ be a snapshot density.
\begin{enumerate}
  \item[(i)] \textit{Matching of linear responses.} For every $g\in L^2(\rho)$
  and $\nu_g=(g-\bar g)\rho$ with $\bar g=\int g\,d\rho$,
  \begin{equation}
    \big\|Dr_\theta[\rho]\,\nu_g\big\|_{L^2(\rho)}\le Q(\theta;\rho)^{1/2}\,\|g-\bar g\|_{L^2(\rho)} .
    \label{eq:jacobian_match}
  \end{equation}
  \item[(ii)] \textit{Irreducible part.} Let $\lambda_1\ge\lambda_2\ge\dots$
  be the eigenvalues of the covariance operator of the random field
  $\mathcal{E}_h\boldsymbol\xi$ in $L^2(\rho;\mathbb{R}^d)$, and
  $D_\theta=2L\,|\mathcal{K}_m|\,d_v$. Then, for every $\theta$,
  \begin{equation}
    Q(\theta;\rho)\ge Q_0(\rho):=\sum_{i>D_\theta}\lambda_i .
    \label{eq:Q0}
  \end{equation}
\end{enumerate}
\end{proposition}

\begin{proof}
(i) Let $A=Dr_\theta[\rho]$. For every test function $\varphi$,
$\mathbb{E}\big[(g(X)-\bar g)\int\varphi\,d\boldsymbol\xi\big]=\int\varphi\,(g-\bar g)\,d\rho$,
so $\nu_g=\mathbb{E}[(g(X)-\bar g)\boldsymbol\xi]$. By linearity and the
Cauchy--Schwarz inequality,
$\|A\nu_g\|\le\mathbb{E}\big[|g(X)-\bar g|\,\|A\boldsymbol\xi\|\big]\le\|g-\bar g\|_{L^2(\rho)}\,Q^{1/2}$.
(ii) By~\eqref{eq:err_layer}--\eqref{eq:err_coef}, a perturbation of the
measure enters $\mathcal{G}_\theta$ only through the coefficients
$\hat{\mathbf{c}}^{\,\ell}(\mathbf{k})\in\mathbb{C}^{d_v}$,
$\ell<L$, $\mathbf{k}\in\mathcal{K}_m$, and $\mathbf{u}^0=\mathcal{P}$ does not
depend on it. Hence $D\mathcal{G}_\theta[\rho]\boldsymbol\xi$ lies in a fixed
subspace $V\subset L^2(\rho;\mathbb{R}^d)$ of dimension at most $D_\theta$,
and $Q\ge\mathbb{E}\,\mathrm{dist}(\mathcal{E}_h\boldsymbol\xi,V)^2$, with
$\mathrm{dist}$ the distance in $L^2(\rho;\mathbb{R}^d)$ as defined in Section \secref{notation}. Since
$\mathcal{E}_h\boldsymbol\xi$ has mean zero and finite second moment
($\mathbf{K}_h$ is bounded), Ky Fan's maximum principle gives
$\mathbb{E}\,\mathrm{dist}(\mathcal{E}_h\boldsymbol\xi,V)^2\ge\sum_{i>D_\theta}\lambda_i$.
\end{proof}

Part (i) states that a small $Q$ forces the network and the teacher to
respond alike to every density perturbation of the form $\nu_g$. Displacing
the particles by a smooth field $\mathbf{e}$ changes the density by
$-\nabla\cdot(\rho\mathbf{e})$ to first order, which is of this form with
$g=-\nabla\cdot(\rho\mathbf{e})/\rho$. This is the perturbation produced by
errors accumulated in the loop. Part (ii) bounds the part of $Q$ that no network of the given size can remove:
the variance of the teacher's fluctuation field outside the best subspace of
dimension $D_\theta$. The response of NEOPIC is not band-limited, since the
activations and skip connections generate high wave numbers, but it depends on
the fluctuation only through the $D_\theta$ real and imaginary parts of the
coefficients $\hat{\mathbf{c}}^{\,\ell}(\mathbf{k})$, $\ell<L$,
$\mathbf{k}\in\mathcal{K}_m$. Two regimes follow. If $D_\theta$ is smaller
than the rank of the teacher's response, $Q_0>0$. For a nearly homogeneous
plasma and a Poisson-type teacher with shape factor $\hat S_h$, the
eigenfunctions are approximately Fourier modes with
$\lambda\propto|\hat S_h(\mathbf{k})|^2/|\mathbf{k}|^2$, and $Q_0$ is the sum
of the eigenvalues beyond the $D_\theta$ largest; it depends on the resolution
of the teacher. If $D_\theta$ exceeds the rank of the teacher's response, which
for PIC and PIF is at most the number of grid points or modes (for a softened
tree code it is infinite), then $Q_0=0$ and part (ii) gives no lower bound. The
value $Q(\theta)$ reached in training is then set by the response functions
that the network can produce and, when several teachers are trained together,
by the fact that a single operator has only one linear response at a given
density. In either case, by~\eqref{eq:noise_expansion}, the expected
teacher-forced loss of a trained model at $N_p$ particles exceeds
$\mathcal{L}_\infty$ by about $(Q+P)(\theta)/N_p$, so that, as long as
$Q(\theta)>0$, the error measured against the teacher on the same particles
decreases like $N_p^{-1/2}$. This error reflects the teacher's particle noise,
which the NEOPIC field does not reproduce. It is not an error with respect to
the exact solution.

\begin{proposition}[Accuracy--robustness trade-off]
\label{prop:tradeoff}
Let
\begin{equation*}
  R(\theta)=\frac1{|\mathcal{J}|}\sum_{j\in\mathcal{J}}\big(Q+P\big)(\theta;\rho_j),
  \qquad
  \Phi_\beta(\theta)=\mathcal{L}_\infty(\theta)+\beta R(\theta),
\end{equation*}
so that $\bar{\mathcal{L}}_N=\Phi_{1/N}+\mathcal{O}(N^{-2})$, and let
$\theta_\beta$ be a minimizer of $\Phi_\beta$. This idealizes training with
$N_{\mathrm{tr}}=1/\beta$ particles.
\begin{enumerate}
  \item[(i)] If $\beta_1<\beta_2$, then
  $\mathcal{L}_\infty(\theta_{\beta_1})\le\mathcal{L}_\infty(\theta_{\beta_2})$
  and $R(\theta_{\beta_1})\ge R(\theta_{\beta_2})$. Training with more
  particles gives a smaller error on noise-free inputs and a larger
  sensitivity to sampling noise.
  \item[(ii)] The expected teacher-forced loss of $\theta_\beta$ at
  inference with $N_p$ particles is $\Phi_{1/N_p}(\theta_\beta)$ up to
  $\mathcal{O}(N_p^{-2})$. Over $\beta$, it is smallest for $\beta=1/N_p$.
  \item[(iii)] If inference uses particle numbers $N_i$ with frequencies $\pi_i\ge0$,
  $\sum_i\pi_i=1$, the
  averaged expected loss is $\Phi_{\bar\beta}$ with $\bar\beta=\sum_i\pi_i/N_i$.
  Among single training particle numbers, it is minimized by the harmonic
  mean $N_H=\bar\beta^{-1}$ of the inference particle numbers. If
  $R(\theta)\ge0$ and inference must cover all $N_p\ge N_{\min}$, the
  worst-case expected loss is $\Phi_{1/N_{\min}}$, which is minimized by
  training with $N_{\min}$ particles.
\end{enumerate}
\end{proposition}

\begin{proof}
(i) Optimality gives
$\Phi_{\beta_1}(\theta_{\beta_1})\le\Phi_{\beta_1}(\theta_{\beta_2})$
and $\Phi_{\beta_2}(\theta_{\beta_2})\le\Phi_{\beta_2}(\theta_{\beta_1})$.
Adding the two inequalities gives
$(\beta_1-\beta_2)\big(R(\theta_{\beta_1})-R(\theta_{\beta_2})\big)\le0$,
hence $R(\theta_{\beta_1})\ge R(\theta_{\beta_2})$, and then the first
inequality gives the statement for $\mathcal{L}_\infty$. (ii) follows
from~\eqref{eq:noise_expansion} with $N=N_p$ and the definition of
$\theta_{1/N_p}$. (iii) $\Phi_\beta$ is affine in $\beta$, and for
$R\ge0$ it is nondecreasing in $\beta$.
\end{proof}

The proposition explains the off-diagonal behaviour of the error-scaling
experiments. A model trained with many particles has a large $R$, and its
loss $\mathcal{L}_\infty+R/N_p$ grows when it is evaluated with fewer
particles. A model trained with few particles has a small $R$; its loss is
nearly flat in $N_p$, but at the higher level $\mathcal{L}_\infty(\theta_\beta)$.
Along the diagonal, the loss is $\min_\theta\Phi_{1/N}$, which decreases with
$N$ and approaches $Q(\theta_{1/N})/N\ge Q_0/N$ when the model error on
noise-free inputs is small.

The noise also enters the stability of the in-loop scheme. To isolate the
learning error, we compare NEOPIC with the teacher run from the same initial
particles, as in the experiments with the CIC reference. Let
$\mathbf{z}^{h,n}_p$ be the teacher particles, advanced by~\eqref{eq:kdk} with
the field $\mathcal{E}_h[\mu^{h,n}_N]$ and the same $\Delta t$, and let $L_h$
and $L_r$ be constants for which~\eqref{eq:lip} holds for $\mathcal{E}_h$ and
for $r_\theta$, respectively. Thus $L_\theta\le L_h+L_r$, where $L_h$ is the
stability constant of the teacher itself.

\begin{proposition}[In-loop error relative to the teacher]
\label{prop:inloop_teacher}
Let
$e^n_h=\|\mathbf{x}^n-\mathbf{x}^{h,n}\|_N+\|\mathbf{v}^n-\mathbf{v}^{h,n}\|_N$
and $\Lambda_h=2\big(1+2|q/m|(L_h+L_r)\big)$. Then for
$\Delta t\le\big(1+2|q/m|(L_h+L_r)\big)^{-1}$,
\begin{equation}
  \mathbb{E}\,e^n_h\le4|q/m|\,t_n\,e^{\Lambda_ht_n}\,
  \max_{k\le n}\Big(\mathbb{E}\,J\big(\theta;\mu^{h,k}_N\big)\Big)^{1/2}.
  \label{eq:inloop_teacher}
\end{equation}
Under (A5)--(A6), the expectation on the right-hand side is given
by~\eqref{eq:noise_expansion} with $N=N_p$.
\end{proposition}

\begin{proof}
As in the proof of Theorem~\ref{thm:trajectory}, with the acceleration error
split as
\begin{equation*}
  \begin{aligned}
  &\mathcal{G}_\theta[\mu^n_N](\mathbf{x}^n_p)-\mathcal{E}_h[\mu^{h,n}_N](\mathbf{x}^{h,n}_p)\\
  &\quad=\big(\mathcal{E}_h[\mu^n_N](\mathbf{x}^n_p)-\mathcal{E}_h[\mu^{h,n}_N](\mathbf{x}^{h,n}_p)\big)\\
  &\qquad+\big(r_\theta[\mu^n_N](\mathbf{x}^n_p)-r_\theta[\mu^{h,n}_N](\mathbf{x}^{h,n}_p)\big)
  +r_\theta[\mu^{h,n}_N](\mathbf{x}^{h,n}_p).
  \end{aligned}
\end{equation*}
The first two terms are bounded by
$(L_h+L_r)\big(|\mathbf{e}^n_{x,p}|+\frac1N\sum_q|\mathbf{e}^n_{x,q}|\big)$, and
the particle norm of the last term is $J(\theta;\mu^{h,n}_N)^{1/2}$. Both runs
use the same integrator and time step, so there is no truncation error.
The discrete Gronwall inequality and
$\mathbb{E}\,J^{1/2}\le(\mathbb{E}\,J)^{1/2}$ give~\eqref{eq:inloop_teacher}.
\end{proof}

Relative to the teacher, the in-loop error is thus controlled by the
teacher-forced loss on teacher data at the inference particle number,
amplified at the rate $\Lambda_h$. The rate depends on the stability of the
teacher, $L_h$, and on the Lipschitz constant of the residual, $L_r$, not on
that of the network as a whole. The measure part of $L_r$ is the sensitivity
of the residual to changes of the particle configuration. By
Proposition~\ref{prop:noise_structure}(i), $Q$ controls this sensitivity in
mean square for the perturbations produced by particle displacements. Sampling
noise in the training data therefore regularizes the quantity that enters the
exponent, in analogy with the Jacobian-mismatch term
in~\cite[Theorem~2.7]{nguyen2022mctangent}. This link is heuristic: $Q$
controls an average over fluctuation directions, whereas $L_r$ is a supremum.

\begin{remark}[Choosing the number of training particles]
\label{rem:choose_ntr}
In~\cite{nguyen2022mctangent,nguyen2025dgnet} the noise level is tuned by
hand: too little noise gives no regularization, and too much destroys the
information in the data. In NEOPIC, choosing the noise level means choosing
$N_{\mathrm{tr}}$, and both limits have a physical meaning. The expansion
of Proposition~\ref{prop:noise_reg} assumes that the fluctuations are small
compared with the signal in the modes that carry the field, i.e.\ a
signal-to-noise ratio
$\mathrm{SNR}_{\mathbf{k}}(t)=\sqrt{N_{\mathrm{tr}}}\,|\hat\rho_{\mathbf{k}}(t)|\gg1$.
Snapshots with $\mathrm{SNR}\lesssim1$ carry little information about the map
on noise-free inputs. At the same time they contribute label noise that the network cannot explain
from its input, of variance about $Q(\theta)/N_{\mathrm{tr}}$ (at least
$Q_0/N_{\mathrm{tr}}$), which a network trained on a finite data set can
partly memorize. This is the information-loss limit. In the other limit,
Proposition~\ref{prop:tradeoff} shows that a large $N_{\mathrm{tr}}$ gives an
accurate but sensitive model. These considerations suggest the smallest
$N_{\mathrm{tr}}$ for which the relevant modes stay above the noise over the
training horizon, and not larger than the smallest (or the harmonic mean of
the) particle numbers planned for inference. Training on data with several
particle numbers would be an alternative, but it requires such data; see
Remark~\ref{rem:noise_caveats}(d).
\end{remark}

\section{Computational complexity}
\seclab{complexity}

We estimate the cost of one evaluation of the NEOPIC field operator, i.e.\ the
map from the $N_p$ particle positions to the electric field at those positions. The particle push and all other parts of
the time step are identical to those of the reference particle methods and
scale as $\mathcal{O}(N_p)$; they are therefore not considered further.

\subsection{Architecture and notation}
\seclab{complexity_architecture}

We use the operator of Section~\secref{architecture} in its discrete
form~\eqref{eq:forward}--\eqref{eq:disc_layer}, with positions
$X=(\mathbf{x}_p)_{p=1}^{N_p}$, equal weights $w_p=1/N_p$, features
$\mathbf{U}^\ell_p\in\mathbb{R}^{d_v}$, $M=|\mathcal{K}_m|=(2m)^d$ retained modes
and the nonuniform Vandermonde matrix $V\in\mathbb{C}^{N_p\times M}$. Since all
particles carry the same charge, the positions are the only input, so each
particle has $d$ input channels and $d$ output channels (the components of
$\mathbf{E}$). One evaluation of the field operator is the map
$X\mapsto\big(\mathbf{E}_\theta[\mu_X](\mathbf{x}_p)\big)_p$. The
transforms~\eqref{eq:forward} and~\eqref{eq:inverse} are evaluated either
directly (NUDFT, as in the DSE formulation of~\cite{lingsch2023beyond}) or with
type-1 and type-2 nonuniform FFTs
(NUFFT)~\cite{barnett2019parallel,shih2021cufinufft}, which compute the same
quantities to a prescribed relative tolerance~$\varepsilon$.

We only require that the maps $\mathcal{P}$, $\mathcal{Q}$ and $\mathcal{W}_\ell$
act pointwise, i.e.\ independently on each particle, and denote their
cost per particle by $C_{\mathcal{P}}$, $C_{\mathcal{Q}}$ and $C_{\mathcal{W}}$.
Any pointwise network (a linear map, an MLP, a $1\times1$ convolution) satisfies
this. In the present work $\mathcal{P}$ is a single linear layer,
$\mathcal{W}_\ell$ is a convolution with kernel size one, and $\mathcal{Q}$ is a
two-layer MLP with hidden width $h$, so that
\begin{equation}
  C_{\mathcal{P}} = d\,d_v,\qquad
  C_{\mathcal{W}} = d_v^2,\qquad
  C_{\mathcal{Q}} = d_v h + h\,d
  \label{eq:pointwise_costs}
\end{equation}
multiply--add operations per particle. Since the output depends on the
particles only through their measure (Lemma~\ref{lem:exact_eval}), the
operator is permutation-equivariant with respect to the particle index, as
required for a field solver. 

\subsection{Cost of the individual operations}

For the NUFFT we use the standard parameters of exponential-of-semicircle
kernel methods~\cite{barnett2019parallel}: a kernel of width
$w\approx\lceil\log_{10}(1/\varepsilon)\rceil+1$ points per dimension and an
upsampling factor $\sigma$ (typically $\sigma=2$), giving a fine grid with
$\sigma^d M$ points. A type-1 transform consists of spreading onto the fine grid
($\mathcal{O}(N_p w^d)$ per channel), an FFT on the fine grid
($\mathcal{O}(\sigma^d M\log M)$ per channel), and a deconvolution
($\mathcal{O}(M)$ per channel). A type-2 transform performs the same steps in
reverse order, with spreading replaced by interpolation. Sorting the
particles and evaluating the one-dimensional kernel values costs
$\mathcal{O}(N_p\,d\,w)$ per field operator evaluation and is shared by all channels and
layers, because the positions do not change within a time step (for multi-stage time integrators this corresponds to every stage).
Table~\ref{tab:complexity} summarizes the costs.

\begin{table}[t]
  \centering\small
  \caption{Leading-order cost of one evaluation of the NEOPIC field operator.
  $L$: number of Fourier layers; $d_v$: lifting (channel) width;
  $M=(2m)^d$: number of retained modes; $w$: NUFFT kernel width;
  $\sigma$: upsampling factor.}
  \label{tab:complexity}
  \begin{tabular}{lll}
    \toprule
    Operation & NUDFT (DSE) & NUFFT \\
    \midrule
    Lifting $\mathcal{P}$ & $N_p C_{\mathcal{P}}$ & $N_p C_{\mathcal{P}}$ \\
    Forward transform (per layer)
      & $N_p M d_v$ & $(N_p w^d + \sigma^d M\log M)d_v$ \\
    Spectral mixing (per layer) & $M d_v^2$ & $M d_v^2$ \\
    Inverse transform (per layer)
      & $N_p M d_v$ & $(N_p w^d + \sigma^d M\log M)d_v$ \\
    Skip $\mathcal{W}_\ell$, activation (per layer)
      & $N_p (C_{\mathcal{W}} + d_v)$ & $N_p (C_{\mathcal{W}} + d_v)$ \\
    Projection $\mathcal{Q}$ & $N_p C_{\mathcal{Q}}$ & $N_p C_{\mathcal{Q}}$ \\
    Sorting and kernel evaluation & -- & $N_p\, d\, w$ \\
    \bottomrule
  \end{tabular}
\end{table}

\subsection{Total cost}

Summing the entries of Table~\ref{tab:complexity} gives the following result.

\begin{proposition}[Cost of the NEOPIC field evaluation]
\label{prop:complexity}
For fixed architecture parameters $(L, d_v, m, h)$ and fixed tolerance
$\varepsilon$, the cost of one evaluation of $\mathcal{G}_\theta$ is
\begin{equation}
  T(N_p) = \alpha\, N_p + \beta ,
  \label{eq:T_total}
\end{equation}
where, for the NUFFT-based evaluation,
\begin{align}
  \alpha_{\mathrm{NUFFT}} &= C_{\mathcal{P}} + C_{\mathcal{Q}}
     + L\big(C_{\mathcal{W}} + d_v + 2\, d_v\, w^d\big) + d\,w,
     \label{eq:alpha_nufft}\\
  \beta_{\mathrm{NUFFT}} &= L\big(2\, d_v\, \sigma^d M\log M + d_v^2 M\big),
     \label{eq:beta_nufft}
\end{align}
and, for the direct evaluation,
\begin{align}
  \alpha_{\mathrm{NUDFT}} &= C_{\mathcal{P}} + C_{\mathcal{Q}}
     + L\big(C_{\mathcal{W}} + d_v + 2\, d_v\, M\big),
     \label{eq:alpha_nudft}\\
  \beta_{\mathrm{NUDFT}} &= L\, d_v^2 M .
     \label{eq:beta_nudft}
\end{align}
In both cases $T(N_p)=\mathcal{O}(N_p)$, since $\alpha$ and $\beta$ are
independent of $N_p$.
\end{proposition}

We make the following remarks.

\begin{remark}[Linear scaling]
\label{rem:linear_scaling}
The number of modes $M$ is a property of the
operator being learned (the resolved wavelengths), not of the particle count.
NEOPIC is therefore an $\mathcal{O}(N_p)$ particle method in the same sense as
PIC, PIF (when the particle costs dominate) or FMM, with a constant determined by the architecture through $\alpha$
and a fixed per-step cost $\beta$ that dominates only for small $N_p$.
\end{remark}

\begin{remark}[NUDFT versus NUFFT]
\label{rem:nudft_vs_nufft}
The two evaluations differ in the per-particle
cost of the transforms, $2Ld_vM$ for the NUDFT versus $2Ld_vw^d$ for the NUFFT.
Since $w^d\ll M$ in practice ($w^3=343$ versus $M=32^3=32768$ for
$\varepsilon=10^{-6}$ and $m=16$ in three dimensions), the NUFFT reduces the
arithmetic per particle by roughly two orders of magnitude in 3D. It also removes the
$\mathcal{O}(N_pM)$ memory of an explicitly assembled $V$
(Section~\secref{memory}). For small $N_p$, small $M$, or in two
dimensions, the direct evaluation can nevertheless be faster, because it
avoids the fixed cost of the fine-grid FFTs in $\beta_{\mathrm{NUFFT}}$ and
consists of dense matrix products that run efficiently on GPUs.
\end{remark}

\begin{remark}[Relation to PIF]
\label{rem:pif_relation}
A PIF field solve with the same $M$ and
$\varepsilon$ performs one type-1 transform (charge deposition) and $d$ type-2
transforms (field components), so $\alpha_{\mathrm{PIF}}\approx(1+d)\,w^d$. In
fact PIF is a special case of~\eqref{eq:err_lift}--\eqref{eq:err_proj} with $L=1$,
$d_v=1$, no pointwise terms, no activation and the fixed diagonal multiplier
$R(k)=-\mathrm{i}k/|k|^2$. The ratio of the transform costs,
\begin{equation}
  \frac{\alpha_{\mathrm{NEOPIC}}}{\alpha_{\mathrm{PIF}}}
  \approx \frac{2\,L\,d_v}{1+d},
  \label{eq:ratio_pif}
\end{equation}
is therefore the price of the additional capacity of the learned operator
(depth, width and nonlinearity), and can be reduced by choosing smaller $L$ or
$d_v$ at a price of reduced capacity. In practice the measured ratio is smaller than~\eqref{eq:ratio_pif},
because the $d_v$ channels of each transform are processed as a batch that shares
sorting, kernel evaluation and memory traffic for the particle coordinates.
\end{remark}

\begin{remark}[Structure of the coupling between particles]
\label{rem:comm_structure}
In NEOPIC, particles
interact statically only through the $M$ Fourier coefficients~\eqref{eq:forward} in each
channel and layer, which are sums over particles. All other operations are local
to individual particles. This makes the scheme well suited to a particle-only
domain decomposition \cite{muralikrishnan2026distributed}, with communication of $\mathcal{O}(L\,d_v\,M)$ numbers per
time step that is independent of $N_p$ and of the spatial distribution of the
particles; this is explored in a follow-up work.
\end{remark}

\begin{remark}[Training]
\label{rem:training}
Backpropagation through~\eqref{eq:forward}
and~\eqref{eq:inverse} costs the same order as the forward pass, since the adjoint
of a type-1 transform is a type-2 transform with opposite sign of the exponent,
and vice versa. The training cost per sample is therefore also given
by~\eqref{eq:T_total} up to a constant factor.
\end{remark}

\subsection{Memory requirements}
\seclab{memory}

We count memory in stored scalars (real numbers; a complex number counts as
two) during one evaluation of $\mathcal{G}_\theta$. As for the cost, the memory
splits into a part proportional to $N_p$ and a part independent of it.

\paragraph{Per-particle storage} The positions and the field require $2dN_p$
scalars. Within a Fourier layer only a few feature arrays of width $d_v$ per
particle are needed at the same time: the input $\mathbf{U}^{\ell}$, its complex copy
passed to the type-1 transform, the complex output of the type-2 transform, and
$\mathcal{W}_\ell\mathbf{U}^{\ell}$. We denote their combined size by $\kappa\,d_v$ per
particle, with $\kappa\approx6$ for a straightforward implementation. The hidden
layer of the projection $\mathcal{Q}$ requires $h$ scalars per particle. During
inference the features of earlier layers can be discarded, so the peak
per-particle storage is $N_p\big(2d+\max(\kappa d_v,\,d_v+h)\big)$, independent of
$L$, $M$ and $\varepsilon$. The NUFFT additionally stores $\mathcal{O}(N_p)$
integers for the sorted particle ordering.

\paragraph{Particle-independent storage} The largest group of parameters is
the spectral weights $R_\ell(k)$, which require $2LMd_v^2$ scalars. The pointwise
maps contribute $n_{\mathrm{pw}}=d\,d_v+L\,d_v^2+d_vh+h\,d$ parameters, plus
$\mathcal{O}(L\,d_v+h)$ biases, which is negligible compared with the spectral
weights. The
Fourier coefficients~\eqref{eq:forward} of one layer require $2d_vM$ scalars, and
the NUFFT fine grids $2b\,\sigma^dM$ scalars, where $b\le d_v$ is the number of
channels transformed in one batch; the FFT workspace is of the same order.

\paragraph{Direct evaluation} If the Vandermonde matrix $V$ is assembled
explicitly, as in the DSE formulation~\cite{lingsch2023beyond}, it requires
$2N_pM$ additional scalars. Because $\mathcal{K}_m$ is a tensor-product set of
wave vectors, $V$ factorizes as
\begin{equation}
  V_{p\mathbf{k}} = \prod_{r=1}^{d} F^{(r)}_{k_r p}, \qquad
  F^{(r)}_{k_r p} = e^{-\mathrm{i}\,k_r x_{p,r}},\quad
  F^{(r)}\in\mathbb{C}^{2m\times N_p},
  \label{eq:tensor_factor}
\end{equation}
so that only the $d$ one-dimensional factors, with $2d\,(2m)$ scalars per
particle, need to be stored. The sums~\eqref{eq:forward} and~\eqref{eq:inverse}
are then evaluated dimension by dimension. For the forward transform, for example,
one first forms the products $U_{pc}F^{(1)}_{k_1 p}$ and
$F^{(2)}_{k_2 p}\cdots F^{(d)}_{k_d p}$ and then contracts over the particle
index $p$ with a single matrix product; the inverse transform proceeds in the
reverse order. The intermediates require $2\big(d_v(2m)+(2m)^{d-1}\big)$ scalars
per particle instead of $2M=2(2m)^d$, while the operation count
$\mathcal{O}(N_pMd_v)$ is unchanged. In three dimensions with $d_v=2m=32$ this
reduces the memory of the transform by a factor of $16$. In two dimensions,
however, $d_v(2m)=M$, so the factorization brings no reduction, and we
therefore use the explicit matrix there.

\begin{proposition}[Memory of the NEOPIC field evaluation]
\label{prop:memory}
Under the assumptions of Proposition~\ref{prop:complexity}, the peak memory of
one inference evaluation of $\mathcal{G}_\theta$ is
\begin{equation}
  S(N_p) = \gamma\,N_p + \delta ,
  \label{eq:S_total}
\end{equation}
with
\begin{align}
  \gamma_{\mathrm{NUFFT}} &= 2d+\max(\kappa d_v,\,d_v+h),
  \label{eq:mem_nufft}\\
  \delta_{\mathrm{NUFFT}} &= 2LMd_v^2 + n_{\mathrm{pw}} + 2d_vM + 2b\,\sigma^dM,
  \label{eq:mem_nufft_delta}\\
  \gamma_{\mathrm{NUDFT}} &= \gamma_{\mathrm{NUFFT}} + \mu,
  \label{eq:mem_nudft}\\
  \delta_{\mathrm{NUDFT}} &= 2LMd_v^2 + n_{\mathrm{pw}} + 2d_vM,
  \label{eq:mem_nudft_delta}
\end{align}
where $\mu=2M$ for an explicitly assembled $V$ and
\begin{equation*}
  \mu=\underbrace{2d\,(2m)}_{\text{factors }F^{(r)}}
  +\underbrace{2\big(d_v(2m)+(2m)^{d-1}\big)}_{\text{intermediates}}
\end{equation*}
for the dimension-wise factorization~\eqref{eq:tensor_factor}. In both cases $S(N_p)=\mathcal{O}(N_p)$.
\end{proposition}

\begin{remark}[NUFFT versus direct evaluation memory requirements]
\label{rem:nufft_versus_direct_memory}
For the NUFFT-based evaluation the memory per particle
does not depend on the number of modes, whereas for the direct evaluation it
grows with $m$. In three dimensions with $m=16$, an explicit $V$ requires
$2M\approx6.6\times10^4$ scalars per particle, i.e.\ about $260$\,GB for
$N_p=10^6$ in single precision. The dimension-wise factorization reduces this
to $\mu\approx4.3\times10^3$ scalars per particle (about $17$\,GB for
$N_p=10^6$), which is still more than an order of magnitude above the
per-particle memory of the NUFFT-based evaluation. This, rather than the
operation count, is what limits the direct evaluation at large $N_p$ in three
dimensions.
\end{remark}

\begin{remark}[Memory requirements for Training]
\label{rem:train_memory}
For backpropagation the features of all layers must be
kept, so the per-particle storage becomes approximately
$2d+L\kappa d_v+h$ per sample, multiplied by the batch size. The optimizer
additionally stores the gradients and, for Adam, two moment estimates of all
parameters, so that the particle-independent part is dominated by
$4\times2LMd_v^2$ scalars. Because the operator is trained at particle numbers
below those used at inference, which Proposition~\ref{prop:tradeoff} also
recommends, the larger training memory does not limit $N_p$ at inference,
which is bounded by $\gamma$ alone.
\end{remark}

\section{Numerical results}
\seclab{results}

We assess NEOPIC in three sets of experiments of increasing generality, followed
by a comparison of computational cost. Section~\secref{landau2d} studies the
dependence of the error on the number of training and inference particles for
2D-2V weak Landau damping, where the error analysis of
Section~\secref{error} applies directly. Section~\secref{multi2d} considers a
single model trained on five 2D-2V mini-apps generated by three different
teachers (PIF, PIC and Barnes--Hut), and Section~\secref{multi3d} repeats this
in 3D-3V with two models that differ by one mini-app in the training set.
Section~\secref{cost} compares the cost of the field evaluation with PIC, PIF
and Barnes--Hut. Further experiments are given in the appendices: the
convergence in the time step (\secref{dtconv}), implicit time
integration in one dimension (\secref{implicit1d}), and the cost of
different architectures (\secref{app_cost}).

\paragraph{Common setting} Unless stated otherwise the operator has $L=4$
Fourier layers of width $d_v=32$ with $m=16$ modes per dimension and a
projection of width $h=128$ (Section~\secref{complexity_architecture}). All accuracy results
use the direct evaluation of the nonuniform transforms (NUDFT); the NUFFT-based
evaluation, which was introduced to remove the memory limit of the direct
evaluation at large $N_p$, is used only in the timings of Section~\secref{cost} for comparison. For
weak Landau damping the trained model with NUDFT gives very close results with both
evaluations (NUFFT tolerance $10^{-6}$). Particles are advanced with the
leapfrog scheme in staggered form, with positions at integer and velocities at
half-integer time steps, started with a half kick
$\mathbf{v}^{1/2}=\mathbf{v}^0+\frac{\Delta t}{2}\frac qm\mathbf{E}^0$; this is
equivalent to~\eqnref{kdk} (Remark~\ref{rem:integrators}). With a magnetic
field we use the Boris scheme, in the same formulation as the corresponding
teacher simulations: for the cyclotron mini-app in the same staggered form, with
the initial velocities taken as $\mathbf{v}^{-1/2}$, and for the Penning trap
in kick--drift--kick form, with positions and velocities at integer time steps
and the initial velocities at $t=0$. On periodic domains the mean of the predicted field over the
particles is subtracted, so that the sum of the self-forces vanishes
($c_{\mathrm{per}}=1$ in~\eqref{eq:output}); in open domains it is not. Positions are normalized per
sample before they enter the network, and the predicted field is rescaled by
$Q_{\mathrm{tot}}/|\Omega|^{(d-1)/d}$, with $Q_{\mathrm{tot}}$ the total charge
and $|\Omega|$ the volume of the domain. After this scaling, each component of
the labels is standardized with its mean and standard deviation over the whole
training set of a model, i.e.\ over all its mini-apps and teachers. These
constants are stored with the model and the standardization is undone at
inference before the mean subtraction in~\eqref{eq:output}. Since this is a
fixed affine map, it can be absorbed into the last linear layer of
$\mathcal{Q}$, and the analysis of Section~\secref{error} applies unchanged.
The models are trained with the mean absolute error ($L^1$ loss) of the predicted field
over all particles and components (Remark~\ref{rem:l1_loss}), using Adam with
learning rate $10^{-3}$ and weight decay $10^{-4}$. The learning rate is halved
every $50$ epochs, and the models are trained for $250$ epochs in the 2D weak
Landau study (Section~\secref{landau2d}) and for $300$ epochs in the 2D and 3D multi-app
studies (Sections~\secref{multi2d} and~\secref{multi3d}).
The models are
trained with a global batch size of $12$ using distributed data parallelism on $12$ NVIDIA A100 GPUs. Training
takes $7$--$34$\,min for the 2D weak-Landau models, increasing with
$N_{\mathrm{tr}}$, $3.2$\,h for the 2D multi-app model and $5.9$ and $7.4$\,h for the 3D models with four and five mini-apps.
For all models the final validation loss is within a factor of $1.5$ of the
training loss.

\paragraph{Error measures} As in Section~\secref{error}, we compare NEOPIC
particle by particle with a reference run that starts from the same initial
particles, and report
\begin{equation}
\begin{gathered}
  e_E(t)=\frac{\sum_p|\mathbf{E}_p(t)-\mathbf{E}^{\mathrm{ref}}_p(t)|_1}
               {\max_{s}\sum_p|\mathbf{E}^{\mathrm{ref}}_p(s)|_1},\\
  e_x(t)=\frac{\sum_p|\mathbf{x}_p(t)-\mathbf{x}^{\mathrm{ref}}_p(t)|_1}{d\,N_p\,L},\qquad
  e_v(t)=\frac{\sum_p|\mathbf{v}_p(t)-\mathbf{v}^{\mathrm{ref}}_p(t)|_1}
               {\sum_p|\mathbf{v}^{\mathrm{ref}}_p(t)|_1},
\end{gathered}
  \label{eq:err_measures}
\end{equation}
where $|\cdot|_1$ is the sum over components and the position difference is the
minimum-image (periodic) distance for periodic problems. In all the cases the length of the domain in all the dimensions is the same and equal to $L$. The field error is
normalized by the maximum over time of the reference norm, because the
reference field itself decays to the noise level in Landau damping and a
pointwise relative error would grow without bound; by~\eqref{eq:norm_chain}
the bounds of Section~\secref{error} apply to these errors. We write $\langle\cdot\rangle_t$ for the average over
the stored time steps. Two kinds of reference are used. With the teacher as reference (the method that generated the training data, run
with the same particles), $\varepsilon_{\mathrm{teacher}}$ drops out and the
measured error is the learned-operator error $\varepsilon_{\mathrm{learn}}$
amplified by the dynamics; with a high-fidelity reference, also run
from the same initial particles, the error contains in addition the teacher's
own error relative to that reference
(Remark~\ref{rem:resolution_mismatch}). In the teacher-forced evaluation the
model is applied to the particles of the teacher run and no time stepping by
the model is involved (Remark~\ref{rem:teacher_forced}).

\paragraph{Field energy} As quantities of interest we use, besides the energy
budget, the field energy per component,
\begin{equation}
  W_i(t)=\int_\Omega E_i^2(\mathbf{x},t)\,d\mathbf{x}
  \approx\frac{|\Omega|}{N_p}\sum_{p=1}^{N_p}E_{p,i}^2(t),\qquad i\in\{x,y,z\},
  \label{eq:field_energy}
\end{equation}
evaluated from the fields at the particles, for the weak Landau, strong Landau, bump-on-tail and two-stream instability test cases. Linear theory predicts $W_i\propto e^{2\gamma t}$ for a single mode
with growth or damping rate $\gamma$, so that the slopes of $\ln W_i$ quoted below
are $2\gamma$. For the cyclotron and the Penning trap mini-apps, where the
particles do not fill the domain, the electrostatic energy is computed
differently, as described there.

\subsection{2D-2V weak Landau damping}
\seclab{landau2d}

The initial distribution is
$f_0(\mathbf{x},\mathbf{v})=(1+\alpha\cos kx)(1+\alpha\cos ky)\,
e^{-|\mathbf{v}|^2/2}/(2\pi)$ with $\alpha=0.05$ and $k=0.5$ on the periodic
domain $[0,4\pi)^2$. The teacher is PIF with cloud-in-cell (CIC) shape function,
$32$ modes per dimension and NUFFT tolerance $10^{-12}$, and the
high-fidelity reference is PIF with a fourth-order shape function and the
same number of modes and tolerance. Five models
are trained on teacher runs with $N_{\mathrm{tr}}\in\{50,100,200,400,800\}\times10^3$
particles over $T=20$ ($400$ steps of $\Delta t=0.05$), using a random
$80/20$ split of the time steps into training and validation data, and each
model is evaluated with $N_p\in\{50,100,200,400,800,1600\}\times10^3$
particles over the same horizon.

\paragraph{Teacher-forced error} Figure~\figref{tf_scaling} shows the time-averaged
teacher-forced field error for all pairs $(N_{\mathrm{tr}},N_p)$. Along the
diagonal $N_p=N_{\mathrm{tr}}$ the error decreases from $2.3\times10^{-3}$ to
$7.3\times10^{-4}$, approximately as $N^{-0.39}$. The diagonal is the column minimum up to small differences, which near the diagonal are second order in $1/N_{\mathrm{tr}}$ and comparable to training variability. Away from the diagonal the
matrix is strongly asymmetric, as predicted by
Proposition~\ref{prop:tradeoff}. Inference at fewer particles than in
training ($N_p<N_{\mathrm{tr}}$, upper left) increases the error considerably:
the model trained with $8\times10^5$ particles has an error of
$7.6\times10^{-3}$ at $N_p=5\times10^4$, ten times its diagonal value and
$3.4$ times the error of the model trained at $N_p$. Inference at more
particles than in training (lower right) is benign. The error of each model
decreases with $N_p$ until it levels off at a value set by the model:
the model trained with $5\times10^4$ particles stays at about
$2\times10^{-3}$ for $N_p\geq10^5$, the one trained with $10^5$ at
$1.2\times10^{-3}$ for $N_p\geq2\times10^5$, and the one trained with
$2\times10^5$ at $7.5\times10^{-4}$ for $N_p\geq8\times10^5$. The models
trained with $4\times10^5$ and $8\times10^5$ particles still improve at
$N_p=1.6\times10^6$. The particle number at which the error levels off thus
moves from about $2N_{\mathrm{tr}}$ for the two smallest models to beyond
$4N_{\mathrm{tr}}$ for $N_{\mathrm{tr}}=4\times10^5$. This is the
accuracy--robustness trade-off of Proposition~\ref{prop:tradeoff}. By~\eqref{eq:noise_expansion} and Remark~\ref{rem:l1_loss},
the expected loss of a trained model at $N_p$ particles is the loss on
noise-free inputs plus a noise term that decreases with $N_p$: $R(\theta)/N_p$
for the squared loss, i.e.\ $(R(\theta)/N_p)^{1/2}$ in the units of the field,
and of order $N_p^{-1/2}$ for the $L^1$ loss used here.
It levels off once the noise term falls below the noise-free loss. Training with more particles lowers
$\mathcal{L}_\infty$ and raises $R$ (Proposition~\ref{prop:tradeoff}(i)), so
that this point moves to larger $N_p$. Sampling noise in the training data
thus regularizes the model against noisier inputs, at the price of a higher
error floor on clean inputs. The evaluation data are independent teacher
runs with a different random seed, so that no training snapshot enters the
evaluation, also on the diagonal (cf.\ Remark~\ref{rem:noise_caveats}(a)).
\begin{figure}[!htbp]
  \centering
  \includegraphics[width=\linewidth]{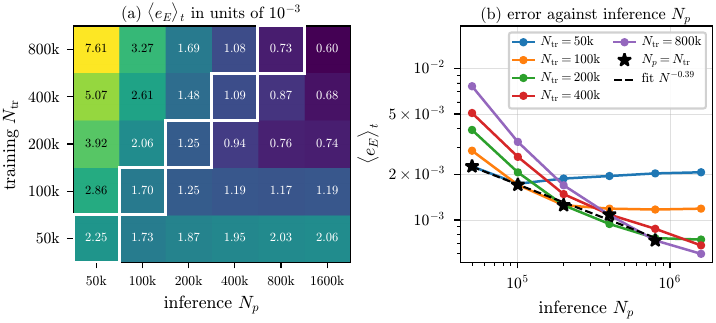}
  \caption{2D-2V weak Landau damping, teacher-forced field error
  $\langle e_E\rangle_t$ of the five models (training particle number
  $N_{\mathrm{tr}}$) evaluated with $N_p$ particles. (a) All pairs, in units
  of $10^{-3}$; diagonal entries ($N_p=N_{\mathrm{tr}}$) outlined. (b) Error
  against $N_p$, with the diagonal (stars) and a power-law fit to it.}
  \figlab{tf_scaling}
\end{figure}

\paragraph{In-loop error} When NEOPIC advances the particles itself, the same
structure appears with larger values (Fig.~\figref{inloop_scaling}(a)--(c)). Against
the CIC teacher, the diagonal field error decreases from $6.7\times10^{-3}$
to $1.5\times10^{-3}$, approximately as $N^{-0.54}$, close to the rate $N^{-1/2}$
expected from Proposition~\ref{prop:inloop_teacher}. The ratio of the in-loop
to the teacher-forced diagonal error is $3.0$ at $5\times10^4$ particles and
decreases to $2.0$ at $8\times10^5$; this is the amplification
by the dynamics in Proposition~\ref{prop:inloop_teacher}, which over $T=20$ is far
below the worst-case Gronwall factor because of phase mixing. The position
and velocity errors are small ($\langle e_x\rangle_t\le3.4\times10^{-4}$ and
$\langle e_v\rangle_t\le7.9\times10^{-4}$ on the diagonal) and follow the same pattern
without a clean power law. Inference below the training particle number is
again the costly direction: at $N_p=5\times10^4$ the electric field error grows from
$6.7\times10^{-3}$ for the model trained at $5\times10^4$ to
$2.0\times10^{-2}$ for the one trained at $8\times10^5$. Over the whole
range of $N_p$, the models trained with $5\times10^4$--$2\times10^5$ particles
stay within a factor $2$ of the best model for each $N_p$, whereas the model
trained with $8\times10^5$ is up to three times worse, in line with
Proposition~\ref{prop:tradeoff}(iii) and Remark~\ref{rem:choose_ntr}. The time
histories in Fig.~\figref{inloop_history} show that the field error is largest
during the first plasma periods and then saturates, while the position error
grows smoothly and the velocity error saturates after $t\approx1$.

Against the high-fidelity reference (Fig.~\figref{inloop_scaling}(d)--(f)) the five
models are indistinguishable, and all three errors decrease as
$N_p^{-0.51}$, $N_p^{-0.47}$ and $N_p^{-0.49}$. The reference is run from the
same initial particles and shares their sampling noise, so this error is
dominated by the different responses of the CIC and the fourth-order shape
functions to that noise. In the splitting of
Remark~\ref{rem:resolution_mismatch}, this is the teacher error
$\varepsilon_{\mathrm{teacher}}(N_p)$ against the high-fidelity reference,
amplified by the dynamics. Its sampling part is the difference
$\mathcal{E}_{\mathrm{CIC}}-\mathcal{E}_{\mathrm{hifi}}$ of the two field
operators applied to the shared fluctuations, of order $N_p^{-1/2}$, and is
thus the same for all models. At $N_p=1.6\times10^6$ it is ten times larger
than the learned-operator error measured against the teacher. The learned
operator therefore does not limit the accuracy at the particle numbers
considered: the remaining error is that between two particle discretizations
of the same particles.

\begin{figure}[!htbp]
  \centering
  \includegraphics[width=\linewidth]{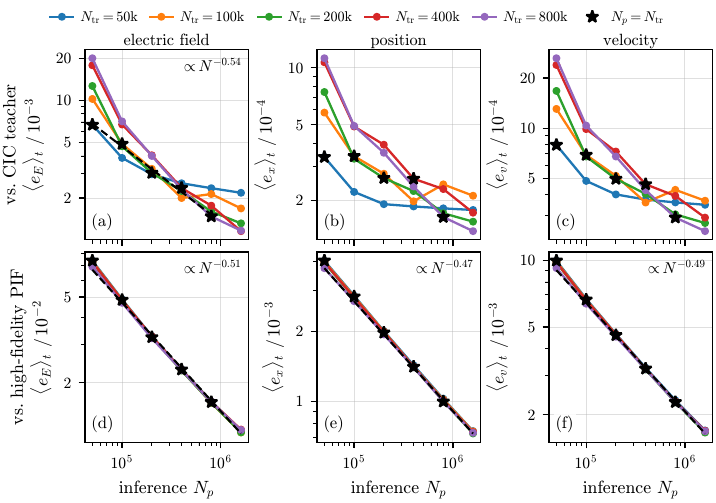}
  \caption{2D-2V weak Landau damping, in-loop errors~\eqref{eq:err_measures}
  averaged over $t\le20$, against the inference particle number $N_p$ for the
  five trained models. (a)--(c) Reference: the CIC teacher with the same
  particles; (d)--(f) reference: high-fidelity PIF (fourth-order shape). Stars
  mark $N_p=N_{\mathrm{tr}}$; dashed lines are power-law fits to the diagonal
  (a) and to the model with $N_{\mathrm{tr}}=8\times10^5$ (d)--(f).}
  \figlab{inloop_scaling}
\end{figure}

\begin{figure}[!htbp]
  \centering
  \includegraphics[width=\linewidth]{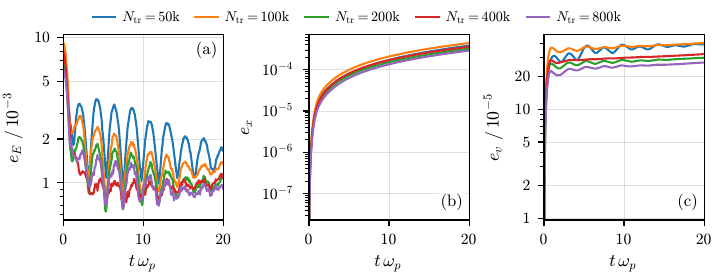}
  \caption{2D-2V weak Landau damping, time histories of the in-loop errors
  against the CIC teacher at $N_p=1.6\times10^6$ for the five trained models.}
  \figlab{inloop_history}
\end{figure}

\paragraph{Generalization} The model trained with $8\times10^5$ particles at
$k=0.5$ was also evaluated at other wave numbers and on other mini-apps
(Fig.~\figref{generalization}). Because positions are normalized per sample and
the predicted field is rescaled by $Q_{\mathrm{tot}}/\sqrt{|\Omega|}$, the
field operator in normalized variables is the same for every $k$: the Poisson
problem on $[0,2\pi/k)^2$ maps to the one on a fixed box. A change of $k$
therefore does not change the map to be learned but the distribution of its
inputs, since $k\lambda_D$, where $\lambda_D$ is the Debye length, sets the damping rate and the oscillation
frequency. For $k>0.5$ the errors even decrease slightly, as the stronger
damping brings the field to the noise level sooner. For $k<0.5$ the
perturbation is weakly damped and remains large over the whole run, and the
field error increases to $4.5\times10^{-3}$ at $k=0.25$ and $8.7\times10^{-3}$
at $k=0.125$, about seven times the value at the training wave number.

Applied to other mini-apps (Fig.~\figref{generalization}(b)), the weak-Landau model
remains accurate for the bump-on-tail instability during the linear phase
($\langle e_E\rangle_t=3.7\times10^{-3}$), and for the two-stream instability
until saturation, after which the error jumps: the time median of $e_E$ is
$4.7\times10^{-4}$, the mean $4.2\times10^{-2}$. It fails for strong Landau
damping, whose large-amplitude, nonlinear density profiles were never seen in
training ($\langle e_x\rangle_t=4.6\times10^{-2}$). A single-case model is thus
limited to the region of input space covered by its training data, which
motivates the multi-app models of the following sections.

\begin{figure}[!htbp]
  \centering
  \includegraphics[width=\linewidth]{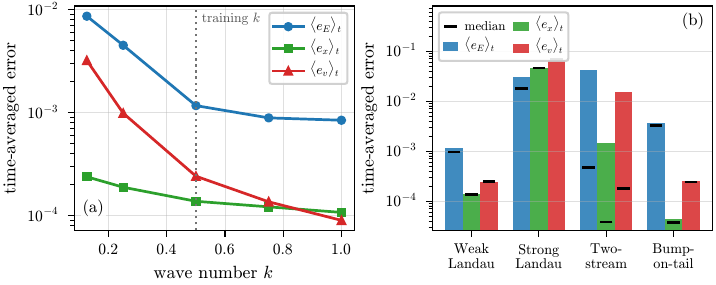}
  \caption{Generalization of the weak-Landau model ($k=0.5$,
  $N_{\mathrm{tr}}=8\times10^5$) at $N_p=1.6\times10^6$, errors against the CIC
  teacher. (a) Other wave numbers; the dotted line marks the training
  wave number. (b) Other mini-apps: time-averaged errors (bars) and time
  medians (black marks).}
  \figlab{generalization}
\end{figure}

\subsection{2D-2V multi-app model}
\seclab{multi2d}

One model is trained on five mini-apps, each generated by a different
reference method: weak ($\alpha=0.05$) and strong ($\alpha=0.5$) Landau
damping with PIF, the two-stream ($k=0.5$) and bump-on-tail ($k=0.21$)
instabilities with PIC (perturbation amplitude $\alpha=0.01$), and the cyclotron mini-app, a
magnetized electron cloud with open boundaries, with Barnes--Hut.
PIF and PIC with FFT field solver use $32$ modes and $32$ grid points per dimension, respectively,
both with the CIC shape function (applied in Fourier space for PIF), PIF
with NUFFT tolerance $10^{-12}$, and Barnes--Hut uses the opening angle
$\vartheta=0.6$ and the softening length $\sqrt{10^{-5}}\approx3.2\times10^{-3}$.
All mini-apps are ran with $5\times10^5$ particles.
The instabilities are the 2D-2V counterparts of the 3D-3V benchmarks
of~\cite{muralikrishnan2024scaling}, with the perturbation and the beams in
the $y$-direction,
\begin{equation*}
  f_0(\mathbf{x},\mathbf{v})=\frac{1+\alpha\cos ky}{2\pi\sigma_x\sigma_y}\,
  e^{-v_x^2/(2\sigma_x^2)}\Big((1-\epsilon)\,e^{-(v_y-v_{b1})^2/(2\sigma_y^2)}
  +\epsilon\,e^{-(v_y-v_{b2})^2/(2\sigma_y^2)}\Big)
\end{equation*}
on $[0,2\pi/k)^2$, with $\epsilon=0.5$, $\sigma_y=0.1$ and
$v_{b1,2}=\mp\pi/2$ for the two-stream and $\epsilon=0.1$,
$\sigma_y=1/\sqrt2$, $v_{b1}=0$ and $v_{b2}=4$ for the bump-on-tail
instability. Unlike in~\cite{muralikrishnan2024scaling}, where all thermal
velocities are equal, we take $\sigma_x=1$ in both cases, so that the
$x$-direction carries an unperturbed Maxwellian and only $W_y$ grows. The cyclotron
mini-app is a Gaussian electron cloud centred at the origin, with standard
deviations $1/(30\sqrt2)$ and $1/(10\sqrt2)$ of the positions in $x$ and $y$,
standard normal velocities and total charge $Q_{\mathrm{tot}}=-10$, in a
uniform magnetic field $B_0=300\,\hat{\mathbf{z}}$, so that $\omega_c\Delta t=3$ for
$\Delta t=0.01$. The time step size $\Delta t=0.05$ for the other four mini-apps and all of them (including the cyclotron) are trained on the first $600$ time steps
($T_{\mathrm{train}}=30$, and $T_{\mathrm{train}}=6$ for the cyclotron) and run
for $1200$ steps at $N_p=10^6$ particles during inference, so that the second half of every run is
an extrapolation in time.

\paragraph{In-loop errors} Figure~\figref{multiapp2d_errors} shows the in-loop errors
against each mini-app's own teacher. Within the training horizon, the time averaged field
error is $4\times10^{-3}$ for both Landau cases, $1.5\times10^{-2}$ and
$8\times10^{-3}$ for the two-stream and bump-on-tail instabilities and $0.11$ for the cyclotron. In the
extrapolation window the time-averaged errors grow by factors between $1.8$
(field, weak Landau) and $6.5$ (positions, bump-on-tail). The largest growth occurs for the
two-stream instability, by a factor $12$ in the field and $39$ in the positions:
after saturation the particles are trapped in the nonlinear potential wells,
and small field differences grow into large trajectory differences. This
sensitivity is a property of the dynamics and not of the model, since two
teacher runs with slightly perturbed fields would separate in the same way;
pointwise errors are therefore not meaningful in this regime and we compare
quantities of interest instead.

\begin{figure}[!htbp]
  \centering
  \includegraphics[width=\linewidth]{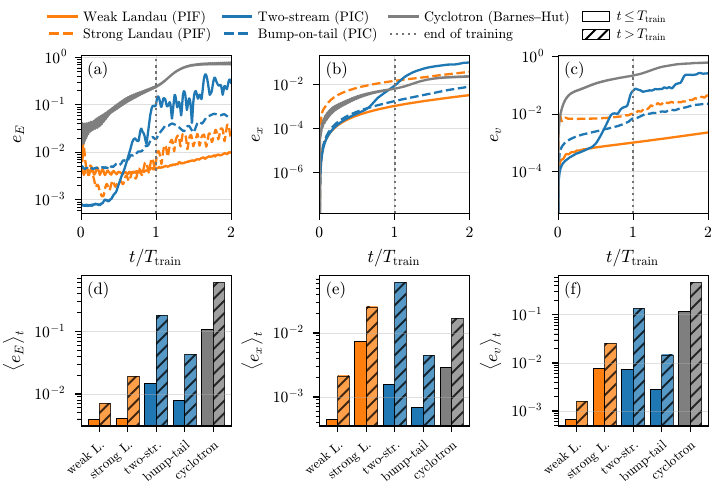}
  \caption{2D-2V multi-app model at $N_p=10^6$, errors~\eqref{eq:err_measures}
  against each mini-app's teacher (colour: teacher). (a)--(c) Time histories
  with time normalized by the training horizon $T_{\mathrm{train}}$; (d)--(f)
  time averages within (plain) and beyond (hatched) the training horizon.}
  \figlab{multiapp2d_errors}
\end{figure}

\paragraph{Quantities of interest} The field energies in
Fig.~\figref{multiapp2d_field} agree closely with the teachers in all four periodic
mini-apps, including the extrapolation window. For weak Landau damping, fits
to the maxima of $W_x$ for $1\le t\le15$ give damping rates of $-0.290$ (PIF)
and $-0.293$ (NEOPIC), against $-0.307$ from linear theory, until the field
reaches the noise level of $10^6$ particles. Strong Landau damping, with its
nonlinear trapping phase, secondary growth and saturated oscillations, is
reproduced over the full horizon. For the instabilities, the growing component
$W_y$ grows at the same rate in NEOPIC and the teacher (fits over the same
windows: $0.614$ versus $0.613$ for the two-stream and $0.297$ versus $0.301$
for the bump-on-tail instability), and the first saturation peak is reached
at the same time ($t=28.70$ versus $28.75$, and $36.35$ versus $36.30$) with
an amplitude that differs by $4\%$ and $1\%$. The non-growing component
$W_x$, which is at the level of the sampling noise, is also reproduced.

\begin{figure}[!htbp]
  \centering
  \includegraphics[width=\linewidth]{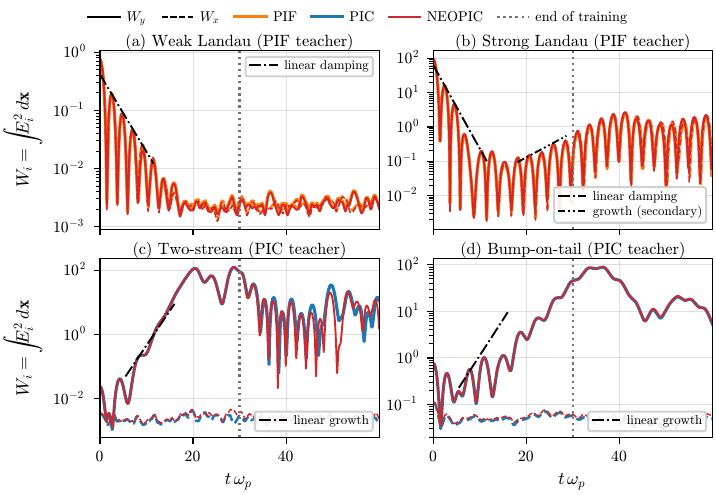}
  \caption{2D-2V multi-app model, field energy per component for the four
  periodic mini-apps: teacher (wide lines, colour as in
  Fig.~\figref{multiapp2d_errors}) and NEOPIC (red). $W_y$ (the growing
  component of the instabilities) is solid and $W_x$ dashed. Dash-dotted: linear theory; dotted: end of
  training.}
  \figlab{multiapp2d_field}
\end{figure}

\paragraph{Conservation} The relative changes of total energy and momentum
are shown in Fig.~\figref{multiapp2d_conservation}. In all periodic mini-apps the energy error of
NEOPIC is of the same size as that of its teacher (for example
$1.5\times10^{-5}$ versus $2.5\times10^{-5}$ for weak Landau damping and
$1.1\times10^{-2}$ versus $9.8\times10^{-3}$ for the two-stream and
$7.4\times10^{-4}$ versus $7.1\times10^{-4}$ for the bump-on-tail
instability), and it does not grow in the extrapolation window. In strong
Landau damping both are of order $10^{-2}$ ($6.7\times10^{-2}$ for NEOPIC and
$3.8\times10^{-2}$ for PIF), which reflects the time step in the strongly
nonlinear phase rather than the field model. Since the mean of the predicted
field is removed, the total force vanishes and momentum is conserved to the
single-precision round-off ($10^{-9}$--$10^{-6}$ relative), whereas a learned
field without this correction would not conserve momentum at all.

\begin{figure}[!htbp]
  \centering
  \includegraphics[width=\linewidth]{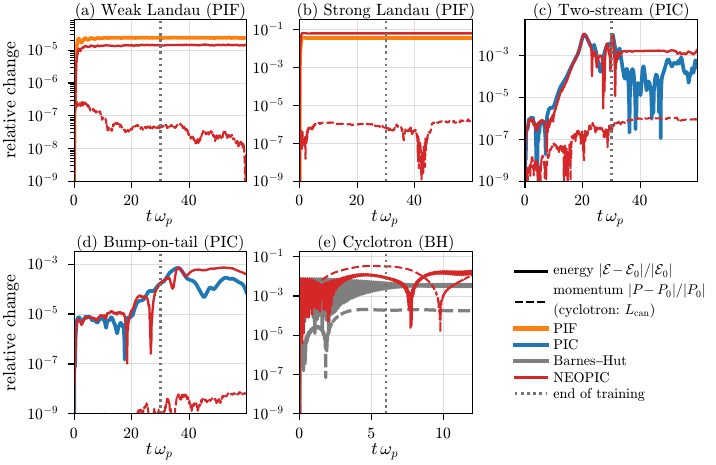}
  \caption{2D-2V multi-app model, relative change of the total energy (solid)
  and of the momentum (dashed; for the cyclotron the canonical angular
  momentum $L_{\mathrm{can}}$~\eqref{eq:Lcan}) for the teachers (wide lines) and NEOPIC (red).}
  \figlab{multiapp2d_conservation}
\end{figure}

\paragraph{Cyclotron mini-app} The cyclotron mini-app is the
most demanding of the five: the domain is open, the cloud occupies about $2\%$
of the solver box, and the strong magnetic field turns the electrostatic
repulsion into an $\mathbf{E}\times\mathbf{B}$ rotation that is sheared
outside the dense core, so that the initially elongated cloud develops two
spiral arms (Fig.~\figref{cyclotron_rho}). Two diagnostics require care in this
case. First, in two dimensions $\frac12\int|\mathbf{E}|^2\,d\mathbf{x}$
diverges logarithmically for a net charge, and the particle average
$\frac{|\Omega|}{2N_p}\sum_p|\mathbf{E}_p|^2$ is only valid for clouds that
fill the box. We therefore use the electrostatic energy
$U=\frac12\sum_pq\,\phi(\mathbf{x}_p)$ with the free-space potential
($G=-\ln r/2\pi$) computed by the spectral free-space solver of Vico, Greengard
and Ferrando \cite{vico2016fast} that the PIC reference uses.
Second, the canonical angular momentum
$\sum_p[m(\mathbf{x}_p\times\mathbf{v}_p)_z+\frac{qB_0}{2}|\mathbf{x}_p|^2]$,
which the exact dynamics conserve for any torque-free field, is not conserved
by the Boris scheme at $\omega_c\Delta t=3$. We therefore use
\begin{equation}
  L_{\mathrm{can}}=\sum_p\Big[\Big(1+\frac{\omega_c^2\Delta t^2}{4}\Big)\,m\,
      (\mathbf{x}^n_p\times\bar{\mathbf{v}}^n_p)_z+\frac{qB_0}{2}|\mathbf{x}^n_p|^2\Big],\quad
  \bar{\mathbf{v}}^n_p=\frac{\mathbf{v}^{n-1/2}_p+\mathbf{v}^{n+1/2}_p}{2},
  \label{eq:Lcan}
\end{equation}
which follows from an exact discrete invariant of the Boris scheme
(\secref{app_cyclotron}). It is conserved exactly for pairwise
Coulomb forces in two dimensions and changes through torques, i.e.\
non-central field errors, and only weakly through errors in the strength of
a central field. The factor $1+(\omega_c\Delta t/2)^2$ tends to one when the
gyration is resolved; the same factor causes the energy oscillation of the
Barnes--Hut reference in Fig.~\figref{multiapp2d_conservation}(e)
(\secref{app_cyclotron}).

Up to the end of training, NEOPIC reproduces the formation of the arms
(Fig.~\figref{cyclotron_rho}), with a field error against a free-space PIC
evaluation at the NEOPIC positions of $1$--$4\%$, against $0.5\%$ for
Barnes--Hut (Fig.~\figref{cyclotron_diagnostics}(a)). Beyond $t\approx7$ this error
rises to about $20\%$. The error maps in Fig.~\figref{cyclotron_dE} show that the
late error is nearly uniform over the cloud and vanishes at its centre,
which indicates an error mainly in the strength of the radial field; the
peak core field is about $10\%$ too large. Such an error leaves $L_{\mathrm{can}}$
nearly unchanged, and indeed the angular momentum of NEOPIC is broken mainly within
the training window (maximum $3.5\times10^{-2}$), where the error maps show a
direction-dependent pattern near the core that exerts a torque. The
overestimated field drives a faster $\mathbf{E}\times\mathbf{B}$ rotation:
measured by the phase of the azimuthal $m=2$ Fourier mode of the density
(\secref{app_cyclotron}), the
arms wind $14\%$ more than in the Barnes--Hut run between $r=0.09$ and
$r=0.29$ at $t=11.98$ (0.815 versus 0.715 turns,
Fig.~\figref{cyclotron_diagnostics}(b)). The extra winding is concentrated in the
shear layer just outside the core, $0.09\lesssim r\lesssim0.14$; for
$0.14\le r\le0.29$ the winding agrees to within $1\%$
(Fig.~\figref{cyclotron_diagnostics}(c)). The total energy of NEOPIC changes by
$1.3\times10^{-2}$ within and $2.0\times10^{-2}$ beyond the training horizon,
against $7.5\times10^{-3}$ for Barnes--Hut.

\begin{figure}[!htbp]
  \centering
  \includegraphics[width=\linewidth]{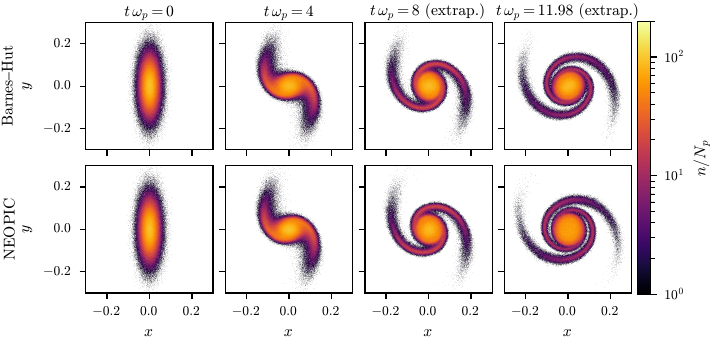}
  \caption{Cyclotron mini-app (2D magnetized electron cloud, open boundaries): particle density
  $n/N_p$ per unit area for Barnes--Hut (top) and NEOPIC (bottom) at
  $N_p=10^6$, binned on a fixed grid about the centre of the cloud. Training
  ends at $t=6$.}
  \figlab{cyclotron_rho}
\end{figure}

\begin{figure}[!htbp]
  \centering
  \includegraphics[width=\linewidth]{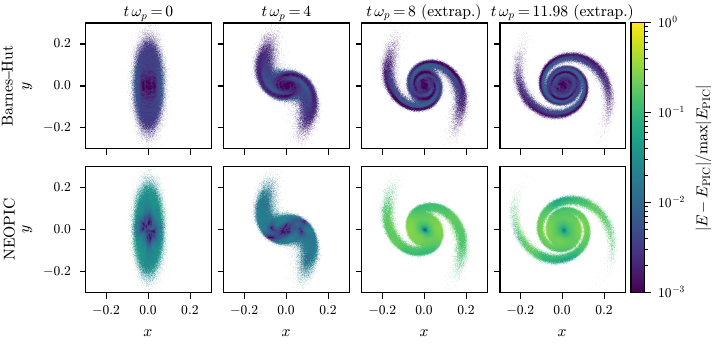}
  \caption{Cyclotron mini-app: local difference between the field of
  each method and a free-space PIC field evaluated at the same particle
  positions, averaged per bin and normalized by the maximum PIC field.}
  \figlab{cyclotron_dE}
\end{figure}

\begin{figure}[!htbp]
  \centering
  \includegraphics[width=\linewidth]{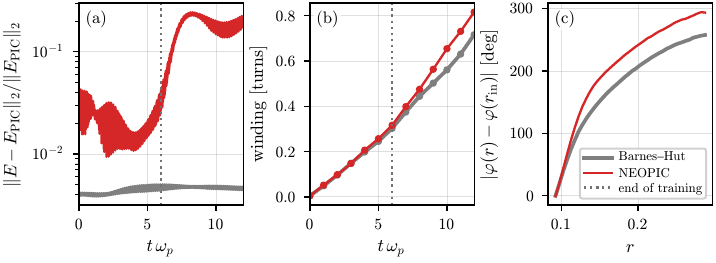}
  \caption{Cyclotron mini-app. (a) Relative $\ell^2$ difference between
  the field of each method and a free-space PIC evaluation at the same
  positions. (b) Winding of the two spiral arms between $r=0.09$ and $r=0.29$,
  from the phase $\varphi(r)$ of the $m=2$ azimuthal mode of the density.
  (c) Arm shape $|\varphi(r)-\varphi(0.09)|$ at $t=11.98$.}
  \figlab{cyclotron_diagnostics}
\end{figure}

\paragraph{Convergence in the time step} The cyclotron mini-app has the
fastest time scale of the 2D mini-apps, and we use it in
\secref{dtconv} to test the time-step dependence of
Theorem~\ref{thm:trajectory}. NEOPIC and PIF started from the same particles
converge at the same rate in the time step, measured by the guiding centres,
and the model error does not depend on the time step, so that the model can
be used with time steps much smaller than those of its training data.

\subsection{3D-3V multi-app models}
\seclab{multi3d}

In three dimensions we train two models: one on four mini-apps (weak and
strong Landau damping with Barnes--Hut, the two-stream and bump-on-tail
instabilities with PIF), and one on the same four mini-apps plus a Penning trap
with PIC as teacher. The second model measures the effect of a more diverse
training set on the original four mini-apps. Training was done with $N_p=100\,000$ particles, whereas inference uses
$163\,840$ particles, $\Delta t=0.05$ and $1200$ steps, of which the first
$600$ ($T_{\mathrm{train}}=30$) are covered by training; the reference runs are
the same for both models.
As in two dimensions, PIF and PIC with FFT field solver use $32$ modes and grid points per
dimension and the CIC shape function, PIF with NUFFT tolerance $10^{-7}$;
Barnes--Hut uses $\vartheta=0.5$ and the softening length $0.05$. The Landau damping and instability mini-apps are those
of~\cite[Section~4]{muralikrishnan2024scaling}: Landau damping with $k=0.5$
and $\alpha=0.05$ (weak) or $\alpha=0.5$ (strong), perturbed in all three
directions, and the instabilities with the parameters given above for 2D,
but with the perturbation and the beams in the $z$-direction and
$\sigma_x=\sigma_y=\sigma_z=\sigma$. The Penning trap also follows~\cite{muralikrishnan2024scaling}:
electrons in the uniform magnetic field $\mathbf{B}_{\mathrm{ext}}=5\,\hat{\mathbf{z}}$
and the quadrupole electric field
$\mathbf{E}_{\mathrm{ext}}=\frac{15}{L}\big(-(x-\tfrac L2),\,-(y-\tfrac L2),\,2(z-\tfrac L2)\big)$
with a neutralizing ion background, total charge $Q_{\mathrm{tot}}=-1562.5$,
Gaussian positions with mean $L/2$ and standard normal velocities. Unlike
in~\cite{muralikrishnan2024scaling}, we take $L=25$ instead of $20$ and the
standard deviations $(3,1,4)$ of the positions.

\paragraph{Quantities of interest} Both models reproduce the field energies
of the four mini-apps (Fig.~\figref{multiapp3d_field}). The damping rates of weak
Landau damping, fitted for $t\le9$ before the field reaches the noise level of
$163\,840$ particles, are $-0.191$ (Barnes--Hut), $-0.190$ (four mini-apps) and
$-0.179$ (five mini-apps). For the two-stream instability the growth rate
agrees within $2\%$ ($0.547$ for PIF, $0.539$ for the four-app and $0.537$ for
the five-app model), and the saturation peak of
$W_z$ is reached about one time unit later and $17$--$26\%$ higher than in
the reference ($t=16.4$ and $W_z=509$ for PIF, against $t=17.7$ and
$W_z=594$ for the four-app and $t=17.3$ and $W_z=642$ for the five-app model). After saturation both models give somewhat lower field
energies than PIF. For the bump-on-tail instability the saturation peak
agrees within $7\%$ in amplitude and $0.4$ in time. In this case the linear
phase is hidden below the noise level at this particle number: $W_z$ stays
within a factor of about three of the non-growing components until $t\approx15$,
and once it has risen well above them it grows more slowly than linear theory
predicts ($2\gamma=0.356$). The slopes of $\ln W_z$ fitted for $10\le t\le25$
($0.151$ for PIF, $0.203$ for the four-app and $0.183$ for the five-app model)
are therefore not growth rates in the linear sense, and no theory line is
shown. In strong Landau damping, whose field stays well above the noise
level, both models follow the reference through the nonlinear phase; the mean
deviation $\langle|\ln(W_i/W_i^{\mathrm{ref}})|\rangle_t$ is $0.07$--$0.11$ for
the four-app and $0.16$--$0.20$ for the five-app model within the training
horizon, and $0.14$--$0.45$ beyond it.

The two models differ most in the level of the field fluctuations, i.e.\ in
the field energy that remains after Landau damping and in the non-growing
components of the instabilities. This level is set by the sampling noise of
the particles and by how the teacher's discretization (resolved modes, shape
function, opening angle) responds to it, so it differs between teachers. To
separate the two effects we repeated weak Landau damping and the bump-on-tail
instability from the same initial particles (taken from the Barnes--Hut run)
with both Barnes--Hut and PIF as reference (Fig.~\figref{multiapp3d_noise}).
In weak Landau damping the fluctuation level of PIF is six times that of
Barnes--Hut ($\bar W=\frac13(W_x+W_y+W_z)\approx1.08$ against $0.17$,
geometric means over $12\le t\le60$).
The four-app model reproduces the level of its teacher, Barnes--Hut ($0.17$),
whereas the five-app model lies in between ($0.31$). For the bump-on-tail
instability, whose teacher is PIF, the field energy of the non-growing
components, $\frac12(W_x+W_y)$, is about $70$ for PIF and $17$--$23$ for
Barnes--Hut. The five-app model again stays in
between, at $25$--$33$, over the whole run including the extrapolation window,
while the four-app model starts at the Barnes--Hut level ($17$ for $t<10$),
rises during the linear phase and switches to the PIF level shortly before the
end of training ($65$ for $25\le t<30$). Its level in the extrapolation window
($47$--$64$) depends on the random seed, which we confirmed with a second seed.
The growing component behaves in the same way during the linear phase: for
$t<12$ the four-app model follows Barnes--Hut closely ($W_z\approx26$ and $39$,
averaged over $0\le t<5$ and $5\le t<12$, against $26$ and $40$ for Barnes--Hut and $84$ and $108$ for
PIF), and the five-app model is slightly above ($35$ and $53$). In both cases
the five-app model sits about one third of the way from Barnes--Hut to PIF on
a logarithmic scale.

We interpret this as follows. The fluctuating part of the field is a linear
response of the operator to the sampling fluctuations of the input, and by
Proposition~\ref{prop:noise_structure} the noise term in the training
loss (Remark~\ref{rem:l1_loss}) penalizes the difference between the response of the model and that of
the teacher. If training data of several teachers cover similar inputs, as
the nearly uniform densities of Landau damping and of the early phase of the
instabilities do, the minimizer of this term is a weighted average of the
teachers' responses, i.e.\ a discretization in between. The four-app model
does not form such an average consistently: it reproduces Barnes--Hut on inputs
that resemble its Landau training data and switches to PIF once the
bump-on-tail density becomes distinctive near saturation. Adding the Penning
trap, a third teacher (PIC) with very different inputs, makes the response
consistent across the mini-apps and over time, including the extrapolation
window. This consistency comes at the price of a higher fluctuation level than
the teacher in weak Landau damping, but it goes along with better energy
conservation in seven of the eight cases of Table~\ref{tab:conservation}, most
clearly for the bump-on-tail instability. A more
diverse training set thus produces a model that behaves like a single
discretization rather than switching between its teachers. In the 2D model, trained and evaluated with more
particles, the fluctuation level of each mini-app follows its teacher
(Fig.~\figref{multiapp2d_field}). Whether a model averages or switches
between its teachers is, however, not the goal in itself; what matters is
the accuracy of the resulting simulations. Averaging the responses of
teachers with different noise responses can be less accurate than either
teacher, and switching between them can make the response inconsistent in
time. Training sets in which the teachers are complementary, each used where
it is accurate, and consistent in the regimes where their inputs overlap
would let a single model follow the most accurate response everywhere.
Constructing such training sets, and making the transition between teachers
smooth, is left to future work (Section~\secref{conclusions}).

\begin{figure}[!htbp]
  \centering
  \includegraphics[width=\linewidth]{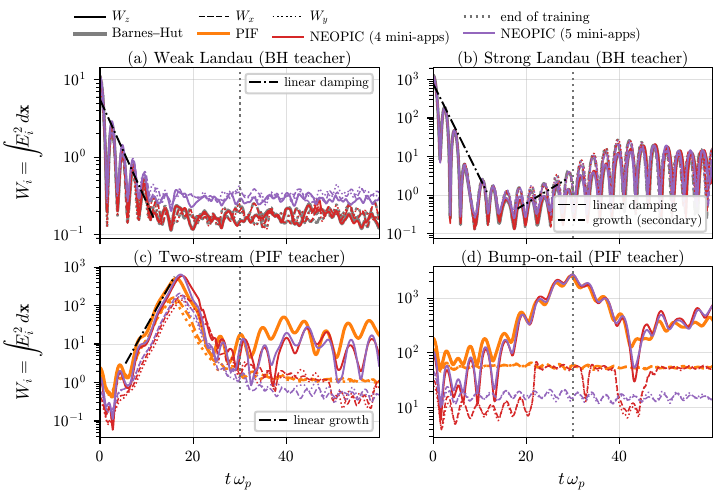}
  \caption{3D-3V multi-app models at $N_p=163\,840$, field energy per component:
  reference (wide lines; grey: Barnes--Hut, orange: PIF), model trained on
  four mini-apps (red) and model trained on five mini-apps including the
  Penning trap (purple). $W_z$ (the growing component of the instabilities)
  is solid, $W_x$ dashed and $W_y$ dotted. For the bump-on-tail instability the
  linear phase is below the noise level and no theory line is shown.}
  \figlab{multiapp3d_field}
\end{figure}

\begin{figure}[!htbp]
  \centering
  \includegraphics[width=\linewidth]{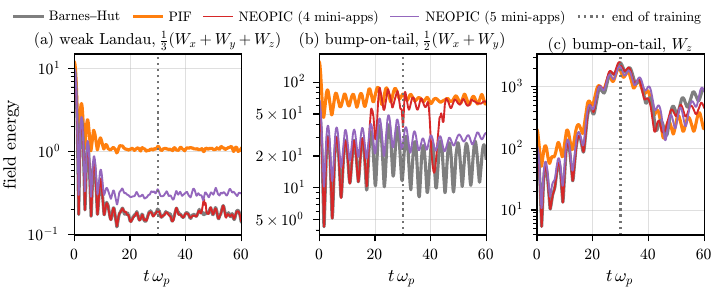}
  \caption{Field fluctuation level of the 3D-3V models: runs from the same
  initial particles (Barnes--Hut) compared with both Barnes--Hut and PIF.
  (a) Weak Landau damping, mean of the three components; (b) bump-on-tail
  instability, mean of the two non-growing components; (c) bump-on-tail
  instability, growing component.}
  \figlab{multiapp3d_noise}
\end{figure}

\paragraph{Conservation} Energy conservation is summarized in
Table~\ref{tab:conservation}. For weak and strong Landau damping both models
conserve energy as well as Barnes--Hut. For the instabilities their energy
errors within the training horizon are larger than those of PIF, by a factor
of $8$--$11$ for the four-app model and $2$--$5$ for the five-app model, and
they do not grow appreciably in the extrapolation window. The five-app model
has the smaller error in all cases except the extrapolation window of the
two-stream instability ($5.1\times10^{-3}$ against $3.5\times10^{-3}$). Barnes--Hut does not conserve momentum, because the tree
approximation breaks the symmetry of the pair forces, whereas NEOPIC
conserves it to $10^{-5}$ relative in these periodic mini-apps through the
removal of the mean field which can also be done for Barnes--Hut if desired.

\begin{table}[!htbp]
  \centering
  \caption{3D-3V multi-app models: maximum relative change of the total
  energy, $\max_t|\mathcal{E}(t)-\mathcal{E}(0)|/|\mathcal{E}(0)|$, within
  ($t\le T_{\mathrm{tr}}$) and beyond ($t>T_{\mathrm{tr}}$) the training
  horizon, for the reference and the models trained on four and five
  mini-apps. The smaller of the two model errors is shown in bold.}
  \tablab{conservation}
  \footnotesize
  \setlength{\tabcolsep}{4pt}
  \begin{tabular}{lcccccc}
    \toprule
    & \multicolumn{2}{c}{reference} & \multicolumn{2}{c}{NEOPIC, 4 apps} & \multicolumn{2}{c}{NEOPIC, 5 apps}\\
    \cmidrule(lr){2-3}\cmidrule(lr){4-5}\cmidrule(lr){6-7}
    mini-app (teacher) & $t\le T_{\mathrm{tr}}$ & $t>T_{\mathrm{tr}}$ & $t\le T_{\mathrm{tr}}$ & $t>T_{\mathrm{tr}}$ & $t\le T_{\mathrm{tr}}$ & $t>T_{\mathrm{tr}}$\\
    \midrule
    weak Landau (BH) & $3.5{\cdot}10^{-4}$ & $3.6{\cdot}10^{-4}$ & $3.8{\cdot}10^{-4}$ & $4.1{\cdot}10^{-4}$ & {\boldmath$3.1{\cdot}10^{-4}$} & {\boldmath$3.2{\cdot}10^{-4}$}\\
    strong Landau (BH) & $8.3{\cdot}10^{-2}$ & $7.9{\cdot}10^{-2}$ & $7.2{\cdot}10^{-2}$ & $6.8{\cdot}10^{-2}$ & {\boldmath$6.9{\cdot}10^{-2}$} & {\boldmath$6.6{\cdot}10^{-2}$}\\
    two-stream (PIF) & $1.9{\cdot}10^{-3}$ & $6.7{\cdot}10^{-5}$ & $2.1{\cdot}10^{-2}$ & {\boldmath$3.5{\cdot}10^{-3}$} & {\boldmath$8.9{\cdot}10^{-3}$} & $5.1{\cdot}10^{-3}$\\
    bump-on-tail (PIF) & $3.9{\cdot}10^{-4}$ & $3.3{\cdot}10^{-4}$ & $3.1{\cdot}10^{-3}$ & $3.6{\cdot}10^{-3}$ & {\boldmath$9.5{\cdot}10^{-4}$} & {\boldmath$8.4{\cdot}10^{-4}$}\\
    Penning trap (PIC) & $3.0{\cdot}10^{-3}$ & $1.6{\cdot}10^{-3}$ & -- & -- & $6.6{\cdot}10^{-2}$ & $5.0{\cdot}10^{-2}$\\
    \bottomrule
  \end{tabular}
\end{table}

\paragraph{Penning trap} The fifth mini-app confines a non-neutral cloud with
an axial magnetic field and an external quadrupole electric field; the total
energy is the sum of the kinetic, self-field and external potential energies
(Fig.~\figref{penning}). Since the cloud does not fill the periodic box, the field
energy is computed on the PIC grid from the particle positions,
$\frac12\int|\mathbf{E}|^2\,d\mathbf{x}$, which is the same diagnostic for both
methods. PIC conserves the total energy to $3\times10^{-3}$. NEOPIC gains up to
$6.6\%$ during the initial relaxation of the cloud in the trap ($t\lesssim10$),
when the self-field changes fastest, and then keeps the energy bounded, with a
maximum change of $5.0\%$ in the extrapolation window. The partition of the energy drifts slowly: at
$t=60$ the self-field energy is $1.04\times10^4$ for NEOPIC against
$1.23\times10^4$ for PIC, and the external potential energy $4.7\times10^3$
against $2.5\times10^3$.

\begin{figure}[!htbp]
  \centering
  \includegraphics[width=\linewidth]{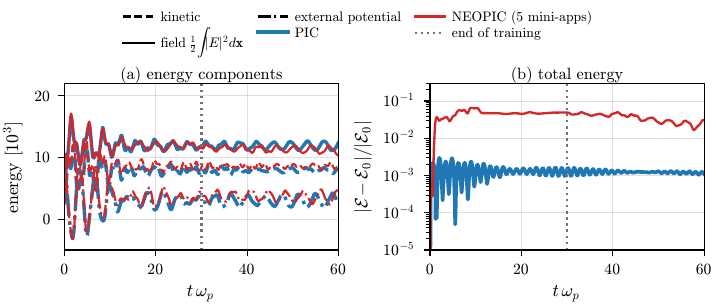}
  \caption{Penning trap (3D, five-app model, $N_p=163\,840$). (a) Kinetic,
  self-field and external potential energy; (b) relative change of the total
  energy.}
  \figlab{penning}
\end{figure}

\subsection{Computational cost}
\seclab{cost}

Figure~\figref{complexity} shows the time per field evaluation on one NVIDIA
A100 GPU as a function of $N_p$ in 2D and 3D, averaged over $20$ time steps
of 2D-2V and 3D-3V weak Landau damping, respectively. For PIC and PIF the
grid or the number of modes grows ($32^2-4096^2$ for 2D and $16^3-256^3$ for 3D) with $N_p$ at a fixed number of particles per
cell ($10$ for 2D and $4$ for 3D), whereas NEOPIC keeps its architecture ($M=32^d$ modes) fixed. We show two
groups of implementations: Python/CuPy implementations of PIC and PIF, the latter with the
same cuFINUFFT interface as NEOPIC (filled markers), and production C++ codes,
IPPL~\cite{matthias_frey_2024_10878166} for PIC and PIF and
Ryoanji~\cite{keller2023cornerstone} for Barnes--Hut (open markers).

Both NEOPIC evaluations scale linearly, as stated in
Proposition~\ref{prop:complexity}; a fit of $T=\alpha N_p+\beta$ gives
$\alpha=79$\,ns and $\beta=64$\,ms in 2D and $\alpha=300$\,ns and $\beta=49$\,ms in
3D for the NUFFT, and $\alpha=155$\,ns ($1.4\,\mu$s) with a small $\beta$ in 2D (3D) for
the direct (NUDFT) evaluation. The direct evaluation is therefore faster in 2D up to
about $7\times10^5$ particles, while in 3D the NUFFT is faster beyond about
$4\times10^4$, in agreement with the operation counts of
Section~\secref{complexity}. The large constant $\beta$ of the NUFFT version
comes from creating the cuFINUFFT plans in every call; persistent plans would
remove most of it. Per particle, NEOPIC with the NUFFT costs $39$ times more
than the Python PIF with the same tolerance in 2D ($2.0$\,ns) and $37$ times more
in 3D ($8.2$\,ns). This is below the ratio $2Ld_v/(1+d)$ of the transform counts
in~\eqref{eq:ratio_pif} ($85$ and $64$), because the $d_v$ channels share the
sorting and the particle data in a batched transform. Against the production
codes the gap is larger: IPPL PIC needs $0.63$\,ns ($1.5$\,ns) per particle in
2D (3D), so NEOPIC is about $120$ to $200$ times more expensive per
particle than an optimized FFT-based PIC code, and about six times more
expensive than Barnes--Hut with opening angle $\vartheta=0.5$ at $1.6\times10^7$
particles ($4.7$\,s against $0.77$\,s). The cost of Barnes--Hut grows rapidly as
$\vartheta$ decreases, and in our tests NEOPIC is cheaper than Barnes--Hut for
$\vartheta\lesssim0.2$ (not shown); such a comparison is only meaningful at
matched accuracy, since the accuracy of NEOPIC is bounded by that of its
teacher. NEOPIC is therefore not competitive with optimized particle methods at equal
particle numbers. However, its cost is independent of the teacher and of the spatial
distribution of the particles, the particles interact only through a constant number of
($\mathcal{O}(Ld_vM)$) Fourier coefficients, which is attractive for a
particle-only parallelization which is load balanced by design \cite{muralikrishnan2026distributed}. In a follow up work, we are exploring these, in a fully distributed parallel implementation of NEOPIC coupled to production particle codes targeting highly non-homogeneous/clustered particle distributions.

The dependence of the cost on the architecture is analysed in
\secref{app_cost}: the cost per particle grows linearly with $Ld_v$
and hardly depends on the number of modes, so that the resolution of the
operator is cheap.

\begin{figure}[!htbp]
  \centering
  \includegraphics[width=\linewidth]{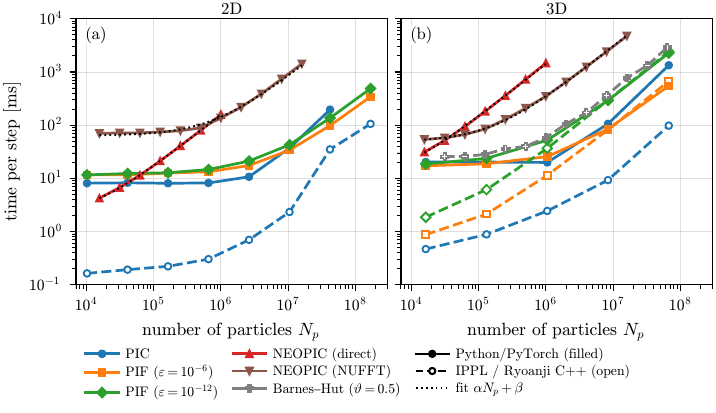}
  \caption{Time per field evaluation (one time step) against the number of
  particles for (a) 2D and (b) 3D Landau damping, on one GPU. Colours
  identify the method, line styles the implementation: Python/PyTorch (solid,
  filled markers) and the C++ codes IPPL and Ryoanji (dashed, open markers).
  Dotted lines: fits $T=\alpha N_p+\beta$ to the NEOPIC timings.}
  \figlab{complexity}
\end{figure}

\clearpage

\section{Conclusions}
\seclab{conclusions}

We have introduced NEOPIC, a framework for learning the field operator of
particle-based methods for kinetic plasma simulations. Its central element is a
nonuniform Fourier neural operator that acts directly on the empirical measure
of the particles and returns the field at the particle positions, without a
mesh, a tree or a deposition step. Because the input is the particle measure
itself, the discrete evaluation coincides with the continuous operator at any
number of particles, and the particle noise is not aliased onto the retained
Fourier modes. NEOPIC replaces only the field solve: it can be trained on data
from PIC, PIF or tree codes, and it can be combined with the standard time
integrators in any production particle code, explicit as well as implicit, the latter with
exact Jacobians obtained by automatic differentiation.

The error analysis shows that the error of a particle simulation with NEOPIC
separates into a Monte Carlo sampling error, a time-stepping error and a
model error, so that, up to the model error, NEOPIC behaves like a standard
Monte Carlo particle method. The particle noise in the training data acts as
an implicit regularizer that drives the response of the model to density
fluctuations towards that of its teacher; the resulting trade-off between
accuracy and robustness gives guidance on the number of training particles in
relation to the particle numbers used in simulations. The cost per time step
is $\mathcal{O}(N_p)$, with constants set by the architecture alone.

The numerical experiments support these findings. For 2D weak Landau
damping, the in-loop field error against the teacher decreases approximately
as $N_p^{-1/2}$ when the training and simulation particle numbers are equal.
For a given model, the teacher-forced error decreases with the number of simulation
particles
until it levels off at a knee, which lies at two to four times the training
particle number for the models trained with fewer particles. Simulations with
fewer particles than in training are the costly direction, whereas the models
trained with fewer particles stay within a factor of two of the best model
over the whole range of particle numbers. Theory and experiments therefore
suggest training close to the smallest particle number expected in the
simulations, which also reduces the cost of generating the training data and
of training. Against a high-fidelity reference, all models converge as
$N_p^{-1/2}$ up to the largest particle number considered, $32$ times the
smallest training particle number. At these particle numbers the error is
dominated by the discretization of the teacher and is about ten times larger
than the error of the learned operator, so that the accuracy of NEOPIC
simulations keeps improving with the number of particles beyond that of the
training data.
A single 2D model trained on five mini-apps generated by three different
methods, including the cyclotron mini-app with open boundaries,
reproduces the damping and growth rates of the teachers to within about $1\%$,
conserves momentum to round-off and energy to the level of its teachers on the
periodic mini-apps, and remains stable up to twice the training horizon. For the
cyclotron mini-app, the field error grows beyond the training horizon, and
diagnostics such as the canonical angular momentum identify its origin in
non-central field errors. In three dimensions, models trained on four and five
mini-apps, including a Penning trap, show that a model trained on several
teachers combines their responses to the particle noise: a more diverse
training set leads to a model that behaves like a single, consistent
discretization rather than switching between its teachers. The time-step
study for the cyclotron mini-app confirms that NEOPIC and PIF have the same
second-order time-stepping error, which is set by the integrator, while the
model error does not depend on the time step, so that the model can be used
with time steps smaller than those of its training data. With the
implicit midpoint rule, NEOPIC remains stable and nearly energy-conserving
in 1D-1V weak Landau damping well beyond the stability limit of the leapfrog scheme. Finally, the cost per
particle grows linearly with the product of the number of layers and the
channel width and hardly depends on the number of Fourier modes, so that the
resolution of the operator is cheap; the smallest architectures we considered
have a cost per time step comparable to that of Barnes--Hut and high-accuracy
PIF, while the current single-GPU implementation is not competitive with
optimized particle codes at equal particle numbers.

This work opens several directions. First, since the particles interact only
through a moderate number of Fourier coefficients per layer, NEOPIC lends
itself to a parallelization over particles alone, without the domain
decomposition and load balancing of mesh- and tree-based methods. Together
with optimizations of the implementation, such as persistent NUFFT plans and
reduced precision, and a coupling to production particle codes, this is the
basis for simulations with billions of particles. Second, the multi-app
results suggest a foundation model for fields in particle-based methods,
pre-trained on many methods and applications and fine-tuned with few
application-specific samples, with the aim of real simulations. Our results
show that a model trained on several teachers switches between their
responses according to the input distribution; making this switch smooth and
controlled, through consistent data in overlapping regimes or explicit gating,
and conditioning the model on the desired fidelity, so that it reproduces
the behaviour of an expensive teacher where only data of a cheaper teacher
were available, are natural ingredients of such a model. Third, the framework
extends beyond electrostatics. Magnetized plasmas with
$\omega_c\Delta t\gg1$ and implicit Darwin models, in which NEOPIC can serve as
the field solver or as a preconditioner of the nonlinear solve, are a natural
next step, followed by electromagnetic models and collisions. Other field
boundary conditions (Remark~\ref{rem:field_bc}) and other particle--field
methods with the same structure, such as vortex methods, can be treated in
the same framework.

\appendix
\addtocontents{toc}{\protect\apptocnumberline}

\section{Further details of the architecture}
\seclab{app_arch}

\subsection{Min--max normalization}
\label{rem:normalization}

The map~\eqref{eq:minmax} depends on the particles only through the bounding
box of the support of $\mu_X$, i.e.\ only through the measure. For a given
configuration it is a fixed affine map applied pointwise, so
Lemma~\ref{lem:exact_eval} and Theorems~\ref{thm:reno} and~\ref{thm:eps_reno}
below apply unchanged. In an open domain it makes the input of the network
invariant under translations and componentwise rescalings of the cloud, which
is useful when the cloud moves and changes its size. Three consequences
should be kept in mind. First, the box is set by the extreme particles: it is
Lipschitz in the positions only with respect to the maximal displacement
$\max_p|\mathbf{x}_p-\mathbf{y}_p|$, not the average displacement used in
Lemma~\ref{lem:lipschitz}, and it is not a continuous function of the measure
in the weak sense. Second, the box depends on $N_p$ through the input
distribution. For a density with compact support it converges to the support
as $N_p$ grows. For a density with unbounded support, such as a Gaussian
cloud, its extent grows slowly, like $\sqrt{2\ln N_p}$ standard deviations per
side, so the same physical density is mapped to slightly different normalized
densities at different particle numbers. On periodic domains filled by the
plasma, the box differs from the domain by $\mathcal{O}(L/N_p)$ per side.
Third, the transforms treat the normalized box as periodic. For periodic
problems this is the physical periodicity, up to the same
$\mathcal{O}(L/N_p)$. For open-boundary problems it identifies opposite faces
of the cloud's bounding box, an artificial periodicity that the learned
operator has to compensate. This periodicity is that of the transforms only. In
the implementation the positions are normalized to
$\boldsymbol{\xi}=(\mathbf{x}-\mathbf{a}_X)\oslash(\mathbf{b}_X-\mathbf{a}_X)\in[0,1]^d$,
the lifting is applied to $\boldsymbol{\xi}$, and the transforms use
$e^{2\pi\mathrm{i}\mathbf{k}\cdot\boldsymbol{\xi}}$, $\mathbf{k}\in\mathcal{K}_m$.
These basis functions are periodic on the unit box, so no explicit map to the
torus is needed; \eqref{eq:minmax} is the same operator written in the
coordinates $-\pi\mathbf{1}+2\pi\boldsymbol{\xi}$, in which an affine lifting
only has rescaled weights and bias. The lifting, however, acts on the coordinates
themselves and not on a periodic function of them. A linear lifting therefore
distinguishes the two sides of the seam, which are neighbours for the
transforms, and on periodic domains a particle that crosses the boundary
changes its lifted features discontinuously
(Remark~\ref{rem:periodic_lifting}).
The error analysis of Section~\secref{error} is written for a fixed
normalization. For the min--max normalization it applies to configurations
whose bounding boxes do not degenerate, and the $N_p$-dependence of the box is
an additional effect that it does not cover.

\subsection{Grid aliasing and particle sampling}
\seclab{app_alias}

The following proposition compares the Fourier coefficients computed on a
uniform grid, after deposition onto a grid, and from particles, and
quantifies the difference between grid aliasing and particle sampling
discussed in Section~\secref{reno}.

\begin{proposition}[Grid aliasing and particle sampling]
\label{prop:alias}
Let $f:\mathbb{T}^d\to\mathbb{C}$ have an absolutely summable Fourier series,
\begin{equation*}
  f(\mathbf{x})=\sum_{\mathbf{q}\in\mathbb{Z}^d}\hat f_{\mathbf{q}}e^{\mathrm{i}\mathbf{q}\cdot\mathbf{x}},
  \qquad
  \hat f_{\mathbf{q}}=\frac{1}{(2\pi)^d}\int_{\mathbb{T}^d}f(\mathbf{x})\,e^{-\mathrm{i}\mathbf{q}\cdot\mathbf{x}}\,d\mathbf{x},
\end{equation*}
and let $\mathbf{k}\in\mathcal{K}_m$.
\begin{enumerate}
  \item[(i)] \emph{Grid quadrature.} Let $\mu_n=n^{-d}\sum_{\mathbf{j}}\delta_{2\pi\mathbf{j}/n}$,
  $\mathbf{j}\in\{-n/2,\dots,n/2-1\}^d$, be the uniform grid with $n\ge2m$
  points per dimension. Then
  \begin{equation}
    \int f\,e^{-\mathrm{i}\mathbf{k}\cdot\mathbf{x}}\,d\mu_n
    =\sum_{\boldsymbol\ell\in\mathbb{Z}^d}\hat f_{\mathbf{k}+n\boldsymbol\ell}
    =\hat f_{\mathbf{k}}+\sum_{\boldsymbol\ell\ne0}\hat f_{\mathbf{k}+n\boldsymbol\ell}.
    \label{eq:grid_alias}
  \end{equation}
  \item[(ii)] \emph{Deposition.} Let the charge of $\mu_X$ be deposited with a
  shape function $S$ (with absolutely summable Fourier coefficients
  $\hat S_{\mathbf{q}}$), $\rho_n(\mathbf{x}_{\mathbf{j}})=\sum_pw_pS(\mathbf{x}_{\mathbf{j}}-\mathbf{x}_p)$.
  Then the DFT coefficients of the grid density are
  \begin{equation}
    n^{-d}\sum_{\mathbf{j}}\rho_n(\mathbf{x}_{\mathbf{j}})\,e^{-\mathrm{i}\mathbf{k}\cdot\mathbf{x}_{\mathbf{j}}}
    =\sum_{\boldsymbol\ell\in\mathbb{Z}^d}\hat S_{\mathbf{k}+n\boldsymbol\ell}\,\hat\mu_{\mathbf{k}+n\boldsymbol\ell}.
    \label{eq:deposit_alias}
  \end{equation}
  \item[(iii)] \emph{Particle quadrature.} Let $\mu_N$ be the empirical measure
  of $N_p$ iid samples of a probability density $\rho$. Then
  \begin{equation}
    \begin{gathered}
    \mathbb{E}\int f\,e^{-\mathrm{i}\mathbf{k}\cdot\mathbf{x}}\,d\mu_N
    =\int f\,e^{-\mathrm{i}\mathbf{k}\cdot\mathbf{x}}\,\rho\,d\mathbf{x},\\
    \mathbb{E}\Big|\int f\,e^{-\mathrm{i}\mathbf{k}\cdot\mathbf{x}}\,d(\mu_N-\rho)\Big|^2
    \le\frac{\|f\|^2_{L^2(\rho)}}{N_p}.
    \end{gathered}
    \label{eq:particle_quadrature}
  \end{equation}
\end{enumerate}
\end{proposition}

\begin{proof}
(i) Inserting the Fourier series and using
$n^{-d}\sum_{\mathbf{j}}e^{\mathrm{i}2\pi(\mathbf{q}-\mathbf{k})\cdot\mathbf{j}/n}=1$
if $\mathbf{q}-\mathbf{k}\in n\mathbb{Z}^d$ and $0$ otherwise gives~\eqref{eq:grid_alias}.
(ii) Inserting $S(\mathbf{x}_{\mathbf{j}}-\mathbf{x}_p)=\sum_{\mathbf{q}}\hat S_{\mathbf{q}}e^{\mathrm{i}\mathbf{q}\cdot(\mathbf{x}_{\mathbf{j}}-\mathbf{x}_p)}$
and summing over $\mathbf{j}$ as in (i) gives~\eqref{eq:deposit_alias}.
(iii) The integral against $\mu_N$ is the mean of the iid variables
$Z_p=f(\mathbf{x}_p)e^{-\mathrm{i}\mathbf{k}\cdot\mathbf{x}_p}$, whose expectation
is the integral against $\rho$ and whose variance is at most
$\mathbb{E}|Z_p|^2=\|f\|^2_{L^2(\rho)}$.
\end{proof}

In each case the coefficient of a retained mode $\mathbf{k}$ receives
contributions from all frequencies of the integrand. Writing
$\int f e^{-\mathrm{i}\mathbf{k}\cdot\mathbf{x}}\,d\mu=\sum_{\mathbf{q}}\hat f_{\mathbf{q}}\,\hat\mu_{\mathbf{k}-\mathbf{q}}$
shows how. For the grid, $\hat\mu_{n,\mathbf{k}-\mathbf{q}}$ equals one on the
alias lattice $\mathbf{q}\in\mathbf{k}+n\mathbb{Z}^d$ and zero elsewhere, so the
high frequencies are folded onto $\mathbf{k}$ coherently and with full weight.
The error~\eqref{eq:grid_alias} is deterministic, depends on how slowly
$\hat f_{\mathbf{q}}$ decays, and does not decrease with the number of
particles. In a grid-based FNO it occurs in every layer, because $f$ is then
the output of the activation, whose spectrum is not
band-limited~\cite{bartolucci2023representation}; for ReLU-type activations
the kinks give $|\hat f_{\mathbf{q}}|\sim|\mathbf{q}|^{-2}$ in one dimension.
For particles, $\hat\mu_{N,\mathbf{k}-\mathbf{q}}=\hat\rho_{\mathbf{k}-\mathbf{q}}+\xi_{\mathbf{k}-\mathbf{q}}$,
where the fluctuations $\xi$ have zero mean and variance
$(1-|\hat\rho_{\mathbf{k}-\mathbf{q}}|^2)/N_p$ at every wave vector. High
frequencies of $f$ contribute only through these fluctuations, i.e.\ through
the sampling noise of the particles, exactly as they do for the continuous
operator applied to the measure $\mu_N$. The resulting error is unbiased, it
decreases as $N_p^{-1/2}$, and by~\eqref{eq:particle_quadrature} it is bounded
by the $L^2$ norm of $f$ independently of its smoothness or bandwidth. This
is the familiar insensitivity of Monte Carlo integration to the regularity of
the integrand; in the network it applies to each layer, and its propagation
through the layers is Lemma~\ref{lem:sampling}.

Statement (ii) is the aliasing of charge deposition in
PIC~\cite{langdon1970effects,hockney2021computer,birdsall2004plasma}. The
particle noise $\xi_{\mathbf{q}}$ is white, i.e.\ present at every wave vector
with the same variance, and~\eqref{eq:deposit_alias} folds the noise of all
alias bands into each resolved mode with the weights
$\hat S_{\mathbf{k}+n\boldsymbol\ell}$. For the cloud-in-cell shape these
weights relative to $\hat S_{\mathbf{k}}$ are, in one dimension,
$(k/(k+n\ell))^2$, which is small for $|k|\ll n$ but of order one near the
grid Nyquist mode. A grid-based FNO applied to the deposited density thus
receives aliased noise in its input and folds the harmonics generated by each
activation back onto the retained modes. NEOPIC with NUDFTs does neither.

\subsection{Translation equivariance}
\seclab{app_equivariance}

In an open domain, a translation of the whole cloud leaves the
normalized positions~\eqref{eq:minmax}, and hence the field at the particles,
unchanged, for any lifting. Within the normalized box, i.e.\ on the torus, a
lifting that does not depend on $\mathbf{x}$ is required: the
operator~\eqref{eq:err_lift}--\eqref{eq:output} is then translation equivariant,
$\mathbf{E}_\theta[\tau_{\mathbf{a}}\mu](\mathbf{x})=\mathbf{E}_\theta[\mu](\mathbf{x}-\mathbf{a})$
for every shift $\mathbf{a}\in\mathbb{T}^d$, where $\tau_{\mathbf{a}}\mu$ is the
shifted measure: by induction,
$\hat{\mathbf{c}}^{\,\ell}_{\tau_{\mathbf{a}}\mu}(\mathbf{k})=e^{-\mathrm{i}\mathbf{k}\cdot\mathbf{a}}\hat{\mathbf{c}}^{\,\ell}_\mu(\mathbf{k})$
and $\mathbf{u}^\ell_{\tau_{\mathbf{a}}\mu}=\mathbf{u}^\ell_\mu(\cdot-\mathbf{a})$.
By Lemma~\ref{lem:exact_eval} this holds exactly at every particle number,
whereas grid-based layers are at best equivariant under shifts by multiples
of the grid spacing; the link between aliasing and the loss of translation
equivariance is discussed in~\cite[Sec.~5.2]{bartolucci2023representation}.
The linear lifting of the positions used in this work breaks this second symmetry.
It does so in the same way for every $N_p$, as a property of the learned
operator rather than of its discretization.

\section{Additional remarks on the error analysis}
\seclab{app_error}

\begin{remark}[Comparison with PIC]
\label{rem:pic_comparison}
For PIC, the sampling error of the electric field depends on the grid:
in~\cite{tranquilli2022deterministic} it scales as $N_p^{-1/2}$ in 1D,
$(|\log\Delta x|/N_p)^{1/2}$ in 2D and $(N_p\Delta x)^{-1/2}$ in 3D, because
the field retains all modes up to the grid Nyquist number. For NEOPIC, as for
PIF, the field depends on the particles only through a fixed set of modes, and
Lemma~\ref{lem:sampling} gives a rate $N_p^{-1/2}$ in every dimension with a
constant independent of any resolution. The constant $C_S$ grows with the
number of modes through $r^\ell_0$; for the exact Coulomb multipliers the
corresponding growth reproduces the dimension-dependent factors above. In
addition, the PIC error contains a spatial discretization term
$\mathcal{O}(\Delta x^\alpha)$ that requires $\Delta x$ to resolve the Debye
length; in NEOPIC the corresponding contribution is part of
$\varepsilon_{\mathrm{teacher}}$ and does not constrain the particle
resolution at inference.
\end{remark}

\begin{remark}[Other integrators and quiet starts]
\label{rem:integrators}
For an explicit one-step integrator of order $m$ that is stable with respect
to perturbations of the force, the same argument holds with $\Delta t^2$
replaced by $\Delta t^m$; multi-stage methods require one field evaluation
per stage. With a quiet start the initial particles are not iid and
Lemma~\ref{lem:sampling} does not apply; the sampling term is then replaced
by the corresponding quasi-Monte Carlo discrepancy. Implicit one-step
schemes, such as the implicit midpoint rule of \secref{implicit1d},
are covered by the same argument if the nonlinear system of each step is
solved to a tolerance below the local error: for $\Delta t\,|q/m|\,L_\theta$
small enough the implicit step is a contraction, its local error is of the
order of the scheme, and the Gronwall argument is unchanged. The bound then
carries the same step size restriction and does not reflect the stability of
implicit schemes at large time steps, which a worst-case estimate cannot
resolve. Each implicit step requires several field evaluations, one per
nonlinear iteration and per Jacobian--vector product of the linear solver.

The experiments use the leapfrog scheme in staggered form,
$\mathbf{x}^{n+1}=\mathbf{x}^n+\Delta t\,\mathbf{v}^{n+1/2}$,
$\mathbf{v}^{n+3/2}=\mathbf{v}^{n+1/2}+\Delta t\frac{q}{m}\mathbf{E}^{n+1}_\theta(\mathbf{x}^{n+1})$,
started with
$\mathbf{v}^{1/2}=\mathbf{v}^0+\frac{\Delta t}{2}\frac{q}{m}\mathbf{E}^0_\theta(\mathbf{x}^0)$.
Combining the last half kick of a step of~\eqref{eq:kdk} with the first half
kick of the next shows that this scheme produces the same positions and
half-step velocities as~\eqref{eq:kdk}, so the analysis applies unchanged. For
the cyclotron mini-app of Section~\secref{multi2d}, the Boris scheme is used in
the same staggered form, but the initial velocities are taken as
$\mathbf{v}^{-1/2}$. The initial velocity then has an error
$\mathcal{O}(\Delta t)$, and the time-stepping error with respect to the exact
solution becomes $\mathcal{O}(\Delta t)$. This error cancels in comparisons
with a reference run that uses the same integrator and initialization
(Proposition~\ref{prop:inloop_teacher}), as in our experiments.
\end{remark}

\begin{remark}[Periodicity of the lifting]
\label{rem:periodic_lifting}
Assumption (A3) requires $\mathcal{P}$ to be Lipschitz with respect to the
periodic distance. A linear lifting of the raw coordinates
$\mathbf{x}\in[-\pi,\pi)^d$ is discontinuous across the periodic boundary, so
the field predicted at a particle can jump when the particle crosses the
boundary, and Lemma~\ref{lem:lipschitz} holds only away from it. Using a
periodic embedding of the positions, e.g.\
$(\sin\mathbf{x},\cos\mathbf{x})$, as input to $\mathcal{P}$ satisfies (A3)
exactly. This choice is appropriate for periodic problems only. With a
periodic lifting, every operation of the network preserves periodicity, so
the predicted field is periodic on the normalized box for any weights; the
pointwise path $\mathcal{W}_\ell\mathbf{u}+\mathbf{b}_\ell$ can no longer
represent non-periodic behaviour, which in the FNO~\cite{li2020fourier} enters
through non-periodic input features. For open boundaries the field is not
periodic on the bounding box of the cloud: for a cloud with an outward field,
for example, $E_x$ has opposite signs on opposite faces. There, either
non-periodic information has to enter through the lifting, as with the linear
lifting used in this work, or the normalized box has to be padded so that the
periodic seam lies outside the particle cloud, as free-space spectral Poisson
solvers enlarge the computational domain~\cite{hockney2021computer,vico2016fast}.
These choices do not affect Theorems~\ref{thm:reno} and~\ref{thm:eps_reno},
which hold for any pointwise lifting.
\end{remark}

\begin{remark}[Other boundary conditions]
\label{rem:field_bc}
In this work we consider periodic and free-space boundary conditions for the fields and the particles.
With Dirichlet or Neumann conditions on the walls of a bounded domain $\Omega$,
the field is still linear in the charge,
$\mathcal{E}_\Omega[\mu](\mathbf{x})=\int_\Omega\mathbf{K}_\Omega(\mathbf{x},\mathbf{y})\,d\mu(\mathbf{y})+\mathbf{E}_{\mathrm{bc}}(\mathbf{x},t)$,
where $\mathbf{E}_{\mathrm{bc}}$ is the vacuum field of the imposed wall
potentials and does not depend on the particles, and
Lemma~\ref{lem:exact_eval} and Theorem~\ref{thm:reno} are unaffected. By
superposition, $\mathbf{E}_{\mathrm{bc}}$ can be computed once by a Laplace
solve and added to the output, like the external fields of the Penning trap in
Section~\secref{multi3d} (see also~\cite{lyu2026latent}), so that the network
learns only the self-consistent part with homogeneous boundary conditions;
since walls exert a net force on the plasma, $c_{\mathrm{per}}=0$
in~\eqref{eq:output}. The kernel $\mathbf{K}_\Omega$ is no longer
translation invariant. NEOPIC can represent this through the position-dependent
lifting, as for the cyclotron mini-app with open boundaries, or, on rectangular
domains, exactly by replacing the complex exponentials of the nonuniform
transforms by sine or cosine series, i.e.\ by odd or even reflections of the
particles across the walls, as in spectral Poisson
solvers~\cite{hockney2021computer}. Particle boundary conditions such as absorption or emission change $N_p$ in time, which the operator on measures accommodates; only the scaling $s$ changes with the total
charge. Charge collected on floating conductors could enter as an additional measure supported on
the boundary.
\end{remark}

\begin{remark}[$L^1$ training loss]
\label{rem:l1_loss}
The models of Section~\secref{results} are trained with the mean absolute error
$J_1(\theta;\mu)=\frac1d\int_{\mathbb{T}^d}|r_\theta[\mu]|_1\,d\mu$, where
$|\cdot|_1$ is the sum over components, instead of~\eqref{eq:train_loss}. The
loss is computed on standardized labels (Section~\secref{results}), which
weights each component by a fixed factor; we omit these weights, which change
only the constants below. Since $J_1$ has the units of the field and $J$ those of its square, orders in
$N$ below refer to the units of the field; in these units the noise term $Q/N$
of~\eqref{eq:noise_expansion} is $(Q/N)^{1/2}$. By the
Cauchy--Schwarz inequality $J_1\le d^{-1/2}J^{1/2}$, and
$J\le d\,\|r_\theta[\mu]\|_{L^\infty}J_1$, so the two losses are small together.
The expansion~\eqref{eq:noise_expansion}, however, requires the loss to be
smooth in the measure (Assumption (A6) for $J$), which fails for $J_1$
because the absolute value is not differentiable where a component of the
residual vanishes. Instead let $r_i$ be a component of $r_\theta[\rho](\mathbf{x})$ and
$s_i^2=\mathbb{E}\,|(Dr_\theta[\rho]\boldsymbol\xi)_i(\mathbf{x})|^2$, so that
$r_{\theta,i}[\mu_N](\mathbf{x})$ fluctuates about $r_i$ with standard deviation
$s_i/\sqrt N$ to leading order. Where $|r_i|\gg s_i/\sqrt N$, the absolute value
is linear and the noise enters only at order $1/N$, below the noise level $N^{-1/2}$,
through terms of the
form of $P$: the bias
$\frac1{2N}\mathbb{E}\,D^2\mathcal{G}_\theta[\rho](\boldsymbol\xi,\boldsymbol\xi)$, a
field of order $1/N$, and the correlation between the residual and the evaluation points, which
are the particles themselves. Where $|r_i|\ll s_i/\sqrt N$, the fluctuation is
approximately Gaussian and contributes $\sqrt{2/\pi}\,s_i/\sqrt N$. For a model
whose residual on noise-free inputs lies below the noise level, the noise term
of $\mathbb{E}\,J_1$ is therefore of order $N^{-1/2}$ and bounded by
$\sqrt{2/(\pi d)}\,(Q/N)^{1/2}$. As for the squared loss, it penalizes the
mismatch $Dr_\theta$ rather than $D\mathcal{G}_\theta$ and vanishes for a
realizable teacher. In this regime, the proofs of Proposition~\ref{prop:tradeoff}(i), (ii) and
of the worst-case statement in (iii) use only that the noise term is $R$ times
an increasing function of $\beta$, and hold with $\beta^{1/2}$ in place of
$\beta$; the minimizer of the averaged loss in (iii) becomes the particle number
$N$ with $N^{-1/2}=\sum_i\pi_iN_i^{-1/2}$. Where both regimes occur, the noise
term mixes contributions of order $N^{-1/2}$ and $N^{-1}$, and these statements
hold only approximately. Finally, since~\eqref{eq:lip} couples
the particles in the mean of $|\cdot|$, the proof of
Proposition~\ref{prop:inloop_teacher} carries over to that norm and gives
\eqref{eq:inloop_teacher} with
$\big(\mathbb{E}\,J\big)^{1/2}$ replaced by $d\,\mathbb{E}\,J_1$.
\end{remark}

\begin{remark}[Scope of the analysis]
\label{rem:noise_caveats}

The analysis is of leading order in $1/N$, concerns the population loss,
assumes that training reaches a minimizer, and uses the squared error (for
the $L^1$ loss see Remark~\ref{rem:l1_loss}). (a) Unlike in data
randomization, where a new noise sample is drawn in every
epoch~\cite{nguyen2025dgnet}, each snapshot carries a single, frozen
realization, and consecutive snapshots of a run have strongly correlated
noise. A random split of the snapshots of one run into training and
validation sets therefore does not give independent validation data, the
empirical loss~\eqref{eq:train_loss} approximates $\bar{\mathcal{L}}_N$ only
through the number of independent realizations, and a network can fit the
frozen noise; teacher runs with several seeds reduce this effect. (b) In a
teacher run the particles become correlated at $t>0$, with pair correlations
of order $1/N$. This changes the covariance of the fluctuations, in plasmas to
a shielded spectrum, and adds a term of order $1/N$ that is linear in the
residual, but the structure of~\eqref{eq:noise_expansion} is unchanged.
(c) For grid-based neural operators trained on deterministic data there is
no fluctuation $\boldsymbol\xi$, and the whole correction of order $1/N$
vanishes; the regularization is specific to learning on particle data.
(d) Training with several particle numbers requires data at each of them,
from separate teacher runs or by subsampling a run with more particles.
Subsampling is not straightforward: the labels must be recomputed on the
subsample, since keeping those of the full set turns the task into
denoising, and at $t>0$ a subsample inherits the correlations of the
$N$-particle system, so that its fluctuations differ from those of a run with
fewer particles. We leave mixed-resolution training to future work.
\end{remark}

\section{Diagnostics of the cyclotron mini-app}
\seclab{app_cyclotron}

\paragraph{Discrete canonical angular momentum} In two dimensions, with
$\mathbf{B}=B_0\mathbf{e}_z$ and $\omega_c=qB_0/m$, the staggered Boris scheme
is exactly
\begin{equation}
  \mathbf{v}^{n+1/2}-\mathbf{v}^{n-1/2}
  =\Delta t\,\frac{q}{m}\mathbf{E}^n+\Delta t\,\frac{q}{m}\,\bar{\mathbf{v}}^n\times\mathbf{B},
  \qquad
  \mathbf{x}^{n+1}=\mathbf{x}^n+\Delta t\,\mathbf{v}^{n+1/2},
  \label{eq:boris_staggered}
\end{equation}
with $\bar{\mathbf{v}}^n$ as in~\eqref{eq:Lcan}; we omit the particle index.
In two dimensions
$(\mathbf{x}\times(\bar{\mathbf{v}}\times\mathbf{B}))_z=-B_0\,\mathbf{x}\cdot\bar{\mathbf{v}}$
and $\mathbf{x}\cdot(\bar{\mathbf{v}}\times\mathbf{B})=B_0(\mathbf{x}\times\bar{\mathbf{v}})_z$.
Taking the cross product of the velocity update with $m\,\mathbf{x}^n$ and
using $\bar{\mathbf{v}}^n=(\mathbf{x}^{n+1}-\mathbf{x}^{n-1})/(2\Delta t)$ and
$\mathbf{x}^n\times\mathbf{v}^{n-1/2}=\mathbf{x}^{n-1}\times\mathbf{v}^{n-1/2}$, which
holds because $\mathbf{x}^n-\mathbf{x}^{n-1}$ is parallel to $\mathbf{v}^{n-1/2}$,
gives after summation over the particles
\begin{equation}
  \begin{gathered}
  J^{n+1/2}-J^{n-1/2}=\Delta t\sum_pq\,(\mathbf{x}^n_p\times\mathbf{E}^n_p)_z,\\
  J^{n+1/2}=\sum_p\Big[m(\mathbf{x}^n_p\times\mathbf{v}^{n+1/2}_p)_z+\frac{qB_0}{2}\,\mathbf{x}^n_p\cdot\mathbf{x}^{n+1}_p\Big].
  \end{gathered}
  \label{eq:boris_invariant}
\end{equation}
The scheme therefore conserves $J$ exactly whenever the total torque
$\sum_pq\,(\mathbf{x}_p\times\mathbf{E}_p)_z$ vanishes. Writing $J^{n+1/2}$ and
$J^{n-1/2}$ in terms of $\mathbf{x}^n$, with
$\mathbf{x}^{n\pm1}=\mathbf{x}^n\pm\Delta t\,\mathbf{v}^{n\pm1/2}$, and inserting
the velocity update for $\mathbf{v}^{n+1/2}-\mathbf{v}^{n-1/2}$ gives
\begin{equation}
  \tfrac12\big(J^{n+1/2}+J^{n-1/2}\big)=L^n_{\mathrm{can}}+\frac{\omega_c\Delta t^2}{4}\,W^n,
  \qquad
  W^n=\sum_pq\,\mathbf{x}^n_p\cdot\mathbf{E}^n_p,
  \label{eq:Lcan_J}
\end{equation}
where the factor $1+(\omega_c\Delta t/2)^2$ in~\eqref{eq:Lcan} comes from the
magnetic part of the kick between the two half steps. For pairwise Coulomb
forces in two dimensions the virial $W$ is constant,
$W=\sum_{p<q}q_pq_q/(2\pi)$, because the logarithmic potential is scale
invariant, and $L_{\mathrm{can}}$ is then conserved exactly. For other fields it
changes through torques, i.e.\ non-central field errors, and through changes
of $W$, e.g.\ by errors in the strength of a central field; the latter carry
the factor $\Delta t^2$ and are much smaller in this setting.

\paragraph{Energy diagnostic} A Boris step rotates the velocity by
$\theta=2\arctan(\omega_c\Delta t/2)$ instead of $\omega_c\Delta t$, so that the
gyration is under-resolved at $\omega_c\Delta t=3$ (period $2\pi/\theta\approx3.2$
steps instead of $2\pi/(\omega_c\Delta t)\approx2.1$). The kinetic energy is
computed from $\bar{\mathbf{v}}$, the average of two velocities rotated by
$\theta$, which measures the gyration energy only with the factor
$\cos^2(\theta/2)=1/(1+(\omega_c\Delta t/2)^2)\approx0.31$; the drift part is
measured correctly. Energy exchanged between the field and the gyration
therefore appears only partly in the total energy, which explains the energy
oscillation of the Barnes--Hut reference in
Fig.~\figref{multiapp2d_conservation}(e). This is a property of the integrator
and the diagnostic, not of the field method, and it disappears when the
gyration is resolved.

\paragraph{Winding and arm shape} The particles are binned on a fixed grid,
and the positions are shifted to the density-weighted centroid of the cloud.
For thin rings of width $\Delta r=0.005$ about the centre we compute the
azimuthal Fourier coefficient
\begin{equation*}
  A_2(r)=\sum_{\text{bins in the ring}}n\,e^{-2\mathrm{i}\Theta},
\end{equation*}
where $n$ is the number of particles in a bin and $\Theta$ its polar angle.
A density with two arms opposite each other is invariant under a rotation by
$\pi$, so that its odd azimuthal modes vanish and $m=2$ is the lowest mode that
carries the orientation of the pattern; the centring removes the $m=1$ mode of a
shifted cloud. The angle $\varphi(r)=-\arg A_2(r)/2$ is the orientation of the
arms at radius $r$. It is defined modulo $\pi$, and $2\varphi$ is unwrapped
along $r$ so that $\varphi$ varies continuously between neighbouring rings.
The arm shape is $|\varphi(r)-\varphi(r_{\mathrm{in}})|$, and the winding
between $r_{\mathrm{in}}=0.09$ and $r_{\mathrm{out}}=0.29$ is
$|\varphi(r_{\mathrm{out}})-\varphi(r_{\mathrm{in}})|/(2\pi)$ turns. It vanishes
for the initial elliptical cloud and grows as the sheared
$\mathbf{E}\times\mathbf{B}$ rotation, with angular velocity $E_r/(B_0r)$, winds
the cloud into spiral arms. The ratio of the windings of NEOPIC and
Barnes--Hut is $1.14$ for $0.09\le r\le0.29$, $1.16$ for $0.10\le r\le0.25$,
$1.08$ for $0.12\le r\le0.22$ and $0.995$ for $0.14\le r\le0.29$ at $t=11.98$,
which confirms that the extra winding is made in the shear layer just outside
the core.

\section{Cyclotron mini-app: convergence in the time step}
\seclab{dtconv}

Theorem~\ref{thm:trajectory} splits the error of NEOPIC into a time-stepping
term of order $\Delta t^2$ and terms set by the model and the particle number.
We test this split on the cyclotron mini-app of Section~\secref{multi2d}, which
has the fastest time scale of the 2D mini-apps, $\omega_c=|q/m|B_0=300$. The
2D multi-app model, trained with $\Delta t=0.01$ ($\omega_c\Delta t=3$), is
run with $N_p=10^6$ particles up to $T=3$, within the training horizon, for
$\Delta t=0.02\cdot2^{-k}$, $k=0,\dots,6$, i.e.\
$0.09\le\omega_c\Delta t\le6$. As reference we use PIF with the free-space
spectral solver~\cite{vico2016fast} at the same time steps; all runs start from
the same particles, and particle data are stored every $0.1$ time units.
Unlike in Section~\secref{multi2d}, the Boris scheme is started with a half
step, $\mathbf{v}^{1/2}$ from $\mathbf{v}^0$ by a Boris step of length
$\Delta t/2$, so that it is of second order (Remark~\ref{rem:integrators}).

\paragraph{Gyration phase and guiding centres} The positions of individual
particles are not a suitable measure here. A Boris step rotates the velocity
by $2\arctan(\omega_c\Delta t/2)$ instead of $\omega_c\Delta t$, so after
a time $t$ the gyration phase of a particle is off by
$\delta\varphi\approx\omega_c^3\Delta t^2\,t/12$, which is of second order
but has a large constant. The resulting position error, about $r_L\,\delta\varphi$ with
the Larmor radius $r_L=|\mathbf{v}_\perp|/\omega_c$, cannot exceed $2r_L$ and
saturates once $\delta\varphi=\mathcal{O}(1)$. At $t=3$ this is the case for
$\Delta t\gtrsim4\times10^{-4}$, so the position, velocity and field errors at
the particles level off at the gyroradius scale
($\langle r_L\rangle=\sqrt{\pi/2}/\omega_c\approx4\times10^{-3}$) for all time
steps considered. At early times they follow the estimate
$\langle r_L\rangle\,2|\sin(\delta\varphi/2)|$ to within a few per cent (for
example $1.42\times10^{-3}$ measured against $1.37\times10^{-3}$ at $t=0.1$
and $\Delta t=1.25\times10^{-3}$). Theorem~\ref{thm:trajectory} remains
valid, but with a magnetic field its constant grows with $\omega_c$, and the
asymptotic regime $\Delta t\ll(12/\omega_c^3T)^{1/2}$ lies below the time
steps used. We therefore compare the guiding centres
\begin{equation}
  \mathbf{x}_{\mathrm{gc}}=\mathbf{x}+\frac{m}{qB_0}\,\bar{\mathbf{v}}\times\mathbf{e}_z,
  \label{eq:guiding_centre}
\end{equation}
with the time-centred velocity $\bar{\mathbf{v}}$ of~\eqref{eq:Lcan}. They do
not depend on the gyration phase and move on the slow time scale of the
$\mathbf{E}\times\mathbf{B}$ drift, which determines the density and the
field. Since
$|\mathbf{x}_{\mathrm{gc}}-\mathbf{y}_{\mathrm{gc}}|\le|\mathbf{x}-\mathbf{y}|+|\mathbf{v}-\mathbf{w}|/\omega_c$,
they are covered by Theorem~\ref{thm:trajectory}; their second-order
convergence seen below, while the positions and velocities level off at the
gyroradius scale, is a property of the dynamics, which the theorem does not
resolve.

\paragraph{Results} Figure~\figref{dt_convergence}(a) shows the mean distance
of the guiding centres at $t=3$. NEOPIC and PIF converge at the same rate,
with local orders between $1.7$ and $2.0$ for $\Delta t\le0.005$ and a fitted
order of $1.75$ for $\Delta t\le0.0025$, and their curves lie on top of each
other: the time-stepping error is that of the integrator and is not changed by
the learned field, as assumed in the splitting of
Theorem~\ref{thm:trajectory}. Most of this error is made in the first steps;
for $\Delta t\le2.5\times10^{-3}$ between half and all of the error at $t=3$
is already present at $t=0.1$. The difference between NEOPIC and PIF at the
same $\Delta t$, the model error, is independent of the time step for
$\Delta t\le0.01$ and amounts to $3.5\times10^{-3}$ at $t=3$, about one
gyroradius. The error with respect to the finest PIF run therefore levels off
at the model error for $\Delta t\lesssim0.005$ ($\omega_c\Delta t\lesssim1.5$),
whereas at the training time step the time-stepping error ($9.9\times10^{-3}$)
is about three times the model error. The model error does not grow when the
time step is reduced by a factor of $32$ below that of the training data: the
model sees only particle configurations, and it can be used with other time
steps than those of its training data.

The relative changes of the total energy and of $L_{\mathrm{can}}$
(Fig.~\figref{dt_convergence}(b)) decrease with $\Delta t$ for PIF, to
below $10^{-3}$ and $10^{-4}$, respectively. For NEOPIC they are
$3.2$--$3.4\times10^{-2}$ and $4.3$--$4.5\times10^{-2}$ for
$\Delta t\le0.0025$, independent of the time step, i.e.\ set by the model, consistent with the torque from the
non-central field error discussed in Section~\secref{multi2d}.

\begin{figure}[!htbp]
  \centering
  \includegraphics[width=\linewidth]{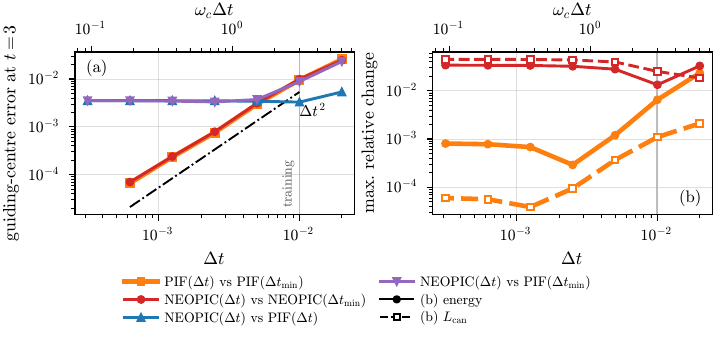}
  \caption{Cyclotron mini-app, 2D multi-app model, $N_p=10^6$, $T=3$.
  (a) Mean distance of the guiding centres~\eqref{eq:guiding_centre} at
  $t=3$ between runs with time step $\Delta t$ and the finest run
  ($\Delta t_{\min}=3.125\times10^{-4}$), for PIF and NEOPIC, and between
  NEOPIC and PIF at the same and at the finest time step; the dash-dotted line
  has slope $2$. (b) Maximum relative change of the total energy (solid) and
  of the canonical angular momentum $L_{\mathrm{can}}$~\eqref{eq:Lcan}
  (dashed) for PIF and NEOPIC. The grey line marks the time step of the
  training data.}
  \figlab{dt_convergence}
\end{figure}

\section{Implicit time integration}
\seclab{implicit1d}

NEOPIC replaces only the field solve, so any particle integrator can be used
with it, including implicit ones that require the Jacobian of the field with
respect to the particle positions. Because $\mathcal{G}_\theta$ is a
differentiable function of the positions, this Jacobian is available exactly by automatic differentiation. This is in contrast to the implicit particle-in-cell methods~\cite{chen2011energy,markidis2011energy}, where the nonlinear system is typically solved by Jacobian-free Newton-Krylov methods with finite-difference approximations of the Jacobian-vector products, and the field depends on the positions only piecewise smoothly through the shape functions. We consider in one dimension the
implicit midpoint (Crank--Nicolson) scheme,
\begin{equation}
  \begin{aligned}
  \mathbf{x}^{n+1}&=\mathbf{x}^n+\Delta t\,\frac{\mathbf{v}^n+\mathbf{v}^{n+1}}{2},\\
  \mathbf{v}^{n+1}&=\mathbf{v}^n+\Delta t\,\frac{q}{m}\,
  \mathbf{E}_\theta\big[\mu^{n+\frac12}_N\big]\big(\mathbf{x}^{n+\frac12}\big),
  \end{aligned}
  \label{eq:implicit_midpoint}
\end{equation}
with $\mathbf{x}^{n+\frac12}=(\mathbf{x}^n+\mathbf{x}^{n+1})/2$ and $\mu^{n+\frac12}_N$
the empirical measure of these midpoint positions. The scheme
is second order, symplectic and free of numerical damping, and is
unconditionally stable for the linear plasma oscillation.

\paragraph{Model and setup} A one-dimensional model with the architecture of
Section~\secref{results} is trained on four mini-apps generated by PIF with
$32$ modes and the cloud-in-cell shape function: weak and strong Landau
damping and the two-stream and bump-on-tail instabilities, each with
$N_{\mathrm{tr}}=10^5$ particles and $600$ time steps of $\Delta t=0.05$
($T_{\mathrm{train}}=30$). The model is trained with batch size $20$ and
weight decay $10^{-5}$ for $250$ epochs on one A100 GPU ($1.2$\,h training time);
all other training settings are as in Section~\secref{results}. 
We then run weak Landau damping ($\alpha=0.05$, $k=0.5$) with $N_p=10^5$ particles up to
$T=60$, twice the training horizon, for $\Delta t\in\{0.05,0.25,0.5,1,1.5,2,2.5\}$,
with NEOPIC advanced by the leapfrog scheme and by~\eqref{eq:implicit_midpoint},
and compare with leapfrog PIF at the same $\Delta t$. All runs start from the
same particles.

\paragraph{Nonlinear solve} The system~\eqref{eq:implicit_midpoint} for
$(\mathbf{x}^{n+1},\mathbf{v}^{n+1})$ is solved by a matrix-free
Newton--Krylov method. Jacobian--vector products are computed by forward-mode
automatic differentiation through the network, so that the Jacobian is exact
for the discrete residual. The linear systems are solved by restarted GMRES
(restart length $10$) with the adaptive relative tolerance
\begin{equation*}
  \eta_k=\max\big(10^{-6},\min(10^{-2},\|F_k\|/\|F_0\|)\big)
\end{equation*}
of inexact Newton methods~\cite{eisenstat1996choosing}, where $F_k$ is the
residual at Newton iteration $k$. As preconditioner we use the cold-plasma response of a
nearly uniform plasma, $\partial a_p/\partial x_q\approx-\omega_p^2(\delta_{pq}-1/N_p)$,
for which the Jacobian decouples into $2\times2$ blocks that are inverted exactly;
$\omega_p^2$ is estimated in every step from the response of the network to a
smooth displacement and comes out between $0.994$ and $0.995$. Newton stops when
the residual has decreased by $10^{-4}$ or when it stagnates at the round-off
level of the single-precision network. Two details matter on a periodic domain.
The Newton iterate must not be mapped back into the box during the iteration,
since this makes the residual discontinuous, and the displacement
$\mathbf{x}^{n+1}-\mathbf{x}^n$ must not be replaced by its minimum image, which
is wrong for particles with $|v|\Delta t>L/2$. Positions are therefore iterated
unwrapped and wrapped only after convergence.

\paragraph{Results} Figure~\figref{implicit1d} shows the field energy and the
maximum relative change of the total energy. At $\Delta t=0.05$ the three runs
coincide and follow the linear damping rate. With the leapfrog scheme, NEOPIC
reproduces the behaviour of PIF at every $\Delta t$: the energy errors agree to
within $6\%$, and both lose stability for $\Delta t\geq2$, beyond the
limit $\omega\Delta t<2$ of the leapfrog scheme for the Langmuir mode
($\omega\approx1.42$), where the field energy grows to $W_x\approx7$, more than three orders of
magnitude above the level of the stable runs, and the total energy changes by
a factor of six to seven. With the implicit
midpoint scheme NEOPIC remains stable at all $\Delta t$ over the whole run,
including the extrapolation to twice the training horizon; the maximum
relative energy change is at most $1.2\times10^{-3}$ up to $\Delta t=2$ and
$1.1\times10^{-2}$ at $\Delta t=2.5$, smaller than that of the leapfrog runs at
every $\Delta t$. Newton converges in $2$ iterations per step at $\Delta t=0.05$
and in $3$ to $4$ for $\Delta t\geq1$. For $\omega\Delta t\gtrsim2$ no scheme
resolves the plasma oscillation or the resonant particles, which then move
more than half a wavelength per step, so that the implicit runs are stable and
nearly energy-conserving but do not reproduce the damping rate
(Fig.~\figref{implicit1d}(c)).

The implicit step is not competitive in cost here: at $\Delta t=1$ it takes
$0.64$\,s, compared with $2.8$\,ms for a leapfrog step of NEOPIC and $3.2$\,ms
for PIF, and requires on average $4.2$ residual evaluations and $43$
Jacobian--vector products. The experiment is meant to show that NEOPIC can be
combined with implicit integrators through exact Jacobians; preconditioners
for two and three dimensions, where only the longitudinal part of a
displacement field responds, and an efficient implementation are left to
future work. We plan to consider the settings in which implicit schemes pay off
most: magnetized plasmas with $\omega_c\Delta t\gg1$, where they can
reproduce the drift dynamics at time steps far larger than explicit schemes
allow~\cite{ricketson2023pseudospectral}, and implicit Darwin
models~\cite{chen2014energy,chen2015multi}, in which NEOPIC can serve either
as the field solver or as a preconditioner of the nonlinear solve.

\begin{figure}[!htbp]
  \centering
  \includegraphics[width=\linewidth]{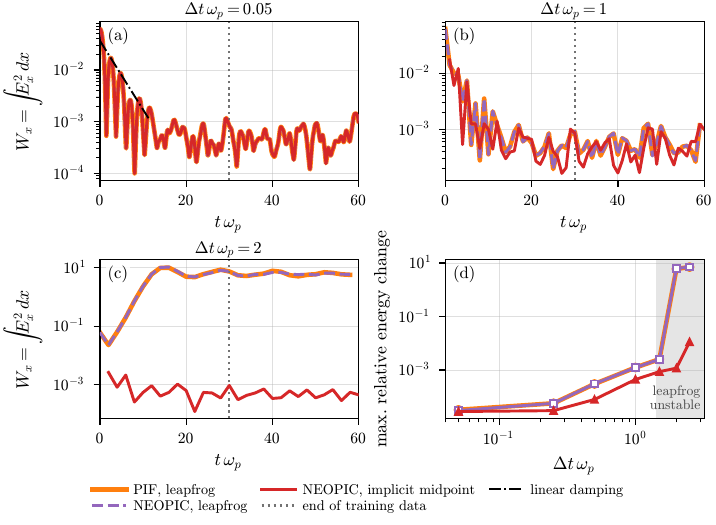}
  \caption{1D-1V weak Landau damping with leapfrog PIF, leapfrog NEOPIC and
  implicit-midpoint NEOPIC ($N_p=10^5$, $T=60$; the training data end at
  $t=30$, dotted). (a)--(c)~Field energy $W_x$ for $\Delta t=0.05$, $1$ and
  $2$; in (a) the dash-dotted line is the linear damping rate. (d)~Maximum
  relative change of the total energy over the run against $\Delta t$; the
  leapfrog curves of PIF and NEOPIC lie on top of each other. The shaded region
  marks $\omega\Delta t>2$, where the leapfrog scheme is unstable for the
  Langmuir mode.}
  \figlab{implicit1d}
\end{figure}

\section{Cost and architecture}
\seclab{app_cost}

The cost of NEOPIC is set by its
architecture. To measure this dependence independently of accuracy, we time
sub-networks of the trained 3D five-app model ($L=4$, $d_v=32$, $m=16$) with
$L\in\{2,3,4\}$, $d_v\in\{8,16,32\}$ and $m\in\{8,16\}$. They are obtained by
keeping the first $L$ layers, the first $d_v$ channels and the lowest modes of
the trained weights; the values of the weights do not affect the cost. Every
sub-network is run with the NUFFT for $20$ steps of 3D weak Landau damping at
$N_p=10^6$ and $4\times10^6$, which gives $\alpha$ and $\beta$ for each
architecture. For the full model this reproduces the timings of
Fig.~\figref{complexity} ($353$ and $1245$\,ms against $348$ and $1244$\,ms).
All $18$ architectures follow
\begin{equation*}
  \alpha\approx2.28\,L\,d_v+4.4\ \text{ns}
\end{equation*}
to within $3\%$ (Fig.~\figref{cost_models}(a)): the cost per particle is
proportional to the number $2Ld_v$ of transforms per particle, as
Proposition~\ref{prop:complexity} predicts when the spreading and
interpolation, $2d_vw^d$ per layer with $w^3=343$, dominate the pointwise
term $d_v^2$. No contribution of $d_v^2$ is visible for $d_v\le32$. With
$2w^3=686$ multiply--add operations per channel and layer, $2.28$\,ns
corresponds to about $3$\,ps per operation, which reflects the memory-bound
spreading and interpolation. The number of modes hardly matters: halving $m$
divides the number of spectral weights by eight but changes the time by at
most $4\%$, since the fine-grid FFTs are cheap compared with spreading and
interpolation. The constant $\beta$ ($18$--$55$\,ms) grows with $L$ and $d_v$
but not with $m$, consistent with the plan creation discussed above rather
than with the FFT term of $\beta_{\mathrm{NUFFT}}$. The time therefore follows
$Ld_v$ and not the number of parameters (Fig.~\figref{cost_models}(b)): the
parameters vary by a factor of $256$, the time by a factor of $6.8$. The
smallest sub-network ($L=2$, $d_v=8$) takes $183$\,ms per step at
$N_p=4\times10^6$, comparable with Barnes--Hut ($172$\,ms) and with PIF at
$\varepsilon=10^{-12}$ ($158$\,ms, interpolated between the measurements at
$1.05\times10^6$ and $8.4\times10^6$ particles), and about three times the
cost of the Python PIC and of PIF at $\varepsilon=10^{-6}$. Since the
sub-networks are not trained, this compares cost only; a trained model of
this size may be less accurate than the full model.

\begin{figure}[!htbp]
  \centering
  \includegraphics[width=\linewidth]{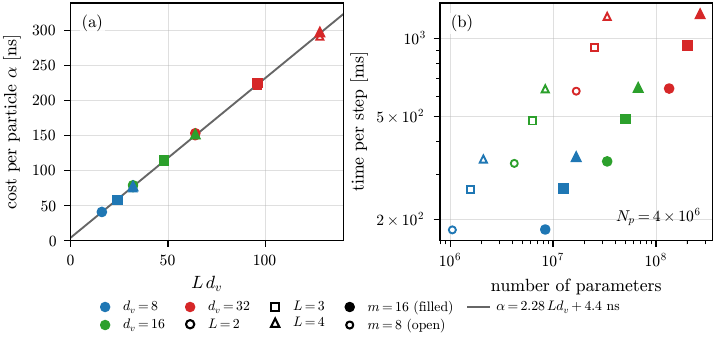}
  \caption{Cost of NEOPIC (NUFFT, 3D) for sub-networks of the trained model
  with $L$ layers, $d_v$ channels and $2m$ modes per direction. (a) Cost per
  particle $\alpha$, from timings at $N_p=10^6$ and $4\times10^6$, against
  $Ld_v$; the line is the least-squares fit. (b) Time per step at
  $N_p=4\times10^6$ against the number of parameters. Colours: $d_v$; markers:
  $L$; filled: $m=16$, open: $m=8$.}
  \figlab{cost_models}
\end{figure}

\section*{Acknowledgments}
The project \textit{A Neural Operator Framework for Particle-based Kinetic Plasma Simulations} - NEOPIC (INF funding code: ZT-I-PF-5-242) underlying this publication
is funded by the Initiative and Networking Fund of the
Helmholtz Association in the framework of the Helmholtz AI project
call. We thank computing resources on JURECA DC from VSR computing project PEPCEXA and on JUWELS Booster from HAICORE@FZJ. The first author would like to thank Prof. Siddhartha Mishra for fruitful discussions on this topic. He also would like to thank Prof. Tan Bui-Thanh and Dr. Hai Nguyen for introducing their work on noise as a regularizer in training, as well as obtaining Jacobians for implicit methods through automatic differentiation.

\section*{Availability}
The NEOPIC implementation, together with the training and evaluation
configurations of all experiments in this paper, is available at
\href{https://github.com/srikrrish/operator_learning}{\texttt{github.com/\allowbreak srikrrish/\allowbreak operator\_learning}}. The version used for
the results in this paper is tagged \texttt{neopic-arxiv-v1}; the version used
for the computational cost measurements in Section~\secref{cost} is tagged
\texttt{neopic-arxiv-v1-nufft}. The scripts for the analysis and the figures are
available, with the same tag \texttt{neopic-arxiv-v1}, at
\href{https://github.com/NEOPIC-framework/neopic-paper-scripts}{\texttt{github.com/\allowbreak NEOPIC-\allowbreak framework/\allowbreak neopic-\allowbreak paper-\allowbreak scripts}}.
The training data sets, trained models and
reference simulation outputs (HDF5 files) are not included in the repository
because of their size; they are available from the corresponding author upon
reasonable request.

\section*{Declaration of generative AI and AI-assisted technologies in \\the manuscript preparation process}

During the preparation of this work the authors used Claude (Anthropic) in order to draft and revise text, structure the manuscript, develop and check mathematical derivations and proofs, and write scripts for data analysis and figure generation, and, in the initial phase of the project, ChatGPT (OpenAI) for brainstorming and for writing parts of the code. All experiments were designed and run by the authors, and the AI-assisted code and analysis scripts were checked by the authors. After using these tools, the authors reviewed and edited the content as needed and take full responsibility for the content of the published article.

\section*{References}

\bibliographystyle{elsarticle-num}

\bibliography{references}

\end{document}